\pdfoutput=1
\RequirePackage{ifpdf}
\ifpdf 
\documentclass[pdftex]{sigma}
\else
\documentclass{sigma}
\fi

\usepackage{tikz-cd}
\usepackage[export]{adjustbox}
\newcommand{\Z}{\mathbb{Z}}
\newdimen\R
\usepackage{lipsum}

\makeatletter
\newtheorem*{rep@theorem}{\rep@title}
\newcommand{\newreptheorem}[2]{%
\newenvironment{rep#1}[1]{%
 \def\rep@title{#2 \ref{##1}}%
 \begin{rep@theorem}}%
 {\end{rep@theorem}}}
\makeatother

\makeatletter
\newtheorem*{rep@prop}{\rep@title}
\newcommand{\newrepprop}[2]{%
\newenvironment{rep#1}[1]{%
 \def\rep@title{#2 \ref{##1}}%
 \begin{rep@prop}}%
 {\end{rep@prop}}}
\makeatother

\newcommand{\Real}{\mathbb{R}}

\newcommand{\Orb}{\mathcal{O}}

\newcommand{\ur}{{\rm ur}}

\numberwithin{equation}{section}

\newtheorem{Theorem}{Theorem}[section]
\newtheorem{Corollary}[Theorem]{Corollary}
\newtheorem{Lemma}[Theorem]{Lemma}
\newtheorem{Proposition}[Theorem]{Proposition}
\newtheorem{Conjecture}[Theorem]{Conjecture}
 { \theoremstyle{definition}
\newtheorem{Definition}[Theorem]{Definition}
\newtheorem{Example}[Theorem]{Example}
\newtheorem{Remark}[Theorem]{Remark} }

\begin{document}
\allowdisplaybreaks

\newcommand{\arXivNumber}{2003.13872}

\renewcommand{\thefootnote}{}

\renewcommand{\PaperNumber}{138}

\FirstPageHeading

\ShortArticleName{Snake Graphs from Triangulated Orbifolds}

\ArticleName{Snake Graphs from Triangulated Orbifolds\footnote{This paper is a~contribution to the Special Issue on Cluster Algebras. The full collection is available at \href{https://www.emis.de/journals/SIGMA/cluster-algebras.html}{https://www.emis.de/journals/SIGMA/cluster-algebras.html}}}

\Author{Esther BANAIAN and Elizabeth KELLEY}

\AuthorNameForHeading{E.~Banaian and E.~Kelley}

\Address{School of Mathematics, University of Minnesota, Minneapolis, MN 55455, USA}
\Email{\href{mailto:banai003@umn.edu}{banai003@umn.edu}, \href{mailto:kell1642@umn.edu}{kell1642@umn.edu}}

\ArticleDates{Received March 31, 2020, in final form December 08, 2020; Published online December 17, 2020}

\Abstract{We give an explicit combinatorial formula for the Laurent expansion of any arc or closed curve on an unpunctured triangulated orbifold. We do this by extending the snake graph construction of Musiker, Schiffler, and Williams to unpunctured orbifolds. In the case of an ordinary arc, this gives a combinatorial proof of positivity to the generalized cluster algebra from this orbifold.}

\Keywords{generalized cluster algebra; cluster algebra; orbifold; snake graph}

\Classification{05E15; 05C70; 16S99}

{\small \tableofcontents}

\renewcommand{\thefootnote}{\arabic{footnote}}
\setcounter{footnote}{0}

\section{Introduction}\label{sec:Introduction}

We extend the snake graph construction of Musiker, Schiffler, and Williams \cite{MSW} and subsequent work of Musiker and Williams \cite{MW} from ordinary cluster algebras from surfaces to generalized cluster algebras from unpunctured orbifolds in order to give a combinatorial proof of positivity for this subclass of generalized cluster algebras. Our work is motivated by, but does not use, unpublished work of Gleitz and Musiker \cite{Gleitz-Musiker}. This paper is a continuation of our previous work~\cite{Banaian-Kelley} and contains both proofs of those results and further results.

Sections \ref{sec:ordinaryClusterAlgebras} and \ref{sec:generalizedClusterAlgebras} give the relevant background information. In Section~\ref{sec:ordinaryClusterAlgebras}, we review the ordinary cluster algebra case. This includes basic definitions and an explanation of the snake graph construction for ordinary cluster algebras from surfaces. In Section~\ref{sec:generalizedClusterAlgebras}, we discuss the analogous definitions for generalized cluster algebras, orbifolds, and laminations on orbifolds.

The discussion of our construction begins in Section~\ref{Sec:Construction}. In this section, we explain how snake graphs can be constructed from triangulations of unpunctured orbifolds and state the following cluster expansion formula:

\begin{Theorem}\label{Thm:GenSnake}
Let $\mathcal{O} = (S,M,Q)$ be an unpunctured orbifold with triangulation $T$ and $\mathcal{A}$ be the corresponding generalized cluster algebra with principal coefficients with respect to $\Sigma_T = (\mathbf{x}_T,\mathbf{y}_T,B_T)$. For an ordinary arc $\gamma$ with generalized snake graph $G_{T,\gamma}$, the Laurent expansion of $x_\gamma$ with respect to $\Sigma_T$ is
\[ [x_\gamma]^{\mathcal{A}}_{\Sigma_T} = \frac{1}{\operatorname{cross}(T,\gamma)} \sum_P x(P)y(P), \]
where the summation is indexed by perfect matchings of $G_{T,\gamma}$.
\end{Theorem}

The proof of this cluster expansion formula follows the same structure as the proof given by Musiker, Schiffler, and Williams~\cite{MSW} for their cluster expansion formula for ordinary cluster algebras from surfaces. Given an arc without self-intersections, which we refer to as an \emph{ordinary arc}, on a triangulated orbifold we lift to a cover $\widetilde{S}_\gamma$ which is a triangulated polygon without internal triangles. A cluster expansion formula for this type of triangulated polygon was already known from the work of Musiker and Schiffler~\cite{Musiker-Schiffler} prior to~\cite{MSW}. We denote the ordinary cluster algebra associated to $\widetilde{S}_\gamma$ as $\widetilde{\mathcal{A}}_\gamma$ and construct a homomorphism $\phi_\gamma$ from $\widetilde{\mathcal{A}}_\gamma$ to the generalized cluster algebra $\mathcal{A}$ associated with the original triangulated orbifold. We then verify our cluster expansion formula by applying $\phi_\gamma$ to the expansion obtained in the lift. Sections~\ref{sec:lift} through~\ref{sec:orbifoldLaminations} contain the material required to prove Theorem~\ref{Thm:GenSnake}. Section~\ref{sec:lift} discusses the lift $\widetilde{S}_\gamma$, Section~\ref{sec:QuadBigon} establishes quadrilateral and bigon lemmas for orbifolds, and then Sections~\ref{Sec:mapSection} through~\ref{sec:orbifoldLaminations} discuss the algebra homomorphism $\phi_\gamma$. The material from these sections is then pulled together in Section~\ref{sec:ClusterExpansionPf} to give a proof of Theorem~\ref{Thm:GenSnake}.

In Section \ref{sec:universalSnakeGraph}, we define a new object called a \emph{universal snake graph}, $UG_n$, which can be used to recover both ordinary snake graphs and our generalized snake graphs and allows us to simplify the arguments and calculations of Musiker and Williams \cite{MW}. In Lemma~\ref{PMUniversalSnake}, we describe the poset of perfect matchings of $UG_n$, and highlight some of its interesting properties, including that the poset is isomorphic to $B_n$. In Theorem \ref{thm:upperright}, we show how to obtain the weighted sum of perfect matchings of $UG_n$ from a matrix product.

We then turn to generalized arcs and closed curves. Similarly to Musiker--Williams \cite{MW}, we associate cluster algebra elements to these arcs and curves via the following definitions.

\begin{Definition}\label{def:genArc}
Let $\mathcal{O} = (S,M,Q)$ be an unpunctured orbifold with triangulation~$T$ and~$\mathcal{A}$ be the corresponding generalized cluster algebra with principal coefficients with respect to $\Sigma_T = (\mathbf{x}_T,\mathbf{y}_T,B_T)$. Let $\gamma$ be a generalized arc with generalized snake graph $G_{T,\gamma}$.
\begin{itemize}\itemsep=0pt
 \item If $\gamma$ has a contractible kink, then $X_{\gamma,T} = -X_{\bar{\gamma},T}$ where $\bar{\gamma}$ is $\gamma$ with this kink removed.
 \item Otherwise, we define \[X_{\gamma,T} = \frac{1}{\operatorname{cross}(T,\gamma)} \sum_P x(P)y(P). \]
\end{itemize}
\end{Definition}

\begin{Definition}\label{def:closedcurve}
Let $\mathcal{O} = (S,M,Q)$ be an unpunctured orbifold with triangulation~$T$ and~$\mathcal{A}$ be the corresponding generalized cluster algebra with principal coefficients with respect to $\Sigma_T = (\mathbf{x}_T,\mathbf{y}_T,B_T)$. Let $\gamma$ be a closed curve with generalized band graph $G_{T,\gamma}$.
\begin{itemize}\itemsep=0pt
 \item If $\gamma$ is a contractible loop, $X_{\gamma,T} = -2$.
 \item If $\gamma$ is isotopic to a curve which bounds a disk containing a unique orbifold point, then $X_{\gamma,T} = 2\cos(\pi/p):= \lambda_p$ where $p$ is the order of the orbifold ponit in this disk.
 \item Otherwise, we define \[X_{\gamma,T} = \frac{1}{\operatorname{cross}(T,\gamma)} \sum_P x(P)y(P), \]
 where the sum is over \emph{good matchings} of $G_{T,\gamma}$.
\end{itemize}
\end{Definition}

\looseness=-1 To verify that these are sensible definitions, in Section \ref{sec:$M$-path} we describe another way to build a~cluster algebra element from an arbitrary arc or curve, building on the work of \cite{MW}. We break the arc (or curve) into smaller \emph{elementary steps}, each of which has a corresponding $2 \times 2$ matrix, and then consider the appropriately ordered product of these matrices. In Theorem~\ref{Thm:ArcsAndGraphs} of Section~\ref{sec:StandardMPath}, we verify that the matrix formulation is consistent with the expansion formula from the snake graph. The advantage of the matrix viewpoint is that we are able to prove Proposition~\ref{Prop:OrbSkein} in Section~\ref{sec:Skein}, which verifies that these cluster algebra elements satisfy the relations given by applying the skein relation to intersections and self-intersections. We also highlight a three-term skein relation, Proposition~\ref{prop:OrbSkeinThreeTerm}, from an intersection of pending arcs which resembles the generalized exchange polynomials in the generalized cluster algebras we consider.

Finally, in Section \ref{sec:RelateToPuncture} we explore the relationship between punctures and orbifold points and show that some of the results of~\cite{MSW} and~\cite{MW} for punctures can be recovered by treating a~puncture as an orbifold point of order infinity.

\section{Ordinary cluster algebras}\label{sec:ordinaryClusterAlgebras}

Cluster algebras were originally introduced in 2000 by Fomin and Zelevinsky~\cite{Fomin-Zelevinsky-I} to provide a concrete combinatorial framework for studying dual canonical bases and total positivity in semisimple groups. Cluster algebras are commutative rings with a distinguished family of ge\-nerators, called \emph{cluster variables}, which occur in distinguished collections of $n$-element subsets $\{x_1, \dots, x_n \}$ called \emph{clusters}. Clusters can be obtained from each other by an involutive process called \emph{mutation}, which replaces a single cluster variable with a unique cluster variable which was not present in the original cluster. Through repeated mutation, one can recover the complete set of cluster variables.

Algebraically, mutation is given by a binomial \emph{exchange relations}. The coefficients in these binomial exchange relations come from a choice of \emph{coefficient variables}. Each cluster $\{ x_1, \dots, x_n \}$ has an associated set of coefficients $\{y_1, \dots, y_n \}$, which are also related to each other via mutation. For more technical definitions, one good source is \cite[Section~2]{Fomin-Zelevinsky-IV}.

Cluster algebras have a variety of notable properties -- these famously include the \emph{Laurent phenomenon}.

\begin{Theorem}[{\cite[Theorem 3.1]{Fomin-Zelevinsky-I}}]
Let $\mathcal{A}$ be an arbitrary cluster algebra. The cluster variables $x_1, \dots, x_n$ can be expressed in terms of any cluster of~$\mathcal{A}$ as a Laurent polynomial with coefficients in~$\mathbb{ZP}$.
\end{Theorem}

The Laurent phenomenon becomes even more compelling with the addition of the \emph{positivity property}.

\begin{Conjecture}[{cf.\ \cite[Section 3]{Fomin-Zelevinsky-I}}]
The coefficients of these Laurent polynomials are in fact \emph{non-negative}.
\end{Conjecture}

Positivity was conjectured in Fomin and Zelevinsky's original paper \cite{Fomin-Zelevinsky-I} and later verified in a variety of cases, including: skew-symmetric cluster algebras \cite{Lee-Schiffler}, cluster algebras of surface type \cite{MSW, Schiffler, Schiffler-Thomas}, acyclic (quantum) cluster algebras \cite{Caldero-Reineke, Efimov,Kimura-Qin, Qin}, bipartite cluster algebras \cite{Nakajima}, and cluster algebras of geometric type~\cite{GHKK}. Note that this last case encompasses all of the prior cases and is the most general setting in which a proof of the positivity conjecture is known.

Throughout, we will use the term \emph{ordinary cluster algebra} to refer to this original definition due to Fomin and Zelevinsky.

\subsection{Cluster algebras from surfaces}

In 2008, Fomin, Shapiro, and D.~Thurston showed that a subset of ordinary cluster algebras can be modeled by triangulations of bordered surfaces with marked points~\cite{FST-triangulatedSurfaces}. Marked points that appear on the interior of the surface are called \emph{punctures}. We work primarily with unpunctured surfaces, but will note some connections to punctures in Section~\ref{sec:RelateToPuncture}. For unpunctured surfaces, we quickly highlight some relevant portions of Fomin, Shapiro, and D.~Thurston's construction and refer the reader to Section~2 of their paper for a much more detailed explanation.

\begin{Definition}An \emph{arc} $\gamma$ on a surface $(S,M)$ is a non-self-intersecting curve in $S$ with endpoints in $M$ that is otherwise disjoint from $M$ and $\partial S$. Curves that are contractible onto $\partial S$ or that cut out an unpunctured monogon or bigon are not considered arcs. Arcs are considered up to isotopy class.
\end{Definition}

Because arcs are considered up to isotopy, we have to be somewhat careful when thinking about the idea of intersections of arcs. Let $\gamma$ and $\gamma'$ be arbitrary arcs on $S$ and $\alpha$ and $\alpha'$ denote arbitrary representatives of their isotopy classes. Then the \emph{crossing number} $e(\gamma,\gamma')$ is the minimal number of crossings of each possible choice of $\alpha$ and $\alpha'$. Two arcs $\gamma$ and $\gamma'$ are considered \emph{compatible} if $e(\gamma,\gamma') = 0.$ Similarly, if $T = \{\tau_1,\ldots,\tau_n\}$ is a triangulation of a surface, we define $e(\gamma,T) = \sum_{i=1}^n e(\gamma, \tau_i)$.

\begin{Definition}
An \emph{ideal triangulation} $T$ of a surface is a maximal collection of pairwise compatible arcs (and boundary arcs).
\end{Definition}

Surfaces with ideal triangulations provide a useful combinatorial tool for studying certain cluster algebras, via the correspondence described in the following theorem.

\begin{Theorem}[Fomin--Thurston \cite{FT-II}]\label{SurfaceCA}
Given a surface with marked points, $(S,M)$, there exists a unique cluster algebra $\mathcal{A} = \mathcal{A}(S,M)$ with the following properties: $(1)$~the seeds are in bijection with tagged triangulations of $(S,M)$;
$(2)$~the cluster variables are in bijection with tagged arcs in $(S,M)$; and $(3)$~the cluster variable $x_\gamma$ corresponding to arc $\gamma$ is given by the lambda length of $\gamma$, in terms of some initial triangulation.
\end{Theorem}
 In this dictionary, mutation of cluster variables in $\mathcal{A}$ is equivalent to flipping arcs in the triangulation~$T$. Each arc $\tau$ in $T$ looks locally like the diagonal in a quadrilateral. Thus, the result of flipping $\tau$ in $T$ is a new triangulation, $T' = (T - \{\tau\}) \cup \{\tau' \}$ where $\tau'$ is the unique other diagonal of the quadrilateral surrounding $\gamma$.

\subsection{Snake graphs}\label{section:snakeGraphs}

Snake graphs were defined by Musiker, Schiffler, and Williams in order to give explicit combinatorial formulas for the cluster variables in any cluster algebra from a surface \cite{MSW}. In doing so, they offered an alternate and combinatorial proof of positivity for cluster algebras from surfaces. For any fixed arc $\gamma$ and triangulation $T$ of a surface $(S,M)$, they construct a graph $G_{T,\gamma}$ by gluing together \emph{tiles} that encode the local geometry at each intersection between $\gamma$ and arcs of the triangulation. The formula for the expansion of $x_\gamma$ with respect to the cluster corresponding to $T$ is then given in terms of perfect matchings of $G_{T,\gamma}$. We briefly review their construction for unpunctured surfaces, but refer the interested reader to Section 4 of their paper for the complete construction and many examples.

Let $(S,M)$ be a bordered surface with triangulation $T$ and $\gamma$ be an ordinary arc (i.e., one without self-intersections) in $S$ but not in $T$. Fix an orientation on $\gamma$ and let $s$ and $t$ denote the starting and ending points of $\gamma$, respectively. Denote the intersection points of $\gamma$ and $T$ as $s = p_0, p_1, \dots, p_{d+1} = t$, in order, and let $\tau_{i_j}$ be the arc in $T$ containing $p_j$. Denote the ideal triangles on either side of $\tau_{i_j}$ as $\Delta_{j-1}$ and $\Delta_j$.

Each intersection $p_j$ is associated with a square tile $G_j$ formed by gluing copies of~$\Delta_{j-1}$ and~$\Delta_j$ along the edge labeled $\tau_{i_j}$. This can be done either so that both triangles have orientation matching~$\Delta_{j-1}$ and~$\Delta_j$ or both triangles have opposite orientation, hence there are two valid planar embeddings of each~$G_j$. The tile $G_j$ is said to have \emph{relative orientation} $\operatorname{rel}(G_j) = +1$ if the orientation of its triangles match the ideal triangles on $S$ and $\operatorname{rel}(G_j) = -1$ otherwise. Two of the edges of the triangle $\Delta_j$ are labeled $\tau_{i_j}$ and $\tau_{i_{j+1}}$; label the remaining edge as $\tau_{[\gamma_j]}$.

{\sloppy The graph $G_{T,\gamma}$ is then formed by gluing together subsequent tiles $G_{1}, \dots, G_{d}$ in order. Ti\-les~$G_j$ and~$G_{j+1}$ are glued along the edge labeled $\tau_{[\gamma_j]}$ after choosing planar embeddings~$\widetilde{G}_j$ and~$\widetilde{G}_{j+1}$ of the tiles such that $\operatorname{rel}\big(\widetilde{G}_j\big) \neq \operatorname{rel}\big(\widetilde{G}_{j+1}\big)$. Gluing together all~$d$ tiles gives a~graph~$\bar{G}_{T,\gamma}$. The graph $G_{T,\gamma}$ can then be obtained from $\bar{G}_{T,\gamma}$ by removing the diagonal edge from each tile.

}

Before stating Musiker, Schiffler, and Williams' expansion formula, we need to review the following definitions and notation.

\begin{Definition}[{\cite[Definition 4.4]{MSW}}]
For an ordinary arc $\gamma$ crossing the sequence of arcs $\tau_{i_1},\dots,\tau_{i_d}$ in $T$, the \emph{crossing monomial} of $\gamma$ with respect to $T$ is defined as
\begin{gather*}
\operatorname{cross}(T,\gamma) = \prod_{j=1}^d x_{\tau_{i_j}}.
\end{gather*}
\end{Definition}

\begin{Definition}[{\cite[Definition 4.5]{MSW}}]
A perfect matching $P$ of a snake graph $G$ which uses edges labeled $\tau_{i_1}, \dots, \tau_{i_k}$ has \emph{weight} $x(P) = x_{\tau_{i_1}}\cdots x_{\tau_{i_k}}$.
\end{Definition}

\begin{Definition}[{\cite[Definition 4.6]{MSW}}]\looseness=-1
$G_{T,\gamma}$ has exactly two perfect matchings that include only boundary edges; these are referred to as the \emph{minimal} and \emph{maximal} matchings of $G_{T,\gamma}$. The distinction between the two depends on the relative orientation of $G_{T,\gamma}$; if $\operatorname{rel}(G_{T,\gamma}) = 1$ (respectively, $-1$), define $e_1$ and~$e_2$ to be the edges that are immediately counterclockwise (respectively, clockwise) from the diagonal. The \emph{minimal matching}, $P_{-}$, is defined to be the unique perfect matching that includes only boundary edges and does not include $e_1$ or $e_2$. The \emph{maximal mat\-ching} $P_{+}$ is the complementary perfect matching on boundary edges that includes~$e_1$ and~$e_2$.
\end{Definition}

The cluster expansion formula from snake graphs involves the symmetric difference of an arbitrary perfect matching $P$ with $P_{-}$, denoted $P \ominus P_{-}$. The edges of $P \ominus P_{-}$ are always a~subgraph of $G_{T,\gamma}$ composed of potentially disjoint cycles, which encloses a finite set of tiles~$\{ G_{i_j} \}_{j \in J}$.

\begin{Definition}[{\cite[Definition 4.8]{MSW}}]
Let $T = \{ \tau_1, \dots, \tau_n \}$ be an ideal triangulation of an unpunctured surface $(S,M)$ and $\gamma$ be an ordinary arc on $(S,M)$. Let $P$ be a perfect matching of $G_{T,\gamma}$ such that $P \ominus P_{-}$ encloses the set of tiles $\{ G_{i_j} \}_{j \in J}$. The \emph{height monomial} of $P$ is
\begin{align*}
y(P) = \prod_{k=1}^n h_{\tau_{k}}^{m_k},
\end{align*}
where $m_k$ is the number of tiles in $\{G_{i_j} \}_{j \in J}$ with diagonal labeled $\tau_{i_j}$ and $h_{\tau_k} = y_{\tau_k}$.
\end{Definition}

Now, we are prepared to state the formula for obtaining Laurent expansions from snake graphs.

\begin{Theorem}[{\cite[Theorem 4.9]{MSW}}]\label{theorem:clusterExpansion_ordinary}
Let $(S,M)$ be a bordered surface with triangulation~$T$, $\mathcal{A}$~be the corresponding cluster algebra with principal coefficients, and~$\gamma$ be an ordinary arc on~$S$. Then $x_\gamma$ can be written as a Laurent expansion in terms of the initial cluster variables as
\[ x_\gamma = \frac{1}{\operatorname{cross}(T,\gamma)}\sum_P x(P)y(P), \]
where the sum ranges across all perfect matchings $P$ of the snake graph $G_{T,\gamma}$.
\end{Theorem}

Subsequently, Musiker and Williams extended the snake graph construction to handle \emph{generalized arcs}, which may contain self-intersections, and closed curves \cite{MW}. Suppose $\gamma$ is now a generalized arc. If $\gamma$ is a contractible loop, then Musiker and Williams define $x_\gamma = 0$. If $\gamma$ contains a contractible kink and $\overline{\gamma}$ denotes the corresponding arc with the kink removed, then they define $x_\gamma = (-1)x_{\overline{\gamma}}$. Musiker and Williams then show that Theorem \ref{theorem:clusterExpansion_ordinary} holds for generalized arcs. For closed curves, they define a corresponding cluster algebra element using a slight modification of snake graphs called \emph{band graphs}. For details, see \cite[Section~3]{MW}.

The set of perfect matchings of a snake graph has a natural poset structure. Describing this structure makes use of \emph{twists}, i.e. local moves on a perfect matching $P$ where the horizontal edges of a single tile are replaced with the vertical edges of that tile, or vice versa. Musiker, Schiffler, and Williams~\cite{MSW-bases} establish the following result about that poset structure, building on previous work by Propp on the poset structure of perfect matchings of bipartite planar graphs~\cite{Propp}.

\begin{Theorem}[{\cite[Theorem 5.2]{MSW-bases}}]\label{Thm:PMPoset}
Consider the set of all perfect matchings of a snake graph $G$ and construct a graph whose vertices are labeled by these perfect matchings and which has an edge between two vertices if and only if the matchings corresponding to those vertices are obtainable from each other by a single twist. An edge corresponding to twisting a tile with diagonal edge~$\tau_j$ is labeled~$y_j$. This graph is the Hasse diagram of a distributive lattice, with minimal element~$P_{-}$, which is graded by the degree of the height monomials associated with each matching.
\end{Theorem}

\subsection{Laminations}

In the cluster algebra context, laminations were used by Fomin--D.~Thurston \cite{FT-II} as a tool for tracking the coefficients of a cluster algebra from a surface using W. Thurston's \cite{Thurston2} shear coordinates and theory of measured laminations. We will review only the relevant portion of their work (for unpunctured surfaces), but refer the reader either to Chapter~12 of their work for further details about laminations in this context, or to the work of Fock--Goncharov~\cite{Fock-Goncharov} or W.~Thurston \cite{Thurston2} for more details about measured laminations and their relationship to matrix mutations.

\begin{Definition}[{\cite[Definition~12.1]{FT-II}}]
Let $(S,M)$ be an unpunctured bordered surface. An \emph{integral unbounded measured lamination} (henceforth referred to as just a \emph{lamination}) on $S$ is a~finite collection of non-self-intersecting and pairwise non-intersecting curves on $S$ such that:
\begin{itemize}\itemsep=0pt
 \item each curve is either a closed curve or a non-closed curve with endpoints on umarked points on~$\partial S$,
 \item no curve bounds an unpunctured disk,
 \item and no curve with endpoints on~$\partial S$ is isotopic to a portion of the boundary containing either no or one marked point(s).
\end{itemize}
\end{Definition}

W.~Thurston's shear coordinates \cite{Thurston2} provide a coordinate system for these laminations.

\begin{Definition}[{\cite[Definition 12.2]{FT-II}}]
Let $S$ be a surface with triangulation $T$ and $L$ be a~lamination on $S$. For each arc $\gamma \in T$, the \emph{shear coordinate} of~$L$ with respect to~$T$ is
\[ b_\gamma(T,L) = \sum_i b_\gamma(T,L_i), \]
where the summation runs over all individual curves in $L$. The shear coordinates $b_\gamma(T,L_i)$ are defined as:
\begin{center}
 \includegraphics[scale=1.0]{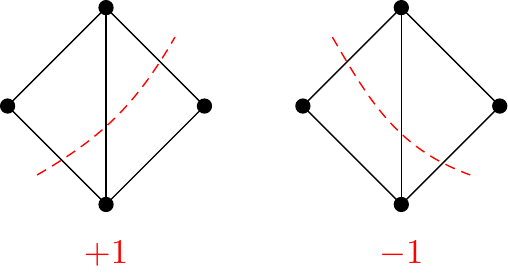}
\end{center}
\end{Definition}

Tracking principal coefficients requires the notion of an \emph{elementary lamination}.

\begin{Definition}
The \emph{elementary lamination} $L_i$ associated to an arc $\tau_i$ in triangulation $T$ is the lamination such that $b_{\tau_i}(T,L_i) = 1$ for $\tau_i \in T$ and $b_\tau(T,L_i)=0$ for $\tau \not\in T$.
\end{Definition}

 An ordinary cluster algebra of surface type with principal coefficients corresponds to a triangulated surface with a lamination composed of all possible elementary laminations.

\section{Generalized cluster algebras}\label{sec:generalizedClusterAlgebras}

\subsection{Background}

In 2014, Chekhov and Shapiro defined an extension of ordinary cluster algebras \cite{Fomin-Zelevinsky-I} without the restriction that exchange polynomials be strictly binomial \cite{Chekhov-Shapiro}. Although other generalizations exist, these algebras are often referred to as \emph{generalized cluster algebras} and we will adopt this nomenclature and will use the term \emph{ordinary cluster algebra} to refer to the original definition due to Fomin and Zelevinsky. Generalized cluster algebras were introduced in order to extend existing work on the Teichm\"{u}ller spaces of Riemann surfaces with holes and orbifold points of order two and three to the more general case of Riemann surfaces with holes and orbifold points of arbitrary order \cite{Chekhov,Chekhov-Mazzocco-Yangians}.

Many of the basic definitions for generalized cluster algebras follow the same structure as the corresponding definitions for ordinary cluster algebras, with differences arising as a result of the modified exchange relations. We briefly highlight the changes that will be relevant in this paper and refer the reader to~\cite{Chekhov-Shapiro} or~\cite{Nakanishi} for more precise descriptions.

Because the exchange relations may now have arbitrarily many terms, describing a generalized cluster algebra requires specifying an additional piece of data: a collection of exchange polynomials $(z_1,\dots,z_n)$ where $z_i$ specifies the exchange relation for a given cluster variable $x_i$. We do not describe the general rule for mutation of exchange polynomials (see~\cite{Chekhov-Shapiro} or~\cite{Nakanishi}) because in subclass of generalized cluster algebra we discuss, the exchange polynomials are fixed under mutation. When an exchange polynomial $z_k$ has degree one, mutation in direction~$k$ reduces to the definition of ordinary mutation. If all the exchange polynomials of a generalized cluster algebra have degree one, it reduces to an ordinary cluster algebra. Hence, ordinary cluster algebras can be understood as a subclass of generalized cluster algebras.

In their original paper, Chekhov and Shapiro prove that generalized cluster algebras exhibit the Laurent phenomenon.

\begin{Theorem}[{\cite[Theorem 2.5]{Chekhov-Shapiro}}]
\label{CS-LP}
Let $\mathcal{A} = (\mathbf{x}, \mathbf{y}, B, \mathbf{Z})$ be an arbitrary generalized cluster algebra over $\mathbb{P}$. The cluster variables $x_1, \dots, x_n$ can be expressed in terms of any cluster of $\mathcal{A}$ as a Laurent polynomial with coefficients in $\mathbb{ZP}$.
\end{Theorem}

Further, they prove that a particular subclass of generalized cluster algebras exhibit positivity. It is important to note, however, that although this proof only applies to a subclass of generalized cluster algebras, positivity is widely expected to hold for all generalized cluster algebras.

\begin{Theorem}[{cf.\ \cite[Section 5]{Chekhov-Shapiro}}]\label{CS-reciprocalDeg2}\looseness=-1
Let~$\mathcal{A}$ be any generalized cluster algebra whose exchange polynomials are all reciprocal and of degree at most two. Then its cluster variables can be expressed in terms of any cluster of $\mathcal{A}$ as Laurent polynomials with non-negative coefficients in~$\mathbb{ZP}$.
\end{Theorem}

Although generalized cluster algebras were defined relatively recently, they have already been the subject of quite a bit of study. This includes work by Nakanishi~\cite{Nakanishi} on a subclass of generalized cluster algebras, for which he gives formulas expressing the cluster variables and coefficients in terms of $c$-vectors, $g$-vectors, and $F$-polynomials; work by Nakanishi and Rupel~\cite{Nakanishi-Rupel} in which they define the notion of a \emph{companion algebra} to a generalized cluster algebra; and work by Paquette and Schiffler~\cite{Paquette-Schiffler} on a related generalization with additional allowed types of exchange relations.

\subsection{Orbifolds}\label{sec:orbifolds}

An orbifold is a generalization of a manifold where the local structure is given by quotients of open subsets of $\mathbb{R}^n$ under finite group actions. Orbifolds arose independently in many mathematical contexts, ranging from the theories of modular or automorphic forms \cite{Satake} to $3$-manifold theory~\cite{Thurston}, and so there are many possible phrasings for a definition. In the context of cluster algebras from orbifolds, we can simply think of orbifolds as surfaces with isolated singular points, referred to as \emph{orbifold points}.
In parallel with the classification of cluster algebras associated with triangulated surfaces~\cite{FST-triangulatedSurfaces, FT-II}, Felikson, Shapiro, and Tumarkin established both a notion of triangulating orbifolds and a classification of cluster algebras from orbifolds~\cite{FST-orbTriang}. In this section, we briefly review some nomenclature and definitions for triangulations of orbifolds that will be used in later sections. For more details and many examples, we refer the reader to \cite[Section~4]{FST-orbTriang}.

\begin{Definition}
An \emph{orbifold} $\mathcal{O}$ is a triple $(S,M,Q)$, where~$S$ is a~bordered surface, $M$~is a~finite set of marked points, and $Q$ is a finite set of orbifold points, such that: no point is both a marked point and an orbifold point (i.e., $M \cap Q = \varnothing$); all orbifold points are interior points of $S$; and each boundary component of $S$ contains at least one marked point. For notational convenience, $\partial \mathcal{O}$ is often used to refer to $\partial S$.
\end{Definition}

\begin{Remark}
Unlike in \cite{Felikson-Shapiro-Tumarkin} where all orbifold points are weight~2 or~$\frac12$, our orbifold points are associated with positive integer orders, $p \geq 2$. Note that our orbifold points of order~2 have weight~$\frac12$ in the language of~\cite{Felikson-Shapiro-Tumarkin}.
\end{Remark}

An orbifold point of order $p$ has an associated constant $\lambda_p = 2\cos(\pi/p)$. In our context, $\lambda_p$ arises geometrically from the length of diagonals in an equilateral $p$-gon which appears in a~particular covering space called the \emph{$p$-fold cover} \cite{Chekhov-Shapiro, L-FV}. In the literature, it also appears in the work of Holm and J{\o}rgensen on non-integral frieze patterns from polygon dissections~\cite{HJ}.

\begin{Definition}
An \emph{arc} $\gamma$ on an orbifold $\mathcal{O} = (S,M,Q)$ is a non-self-intersecting curve in $S$ with endpoints in $M$ that is otherwise disjoint from $M$, $Q$, and $\partial \mathcal{O}$. Curves that are contractible onto $\partial \mathcal{O}$ are not considered arcs. Arcs are considered up to isotopy class. An arc which cuts out an unpunctured monogon with exactly one point in $Q$ is called a \emph{pending arc}, while all other arcs are called \emph{standard arcs}.
\end{Definition}

The two possible ways to draw pending arcs are shown below. We will prefer to draw pending arcs as cutting out unpunctured monogons which contain exactly one orbifold point, shown on the right hand side, as this is more geometrically suggestive. In particular, this makes it clear that if a pending arc crosses another arc, it necessarily does so an even number of times. We will sometimes refer to a pending arc as being \emph{incident} to the orbifold point it encloses, in the spirit of the left hand side.

\begin{center}
 \includegraphics[scale=1.5]{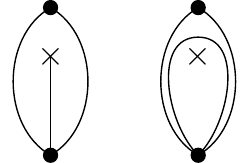}
\end{center}

\begin{Definition}Two arcs are considered \emph{compatible} if their isotopy classes contain non-inter\-sec\-ting representatives. A \emph{triangulation} is a maximal collection of pairwise compatible arcs.
\end{Definition}

Note that in a triangulation, every pending arc is necessarily enclosed by a bigon with one orbifold point, or a monogon with two orbifold points. See \cite[Fig.~4.1]{Felikson-Shapiro-Tumarkin}. While in many pictures it appears that the pending arc is in a bigon, we can identify the two vertices to recover a monogon with two orbifold points.
 There is also the special case of a sphere with one marked point and three orbifold points, which has exactly one triangle made up of three pending arcs, as in \cite[Fig.~3.5]{Felikson-Shapiro-Tumarkin}. Our construction works in this special case as well.

We will also consider generalized arcs. Allowing self-intersections introduces the possibility of winding around orbifold points. By convention, we will consider counterclockwise winding to be positive and clockwise winding to be negative. A generalized arc exhibits modular behavior when winding around an orbifold point. For an orbifold point of order $p$, a winding arc can have up to $p-1$ self-intersections. Once the number of self-intersections reaches $p$, the arc is isotopic to an arc with no self-intersections -- i.e., one that does not wind around the orbifold point at all. If a winding arc has $k > p$ self-intersections, then it is isotopic to an arc with $k \mod{p}$ self-intersections and~$\big\lfloor \frac{k}{p}\big\rfloor$ contractible kinks. The below diagram shows examples of possible winding behavior around an orbifold point of order~$4$.

\begin{center}
 \includegraphics[scale=1.8]{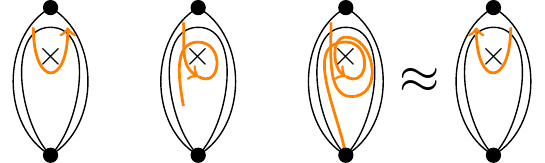}
\end{center}

For an orbifold point of order $p$, winding counter-clockwise with $k$ self-intersections is isotopic to winding clockwise with $(p-1)-k$ self-intersections. For convenience, we will use the phrasing ``winding $k$ times" to refer to winding with $k$ self-intersections. So, for example, ``winding 0 times" simply refers to crossing a pending arc twice with no self-intersections occurring between those crossings.

It is also possible to have closed curves with no self-intersections which we refer to as \emph{loops}. A non-contractible loop is often called an \emph{essential loop}.

Triangulated orbifolds can provide geometric realizations for some ordinary cluster algebras which cannot be realized as triangulated surfaces. This realization is due to Felikson, Shapiro, and Tumarkin, who describe a correspondence between skew-symmetrizable ordinary cluster algebras and triangulated orbifolds \cite{Felikson-Shapiro-Tumarkin}. In a later paper, Felikson and Tumarkin generalize the bracelet, bangle, and band bases to ordinary cluster algebras from unpunctured orbifolds with at least two marked points on the boundary \cite{Felikson-Tumarkin}. \c{C}anak\c{c}i and Tumarkin later showed that assumption about the number of marked points on the boundary can be removed and extended the snake graph and band graph constructions to triangulated orbifolds which correspond to ordinary cluster algebras \cite{Canakci-Tumarkin}.

\subsection{Generalized cluster algebras from orbifolds}\label{subsec:genCAfromOrb}

Recall that a subset of ordinary cluster algebras have a geometric realization in terms of triangulated surfaces, as discussed in Section \ref{sec:ordinaryClusterAlgebras}. Analogously, there exists a subset of generalized cluster algebras which can be realized geometrically as triangulated orbifolds. In these generalized cluster algebras, the exchange polynomials have the form $z_i = 1 + \lambda_p u + u^2$ if the cluster variable~$x_i$ corresponds to a~pending arc incident to an orbifold point of order~$p$ or $z_i = 1 + u$ if~$x_i$ is a standard arc.

Triangulated orbifolds first arose in a cluster algebra context when Felikson, Shapiro, and Tumarkin \cite{ FST-orbTriang, Felikson-Shapiro-Tumarkin} studied \emph{unfoldings} of skew-symmetrizable ordinary cluster algebras. Later, Chekhov and Shapiro \cite{Chekhov-Shapiro} showed that mutations for orbifold points of order greater than two are given by trinomial exchange relations with reciprocal coefficients.
Chekhov and Shapiro showed that both the Laurent phenomenon and positivity hold for such generalized cluster algebras using arguments similar to those given by Fomin and Zelevinsky \cite{Fomin-Zelevinsky-LP, Fomin-Zelevinsky-CAII} for ordinary cluster algebras. Labardini-Fragoso and Velasco~\cite{L-FV} showed that the generalized cluster algebras associated to polygons with a single orbifold point of order~3 are equivalent to Caldero--Chapoton algebras of quivers with relations arising from this polygon.

When working with triangulated orbifolds, it is often useful to consider some covering space. Which particular covering space is most useful varies depending on the application, but covering spaces that appear in the literature include the \emph{associated orbifolds} of Felikson, Shapiro, and Tumarkin \cite{FST-orbTriang} and the polygonal \emph{$p$-fold covering} of an orbifold with a single orbifold point of order~$p$~\cite{Chekhov-Shapiro, L-FV}.

\begin{figure}\centering
\includegraphics[scale=1.5]{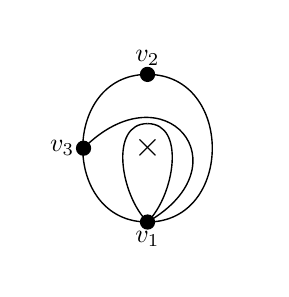}
\includegraphics[scale=1.3]{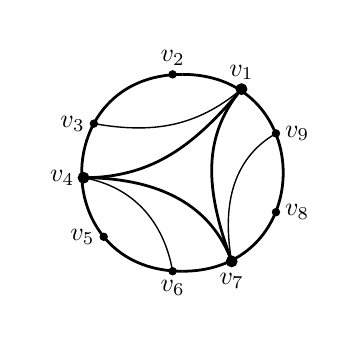}
\includegraphics[scale=1.3]{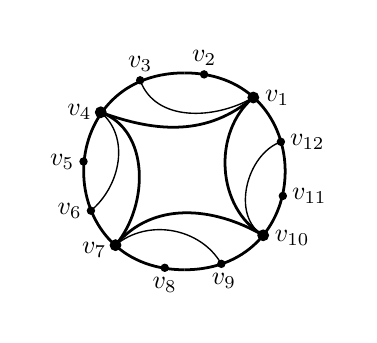}
\caption{An example of a triangulated orbifold with a single orbifold point of order $p$ (left) and the $p$-fold covering spaces for $p=3$ (middle) and $p=4$ (right).}\label{fig:pfold}
\end{figure}

\subsection{Laminations}\label{sec:generalizedLaminations}
In \cite[Section~6]{FST-orbTriang}, Felikson, Shapiro, and Tumarkin extend Fomin and Thurston's work on laminations~\cite{FT-II} in order to track coefficients for cluster algebras from orbifolds. Although they define laminations on an object called an \emph{associated orbifold} and we work with laminations on the original orbifold, much of their work transfers to our setting. We use their definition of a~lamination on an orbifold.

\begin{Definition}[{\cite[Definition 6.1]{FST-orbTriang}}]
Let $\mathcal{O} = (S,M,Q)$ be an unpunctured orbifold. An \emph{integral unbounded measured lamination} (henceforth just a \emph{lamination}) on $\mathcal{O}$ is a finite collection of non-self-intersecting and pairwise non-intersecting curves on $\mathcal{O}$ such that:
\begin{itemize}\itemsep=0pt
 \item Each curve is either a closed curve or a non-closed curve for which each end is either an unmarked point on the boundary of $\mathcal{O}$ or an orbifold point in~$Q$.
 \item No curve bounds an unpunctured disk or a disk containing a unique point of $M \cup Q$.
 \item No curve with both endpoints on the boundary of~$\mathcal{O}$ is isotopic to a portion of the boundary containing either no or one marked point(s).
 \item No two curves begin at the same orbifold point.
\end{itemize}
\end{Definition}

For the associated shear coordinates, however, we adopt a modified definition. The difference in our definition stems from the fact that we draw pending loops as arcs around orbifold points, rather than as having one endpoint at an orbifold point.

\begin{Definition}
Let $\mathcal{O}$ be an orbifold with triangulation $T$ and $L$ be a lamination on $\mathcal{O}$. For each arc $\gamma \in T$, the \emph{shear coordinate} of $L$ with respect to $T$ is
\[b_\gamma(T,L) = \sum_i b_\gamma(T,L_i), \]
where the summation runs over all individual curves in $L$. The shear coordinates $b_\gamma(T,L_i)$ are defined as:
\begin{center}
 \includegraphics[scale=1]{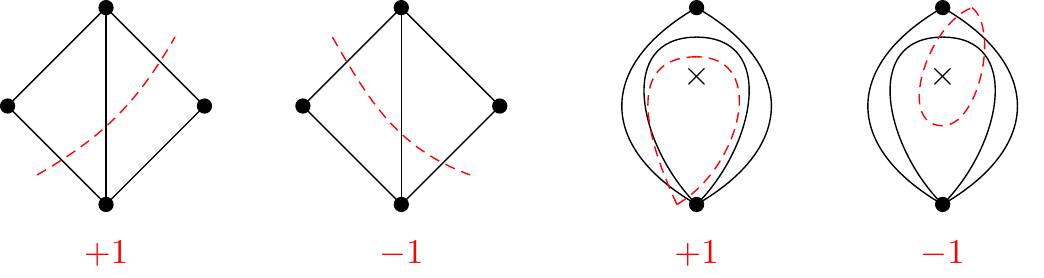}
\end{center}
\end{Definition}
To understand why these are the natural shear coordinate definitions for pending arcs, consider the corresponding view in the covering space. A pending curve $L_i$ for which $b_{\gamma}(T,L_i) = +1$ appears in the covering space as

\begin{center}
\begin{tikzpicture}[scale = 1.5]
\draw(0:2\R) -- (36:2\R) -- (72:2\R) -- (36*3:2\R) -- (36*4:2\R) -- (36*5:2\R) -- (36*6:2\R) -- (36*7:2\R) -- (36*8:2\R) -- (36*9:2\R) -- (0:2\R);
\draw(0:2\R) -- (72:2\R) -- (36*4:2\R) -- (36*6:2\R) -- (36*8:2\R) --(0:2\R);
\draw[dashed, red] (-10:2\R) -- (72-10:2\R) -- (36*4 - 10:2\R) -- (36*6 -10 :2\R) -- (36*8 - 10:2\R) --(-10:2\R);
\end{tikzpicture}
\end{center}

Notice that each copy of the lamination $L_i$ crosses a copy of the pending arc twice. One of these crossings contributes $+1$ to the shear coordinate and the other crossing contributes $0$, for a net shear coordinate of $+1$. The picture for $b_{\gamma}(T,L_i) = -1$ is analogous.

This extended shear coordinate definition then allows us to apply the usual definition of an elementary lamination to both standard and pending arcs.

Recall that if $\tau_i$ is a standard arc, the corresponding elementary lamination $L_i$ can be found by shifting its endpoints clockwise. Similarly, if $\tau_i$ is a pending arc, then $L_i$ can be found by shifting the singular endpoint clockwise. Examples of both are shown below.

\begin{center}
 \includegraphics[scale=1.5]{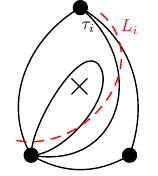}\qquad
 \includegraphics[scale=2]{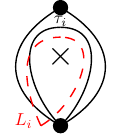}
\end{center}

Other types of crossings contribute 0 to the shear coordinate of the pending arc. If we look at these crossings in the cover, they resemble crossings in standard triangulations that contribute~0 to the shear coordinates.

\begin{center}
\begin{tikzpicture}[scale = 1.25]

\tikzset{->-/.style={decoration={
 markings,
 mark=at position #1 with {\arrow{>}}},postaction={decorate}}}

\draw[out=30,in=-30,looseness=1.5, xshift = -4\R, yshift = -2\R] (0,0.3) to (0,2.5);
\draw[out=150,in=-150,looseness=1.5, xshift = -4\R, yshift = -2\R] (0,0.3) to (0,2.5);

\draw[in=0,out=45,looseness =1.25, xshift = -4\R, yshift = -2\R] (0,0.3) to (0,2);
\draw[in=180,out=135,looseness=1.25, xshift = -4\R, yshift = -2\R] (0,0.3) to (0,2);

\draw[dashed, blue, xshift = -4\R, yshift = -2\R] (-1, 0.7) -- (1, 0.7);

\draw[fill=black, xshift = -4\R, yshift = -2\R] (0,0.3) circle [radius=2pt];
\draw[fill=black, xshift = -4\R, yshift = -2\R] (0,2.5) circle [radius=2pt];
\draw[thick, xshift = -4\R, yshift = -2\R] (0,1.7) node {$\mathbf{\times}$};

\draw(0:2\R) -- (36:2\R) -- (72:2\R) -- (36*3:2\R) -- (36*4:2\R) -- (36*5:2\R) -- (36*6:2\R) -- (36*7:2\R) -- (36*8:2\R) -- (36*9:2\R) -- (0:2\R);
\draw(0:2\R) -- (72:2\R) -- (36*4:2\R) -- (36*6:2\R) -- (36*8:2\R) --(0:2\R);
\draw[dashed, blue, thick] (-30: 2.1\R) -- (30: 2.1\R);
\draw[dashed, blue, thick] (-30 + 72: 2.1\R) -- (30+72: 2.1\R);
\draw[dashed, blue, thick] (-30+ 144: 2.1\R) -- (30+144: 2.1\R);
\draw[dashed, blue, thick] (-30-72: 2.1\R) -- (30-72: 2.1\R);
\draw[dashed, blue, thick] (-30-144: 2.1\R) -- (30-144: 2.1\R);
\end{tikzpicture}
\end{center}

However, these new elementary laminations associated to pending arcs will contribute 2 to a standard arc when they cross in a meaningful way. This contribution of 2 is also seen in generalized mutation rules.

\begin{center}
\begin{figure}[h!]
\begin{center}
\begin{tabular}{cc}
 \begin{tikzpicture}[scale = 2.5]
 \tikzset{->-/.style={decoration={markings, mark=at position #1 with {\arrow{>}}},postaction={decorate}}}

 \draw[out=30,in=-30,looseness=1.5] (0,0.3) to (0,1.5);
 \draw[out=150,in=-150,looseness=1.5] (0,0.3) to (0,1.5);

 \draw[in=0,out=45,looseness =1.25] (0,0.3) to (0,1.3);
 \draw[in=180,out=135,looseness=1.25] (0,0.3) to (0,1.3);

 \draw[in=0,out=30,looseness =1.25, dashed, red, thick] (-0.2,0.3) to (0,1.2);
 \draw[in=180,out=115,looseness=1.25, dashed, red, thick] (-0.2,0.3) to (0,1.2);
 \draw[out = 115, in = 210, dashed, green, looseness = 1.5, thick] (-0.2,0.3) to (0.2, 1.5);
 \draw[out = 30, in = 300, looseness = 1.5, dashed, blue, thick] (-0.2,0.3) to (0.2, 1.5);

 \draw[thick] (0,1) node {$\mathbf{\times}$};

 \node[] at (-0.26,1.4) {$\alpha$};
 \node[] at (0.26,1.41) {$\beta$};
 \node[] at (0,1.36) {$\rho$};

 \draw(0,0.3) -- (-1,.9) -- (0,1.5) -- (1,.9) -- (0,0.3);
 \end{tikzpicture} &
 \begin{tikzpicture}[scale=2.5]
 \draw[out=30,in=-30,looseness=1.5, xshift = 3\R] (0,0.3) to (0,1.5);
 \draw[out=150,in=-150,looseness=1.5, xshift = 3\R] (0,0.3) to (0,1.5);

 \draw[in=0,out=-45,looseness =1.25, xshift = 3\R] (0,1.5) to (0, 0.7);
 \draw[in=180,out=-135,looseness=1.25, xshift = 3\R] (0,1.5) to (0,0.7);

 \draw[in=0,out=30,looseness =1.25, dashed, red, thick, xshift = 3\R] (-0.2,0.3) to (0,1.2);
 \draw[in=180,out=115,looseness=1.25, dashed, red, thick, xshift = 3\R] (-0.2,0.3) to (0,1.2);
 \draw[out = 115, in = 210, dashed, green, looseness = 1.5, thick, xshift = 3\R] (-0.2,0.3) to (0.2, 1.5);
 \draw[out = 30, in = 300, looseness = 1.5, dashed, blue, thick, xshift = 3\R] (-0.2,0.3) to (0.2, 1.5);

 \draw[thick, xshift = 3\R] (0,1) node {$\mathbf{\times}$};

 \draw[xshift = 3\R] (0,0.3) -- (-1,.9) -- (0,1.5) -- (1,.9) -- (0,0.3);
 \end{tikzpicture} \\
 $\begin{bmatrix} 0 & 1 & -1 \\ -1 & 0 & 1 \\ 1 & - 1 & 0 \\ 1 & 0 & 0 \\ 0 & 1 & 0 \\ 0 & 0 & 1 \\ \end{bmatrix}$ & $\begin{bmatrix} 0 & -1 & 1 \\ 1 & 0 & -1 \\ -1 & 1 & 0 \\ 1 & 0 & 0 \\ 0 & 1 & 0 \\ 2 & 0 & -1 \\ \end{bmatrix}$ \\
\end{tabular}
\end{center}
\caption{An example of how flipping a pending arc impacts the shear coordinates of a lamination. The lefthand side of the diagram is an example of an elementary lamination from a triangulated orbifold. This diagram is best viewed in color.}\label{fig:LamFlipPending}
\end{figure}
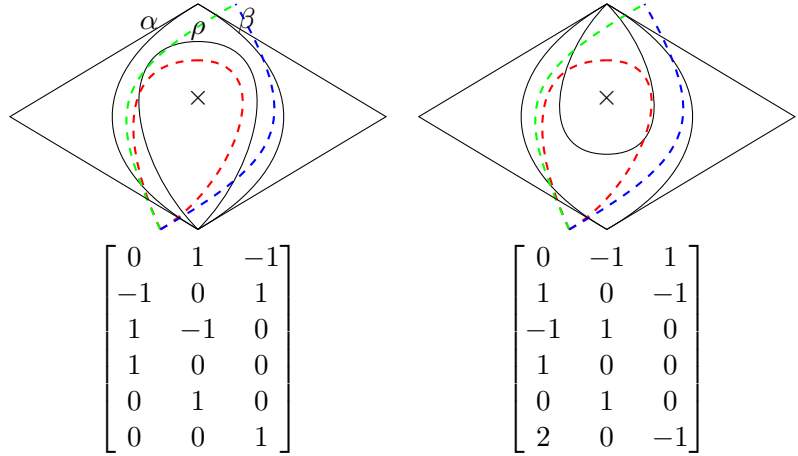
\end{center}

On the left of Fig.~\ref{fig:LamFlipPending}, we have the elementary lamination associated to this triangulation. When we flip the pending arc, then the lamination associated to the pending arc intersects $\alpha$ nontrivially twice. This matches the result of mutating (with generalized mutation) the extended $B$-matrix associated to the lefthand picture at the index representing the pending arc. The third column and row correspond to the pending arc.

\begin{Remark}
The mutation of the extended portion of the B-matrix resembles the result of mutating at an orbifold point of weight 2 in the sense of \cite{Felikson-Shapiro-Tumarkin}. While the dynamics of our $x$-variables mimic those of orbifold points of weight $\frac12$, the $y$-variables more closely resemble those from orbifold points of weight $2$. This seems to follow from the duality between $c$-vectors and $g$-vectors \cite{NZ}.
\end{Remark}

\section{Constructing snake graphs from orbifolds} \label{Sec:Construction}

In this section, we show how to construct a snake graph (respectively, a band graph) from an arc (closed curve) on a trianguated unpunctured orbifold. Later, we will verify that the weighted sum of perfect matchings gives the correct cluster variable in the case when $\gamma$ is an arc, Theorem~\ref{Thm:GenSnake}, and that for arbitrary $\gamma$, these satisfy skein relations, Proposition~\ref{Prop:OrbSkein}.

\subsection{Tiles}

If $\gamma = \tau_i$ for $1 \leq i \leq n$ (recall the final $c - n$ arcs are boundary arcs), then $G_\gamma$ is a single edge labeled with $\tau_i$. Otherwise, $\gamma$ must cross at least one arc in $T$.

 Let $\tau_{i_1},\ldots, \tau_{i_d}$ be the set of internal arcs of $T$ that $\gamma$ crosses, given a fixed orientation of $\gamma$. For each standard arc $\tau_{i_j}$ that $\gamma$ crosses, we construct a square tile $G_{j}$ by taking the two triangles that $\tau_{i_j}$ borders and gluing them along $\tau_{i_j}$ such that either both either the same orientation relative to $\mathcal{O}$. We say that the square tile produced has relative orientation $+1$ if the orientation of its triangles matches that of $\mathcal{O}$ and $-1$ otherwise. We denote this as $\operatorname{rel}(G_j) = \pm 1.$

\begin{center}
\begin{tikzpicture}[scale = .14cm]
\draw (45:0.3cm) to node[auto]{$a$}(135:0.3cm) to node[auto]{$b$}(180+45:0.3cm) to node[auto]{$c$}(270+45:0.3cm) to node[auto]{$d$} (45:0.3cm);
\draw[gray](45:0.3cm) to node[below]{$\rho_i$} (225:0.3cm);
\draw[xshift = 1 cm] (45:0.3cm) to node[auto]{$b$}(135:0.3cm) to node[auto]{$a$}(180+45:0.3cm) to node[auto]{$d$}(270+45:0.3cm) to node[auto]{$c$} (45:0.3cm);
\draw[xshift = 1 cm, gray](45:0.3cm) to node[below]{$\rho_i$} (225:0.3cm);
\end{tikzpicture}
\end{center}

Next, we consider the case when $\tau_{i_j}$ is a pending arc incident to an orbifold point of order $p$. If $\gamma$ is a generalized arc who shares an endpoint with $\tau_{i_j}$, then it could be that $\gamma$ only crosses~$\tau_{i_j}$ once. In this case, $j = 1$ or $j = k$, and we use a square tile as before. However, the labels of some edges will be given by normalized Chebyshev polynomials, $U_\ell(x)$, evaluated at~$\lambda_p$. Recall $\lambda_p = 2\cos(\pi/p)$.

\begin{Definition} \label{def:Chebyshev}
We let $U_\ell(x)$ denote the $\ell$-th normalized Chebyshev polynomial of the second kind, for $\ell \geq -1$. These are given by initial polynomials $U_{-1}(x) = 0$, $U_0(x) = 1$, and the recurrence, \[
U_\ell(x) = xU_{\ell-1}(x) - U_{\ell-2}(x).
\]
\end{Definition}

For instance, $U_1(x) = x$, $U_2(x) = x^2 - 1$, $U_3(x) = x^3 - 2x$. These polynomials are normalized as they can be recovered by evaluating the standard Chebyshev polynomials of the second kind at~$x/2$.

The following lemma verifies that, up to sign, these labels are independent of increasing or decreasing the winding around an orbifold point by an integer multiple of its order.

\begin{Lemma}\label{lem:ChebyPeriodic}
Evaluations of Chebyshev polynomials at $\lambda_p$ are periodic, in the sense that $U_{k+p}(\lambda_p)\!\allowbreak = -U_{k}(\lambda_p)$. In particular, $U_{p-1}(\lambda_p) = -U_{-1}(\lambda_p) = 0$.
\end{Lemma}

Lemma \ref{lem:ChebyPeriodic} can be readily proven using basic properties of Chebyshev polynomials. We see in Lemmas \ref{lem:Mpathdontcare} and \ref{lem:ChebyshevMatrices} that our statistics are still well-defined up to sign.

The edge labels of these tiles contain $U_\ell(\lambda_p)$ and $U_{\ell - 1}(\lambda_p)$ where $\ell$ is the number of self-intersections of $\gamma$ around the orbifold point. For concision, $U_\ell$ is used as shorthand for $U_\ell(\lambda_p)$ throughout the paper. Moreover, $\alpha$ and $\beta$ may be standard or pending arcs. If one of these arcs is pending, then this is in fact a monogon enclosing two orbifold points.

\begin{center}
\begin{tabular}{|c|c|}
\hline
 \includegraphics[scale=1.5]{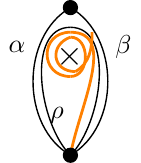}& \includegraphics[scale=1.5]{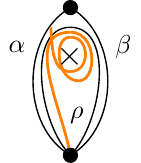} \\ \hline \includegraphics[scale=1.5]{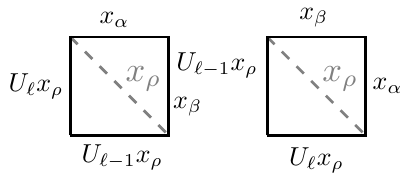} & \includegraphics[scale=1.5]{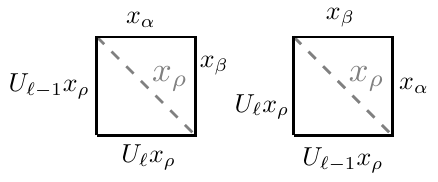} \\ \hline
\end{tabular}
\end{center}

Above, each tile on the left has positive orientation and the tile on the right has negative orien\-tation. We can see this, for example, by comparing the relative orientation of edges labeled~$x_2$ and~$x_a$ with~$\tau_2$ and $\tau_a$.

\begin{Remark}Musiker and Williams discuss a similar example in~\cite{MW} with a puncture rather than an orbifold point. We compare these cases in Section~\ref{sec:RelateToPuncture}.
\end{Remark}

In most cases, if $\gamma$ crosses a pending arc $\tau_{i_j}$, it crosses it twice consecutively, so that $\tau_{i_j} = \tau_{i_{j+1}}$ or $\tau_{i_j} = \tau_{i_{j-1}}$. In this case, we introduce a hexagonal tile which accounts for both intersections. These hexagonal tiles also will have edges labeled by Chebyshev polynomials evaluated at $\lambda_p$, and we again let $\ell$ be the number of self-intersections of $\gamma$ as it winds around the orbifold point. Because these hexagonal tiles can be thought of as ``containing" two square tiles, we assign them a tuple of signs. A hexagonal tile has relative orientation $(+,-)$ if the south-west triangle matches the orientation of the surface and the north-east triangle does not, as on the left hand side of the below diagram, and $(-,+)$ otherwise, as on the right hand side.

\begin{center}
 \includegraphics[scale=1.5]{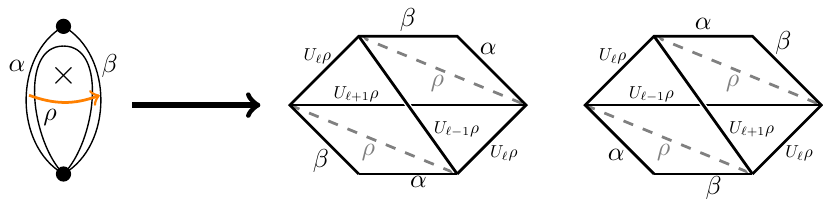}
\end{center}

In Section \ref{subsec:LiftGenArc}, we give a geometric intuition for why the edge labels $U_\ell\rho$, $U_{\ell+1}\rho$, and $U_{\ell-1}\rho$ appear in this particular arrangement on the hexagonal tiles. This geometric intuition is based on crossing diagonals in the $p$-fold cover. We formally justify these hexagonal tiles in Section~\ref{sec:universalSnakeGraph}, however, using matrix products associated to arcs, perfect matchings of abstract graphs, and Lemma~\ref{lem:ChebyshevMatrices}.

We give puzzle pieces to construct a generalized snake graph from such an arc. Again, $\alpha$~and~$\beta$ could be standard or pending arcs.

\begin{center}
{\renewcommand{\arraystretch}{2}
\begin{tabular}{|@{\,}c@{\,}|@{\,}c@{\,}|@{\,}c@{\,}|@{\,}c@{\,}|}
\hline
 \includegraphics[scale = 2.1]{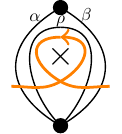} &
 \includegraphics[scale=2.1]{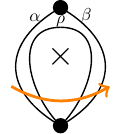} & \includegraphics[scale = 2.1]{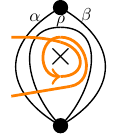} & \includegraphics[scale=2.1]{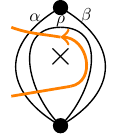} \\ \hline \includegraphics[scale = 1.2]{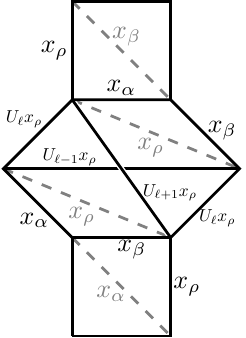} & \includegraphics[scale=1.2]{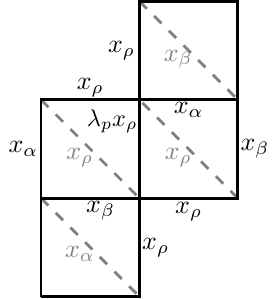} & \includegraphics[scale = 1.2]{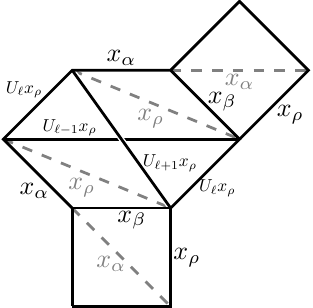} & \includegraphics[scale=1.2]{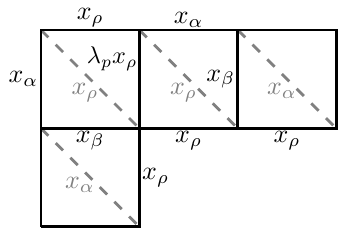}
 \\ \hline
\end{tabular}}

\end{center}

Note that if $\ell = 0$ (i.e., the arc does not intersect itself) then the edge labeled $U_{\ell-1}(\lambda_p)$ has weight 0. Thus, we can delete it and recover a standard snake graph with square tiles. We show side by side the general hexagonal tiles and the $k=0$ cases. By Lemma \ref{lem:ChebyPeriodic}, this is also true if $k = p-2$. Moreover, if $\ell = p-1$, then two edges have weight zero and one has negative weight. Using the symmetry of an orbifold point, this is equivalent to not crossing the pending arc at all. Thus, we will assume that $\gamma$ winds less than $p-1$ times around an orbifold point of order~$p$ to avoid including absolute values in our labels.

\subsection[Gluing $G_{T,\gamma}$ puzzle pieces]{Gluing $\boldsymbol{G_{T,\gamma}}$ puzzle pieces}

To construct generalized snake graphs, we will glue together tiles corresponding to arcs crossed consecutively by $\gamma$. If $\gamma$ crosses $\tau_i$ and $\tau_{i+1}$ consecutively, and $\tau_i$ and $\tau_{i+1}$ are distinct arcs, then these arcs form a triangle. Call the third arc in this triangle $\tau_{[i]}$. Then, we glue tiles~$G_i$ and~$G_{i+1}$ along the edge $\tau_{[i]}$. using the appropriate planar embeddings so $\operatorname{rel}(T,G_i) \neq \operatorname{rel}(T,G_{i+1})$. Note that this rule does not differentiate between standard and pending arcs. If~$G_i$ and~$G_{i+1}$ are either both square or both hexagonal, then the statement of the rule is clear. If~$G_i$ is square and~$G_{i+1}$ is hexagonal, then $\operatorname{rel}(T,G_{i+1})$ should be understood to mean the orientation of the south-west triangle of $G_{i+1}$. Likewise, if $G_i$ is hexagonal and $G_{i+1}$ square, then $\operatorname{rel}(T,G_{i})$ should be understood to mean the orientation of the north-east triangle of~$G_i$.

Because the choice of relative orientation for the first tile,~$G_1$, is not fixed, there are two valid planar embeddings of $G_{T,\gamma}$ for any $\gamma$. Our cluster expansion formula produces the same result for either choice of planar embedding, so the choice is unimportant. We also make a choice to glue the tiles so that our snake graphs travel from south-west to north-east; this also will not affect any statistics related to the snake graph.

Finally, we can construct generalized band graphs using the same ideas. Band graphs calculate the length of closed curves on a surface. Choose a point~$p$ on $\gamma$ such that~$p$ does not lie on any arc in $T$ or at an intersection of $\gamma$ with itself. For simplicity, we require $p$ to not be in the interior of a pending arc. Then, construct the snake graph for $\gamma$, picking an orientation and starting and ending at~$p$. Because the first and last tile correspond to arcs bordering the same triangle, they will always have a common edge. We glue the first and tile along this edge, producing a graph which resembles an annulus or a Mobius strip.

Band graphs have the same associated statistics as snake graphs and a version of a perfect matching on a band graph, called a \emph{good matching}, is defined similarly. Musiker and Williams note that a good matching of a band graph can always be obtained from a perfect matching of the original, unglued snake graph used to construct the band graph. To do so, one takes a perfect matching that uses at least one of the glued edges and deletes that glued edge. For further discussion and details, see \cite[Section~3]{MW}.

\subsection{Cluster expansion formulas}

We use Musiker, Schiffler, and Williams' definitions for minimal and maximal matchings, the crossing monomial $\operatorname{cross}(T,\gamma)$, and the weight $x(P)$ and height monomial~$y(P)$ associated to a~perfect matching~$P$, as stated in Section \ref{section:snakeGraphs}. Using this language, we can establish the following theorem for Laurent expansions of arcs (both standard and pending), a more general version of Theorem~4.9 from~\cite{MSW}.

\begin{Theorem}\label{Thm:GenSnake+}
Let $\mathcal{O} = (S,M,Q)$ be an unpunctured orbifold with triangulation $T$ and $\mathcal{A}$ be the corresponding generalized cluster algebra with principal coefficients with respect to $\Sigma_T = (\mathbf{x}_T,\mathbf{y}_T,B_T)$. For an ordinary arc $\gamma$ with generalized snake graph $G_{T,\gamma}$, the Laurent expansion of $x_\gamma$ with respect to $\Sigma_T$ is
\[ [x_\gamma]^{\mathcal{A}}_{\Sigma_T} = \frac{1}{\operatorname{cross}(T,\gamma)} \sum_P x(P)y(P), \]
where the summation is indexed by perfect matchings of $G_{T,\gamma}$.
\end{Theorem}

\begin{Example}
 The table below shows snake graphs for a variety of curves on the triangulated orbifold corresponding to $\mathcal{A} = \left(\mathbf{x},\mathbf{y}, \left[\begin{smallmatrix} 0 & -1 \\ 1 & 0 \end{smallmatrix}\right], \big(1+\mu u + u^2, 1 + \lambda u + u^2\big) \right)$.
\end{Example}

\begin{center}
\begin{tabular}{|@{\,}c@{\,}|@{\,}c@{\,}||@{\,}c@{\,}|@{\,}c@{\,}|}
\hline
 \includegraphics[scale=1.6]{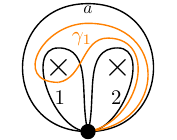} & \includegraphics[scale=1]{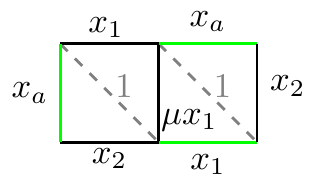} &\includegraphics[scale=1.6]{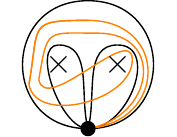} & \includegraphics[scale=1]{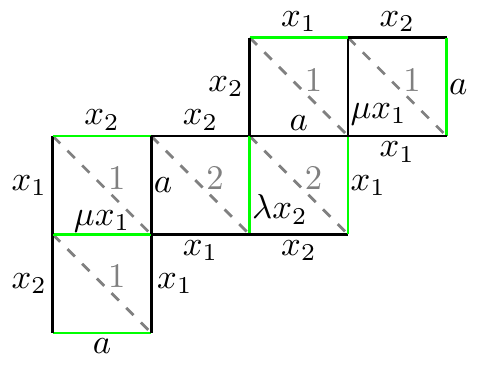} \\ \hline
 \multicolumn{2}{@{\,}c@{\,}}{$x_{\gamma_1} = \frac{1}{x_1}\big(\textcolor{green}{x_a^2} + \mu y_2x_2x_a + y_2^2 x_2^2$\big)} & \multicolumn{2}{@{\,}c@{\,}}{$\begin{array}{@{\,}l@{\,}} x_{\gamma_2} = \frac{1}{x_1^2x_2}\big(x_a^2x_1^2y_1^4y_2^2+\textcolor{green}{\mu\lambda x_a^2 x_1x_2y_1^3y_2}+\lambda x_a x_1x_2^2y_1^2y_2 \\ \qquad \qquad {}+ \lambda x_a^3x_1 y_1^4 y_2 + \mu^2 x_a^2 x_2^2 y_1^2 + 2\mu x_a x_2^3 y_1 \\ \qquad\qquad{} + 2 \mu x_a^3 x_2 y_1^3 + x_2^4 + 2x_a^2x_2^2y_1^2 + x_a^4y_1^4\big) \end{array}$}
 \end{tabular}
\end{center}
\begin{center}
\begin{tabular}{|@{\,}c@{\,}|@{\,}c@{\,}||@{\,}c@{\,}|@{\,}c@{\,}|}
 \hline
\includegraphics[scale=1.7]{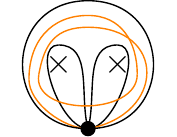} & \includegraphics[scale=1]{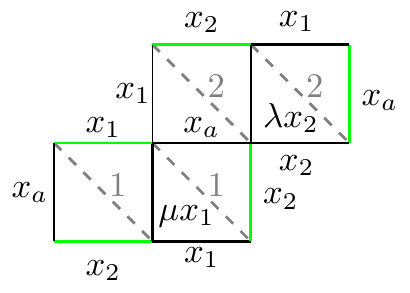}&\includegraphics[scale=1.7]{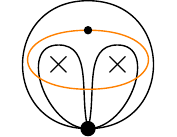}& \includegraphics[scale=1]{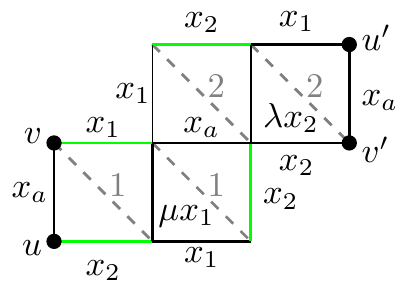}\\ \hline
\multicolumn{2}{@{\,}c@{\,}}{$\begin{array}{@{\,}l@{\,}} x_{\gamma_3} = \frac{1}{x_1x_2}\big(x_ax_1^2+ \lambda y_2x_a^2x_1 + y_2^2 x_a^3 \\ \qquad \qquad {}+\mu y_1 y_2^2 x_a^2 x_2 + \textcolor{green}{y_1^2y_2^2 x_ax_2^2}\big) \end{array}$} &
\multicolumn{2}{@{\,}c@{\,}}{$\begin{array}{@{\,}l@{\,}} x_{\gamma_4} = \frac{1}{x_1x_2}\big(x_1^2+ \lambda y_2x_ax_1 + y_2^2 x_a^2 \\ \qquad \qquad {}+ \mu y_1 y_2^2 x_a x_2 + \textcolor{green}{y_1^2y_2^2x_2^2}\big)\end{array}$}
\end{tabular}
\end{center}

Labels for arcs in the initial triangulation are only shown in the first orbifold diagram, but are consistent throughout. Snake graphs are shown for each (orange) curve $\gamma_i$, with one perfect matching and the corresponding term in the Laurent expansion highlighted. Both $\gamma_1$ and $\gamma_2$ are cluster variables of $\mathcal{A}$ which can be obtained via the respective mutation sequences $\mu_1$ and $\mu_2\mu_1$.

The second half of this example illustrates our results for generalized arcs and closed curves. Since $\gamma_3$ and $\gamma_4$ cross the same arcs in the same orientation, the shapes of the two associated graphs are the same. However, in the band graph associated to $\gamma_4$, we identify $u$ with $u'$ and $v$ with $v'$. In each graph, we have highlighted the maximal matching and the corresponding term in the Laurent expansion.

Note that in each example, our expression for $x_{\gamma_i}$ is given after canceling a mutual factor from the crossing monomial and the numerator. Although the exact mutual factor depends on the curve being considered, cancellation of this type occurs whenever we cross pending arcs.

In the Sections~\ref{sec:lift} to~\ref{sec:ClusterExpansionPf} we prove Theorem~\ref{Thm:GenSnake+} when $\gamma$ is an ordinary arc. Then, $x_\gamma$ is a~cluster variable in the associated generalized cluster algebra. Moreover, we are able to lift $\gamma$ to a~construct a triangulated polygon where expansion formulas are already known. In Section~\ref{subsec:LiftGenArc}, we explain why this tactic does not work for generalized arcs.

\begin{figure}[h!] \centering
 \begin{tabular}{|c|c|}
 \hline
 \includegraphics[scale=1.0,valign=m]{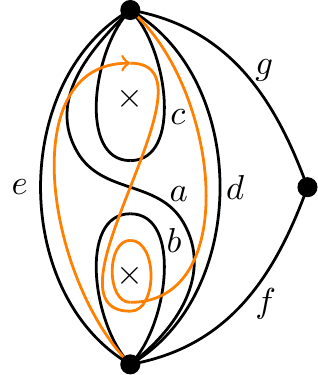} & \includegraphics[scale=1.0,valign=m]{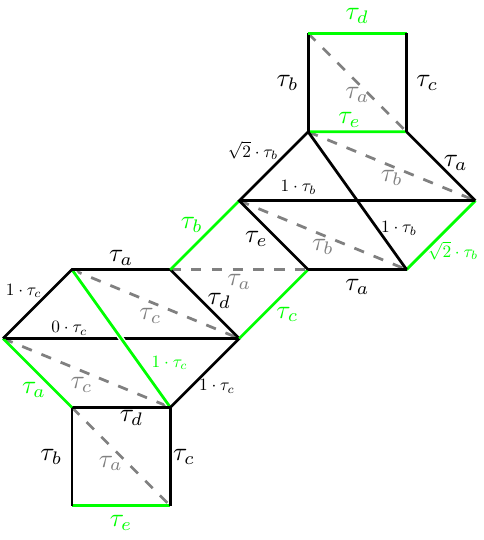} \\ \hline
 \multicolumn{2}{c}{$ x_{\gamma} = \frac{1}{x_a^3x_b^2x_c^2}\big(\textcolor{green}{\sqrt{2}y_ay_b^2y_cx_a x_b^2 x_c^3 x_d x_e^2}+ y_ay_by_c^2x_ax_b^2 x_c^2x_d^2x_e + \cdots\big) $}
 \end{tabular}
 \caption{An example of a generalized snake graph from a triangulated orbifold with one orbifold point of order 3 (top) and one orbifold point of order 4 (bottom).} \label{fig:my_label}
\end{figure}

\section[The lift $\widetilde{S_\gamma}$]{The lift $\boldsymbol{\widetilde{S_\gamma}}$}\label{sec:lift}

In the following sections, let $\mathcal{O} = (S,M,Q)$ with triangulation $T = \{ \tau_1,\ldots, \tau_n, \tau_{n+1},\ldots, \tau_{n+c}\}$ where $\tau_1,\ldots, \tau_n$ are internal arcs and $\tau_{n+1},\ldots, \tau_{n+c}$ are boundary arcs. Let $\gamma \notin T$ be an arc on an orbifold, and pick an orientation of $\gamma$. Let $\tau_{i_1},\ldots, \tau_{i_d}$ be the arcs crossed by $\gamma$ in order. Note that it is possible to have $j \neq k$ and $\tau_{i_k} = \tau_{i_j}$, since $\gamma$ may cross a given arc in $T$ multiple times. It is even possible to have $\tau_{i_j} = \tau_{i_{j+1}}$; this occurs only when $\tau_{i_j}$ is a pending arc.

We define a polygon $\widetilde{S_\gamma}$ with triangulation $\widetilde{T_\gamma}$ which lifts the local configuration of $\mathcal{O}$ and $T$ around $\gamma$. The triangulation $\widetilde{T_\gamma}$ consists of arcs $\sigma_1,\ldots, \sigma_d,\sigma_{d+1},\ldots, \sigma_{2d + 3}$ where $\sigma_{d+1},\ldots, \sigma_{2d + 3}$ are boundary arcs. We also construct a lift of $\gamma$ in $\widetilde{S_\gamma}$, denoted as $\widetilde{\gamma}$; in short, $\widetilde{\gamma}$ will be the arc in $\widetilde{S_\gamma}$ which crosses all arcs in $\widetilde{T_\gamma}$.

\looseness=1 Musiker, Schiffler, and Williams gave a construction of $\widetilde{S_\gamma}$ and $\widetilde{T_\gamma}$ for the case where $\gamma$ is an arc on a surface \cite{MSW}. We describe an extension to their construction and refer the interested reader to their paper for details of the original construction. Essentially, they keep track of when consecutive arcs in $T$, $\tau_{i_j}$ and $\tau_{i_{j+1}}$, share a vertex on the right or on the left of $\gamma$. We will let~$t_j$ denote the vertex shared by~$\tau_{i_j}$ and~$\tau_{i_{j+1}}$.
The corresponding consecutive arcs~$\sigma_j$ and~$\sigma_{j+1}$ in~$\widetilde{T_\gamma}$ share a vertex, $s_j$, on the same side of $\widetilde{\gamma}$ and $\widetilde{S_\gamma}$ is constructed by gluing together the fans formed by sets of consecutive arcs in $T$ which share a given vertex~$t_j$. Musiker, Schiffler, and Williams also provide a projection map $\pi\colon \widetilde{T_\gamma} \to T$ such that $\pi(\sigma_j) = \tau_{i_j}$. This map also can be applied to boundary arcs in $\widetilde{S_\gamma}$; we will give a full definition of~$\pi$ in Section~\ref{Sec:mapSection}.

This construction can be used in the orbifold case when $\gamma$ crosses consecutive standard arcs; what remains is to analyze the case when $\gamma$ crosses a pending arc.

There are several possible configurations for this case. Let $\tau_{i_j}$ be a pending arc; then, $\tau_{i_j}$ is enclosed by a bigon or monogon with sides $\alpha$ and $\beta$. If this is a bigon, let $v$ be the vertex shared by $\tau_{i_j}$, $\alpha,$ and $\beta$ and let $w$ be the vertex only shared by $\alpha$ and $\beta$, as shown below. If this is a~monogon, let $v = w$ be the unique vertex shared by $\alpha$, $\beta$ and $\rho$. Our configuration of $s_{j-1}$, $s_j,$ and $s_{j+1}$ will depend on how $\gamma$ interacts with the bigon.

It could be that $\gamma$ is the result of flipping $\rho$. In this case, $d = 2$, $s_1 = v$, and $\widetilde{S_\gamma}$ is a~triangulated pentagon, as below. We label arcs $\sigma_j$ with $\widetilde{\tau_k}$ if $\pi(\sigma_j) = \tau_k$.

\begin{center}
\begin{tikzpicture}[scale=1.8]
\node[circle, fill = black, scale = 0.25] (A) at (0,0) {$w$};
\node[circle, fill = black, scale = 0.25] (B) at (2,0) {$v$};
\node[] at (-.2,0) {$w$};
\node[] at (2.2,0) {$v$};
\node[] (E) at (1,0) {$\times$};
\node[](F) at (0.8, 0){};
\node[](G) at (1.2,0){};
\draw (A) to [out=60,in=120] node[auto]{$\alpha$}(B);
\draw (A) to [out=-60,in=-120] node[below]{$\beta$}(B);
\draw(2,0) to [out = 150, in = 90] node[below, xshift = 2pt]{$\rho$}(0.8, 0);
\draw(2,0) to [out = 210, in = 270] node[auto]{}(0.8, 0);
\draw[dashed, red](0,0) to [out = 30, in = 90] (1.2,0);
\draw[dashed, red](0,0) to [out = -30, in = 270] node[below]{$\gamma$}(1.2,0);
\node[](arrw) at (3,0) {$\rightarrow$};
\node[circle, fill = black, scale = 0.1](P1) at (4,-1){};
\node[circle, fill = black, scale = 0.1](P2) at (6,-1){};
\node[](P5) at (4,0){$s_0 = w$};
\node[](P3) at (6,0){$w = s_d$};
\node[](P4) at (5,0.7){$s_1$};
\draw (P1) -- node[below] {$\lambda_p \cdot \tilde{\rho}$} (P2) -- node[right]{$\tilde{\alpha}$} (P3) -- node[above]{$\tilde{\beta}$} (P4) -- node[left]{$\tilde{\alpha}$} (P5) -- node[left]{$\tilde{\beta}$} (P1);
\draw(P4) -- node[auto] {$\tilde{\rho}$} (P1);
\draw (P4) -- node[left] {$\tilde{\rho}$}(P2);
\draw[dashed, red] (P5) -- (P3);
\end{tikzpicture}
\end{center}

Next, consider the case where $d > 2$ and $\tau_{i_1}$ or $\tau_{i_d}$ is a pending arc. Suppose, without loss of generality, that $\tau_{i_1} = \rho$ is a pending arc. Since, as an ordinary arc, $\gamma$ necessarily crosses $\rho$ twice, $\rho = \tau_{i_1} = \tau_{i_2}$. Then regardless of whether $\tau_{i_3} = \alpha$ or $\tau_{i_3} = \beta$, we set $s_1 = s_2 $. See below for the case where $\tau_{i_3} = \beta$.

\begin{center}
\begin{tikzpicture}[scale=1.8]
\node[circle, fill = black, scale = 0.25] (A) at (0,0) {$w$};
\node[circle, fill = black, scale = 0.25] (B) at (2,0) {$v$};
\node[] at (-.2,0) {$w$};
\node[] at (2.2,0) {$v$};
\node[] (E) at (1,0) {$\times$};
\node[](F) at (0.8, 0){};
\node[](G) at (1.2,0){};
\node[] (H) at (1.2, -0.8) {};
\draw (A) to [out=60,in=120] node[auto]{$\alpha$}(B);
\draw (A) to [out=-60,in=-120] node[below]{$\beta$}(B);
\draw(2,0) to [out = 150, in = 90] node[below, xshift = 2]{$\rho$}(0.8, 0);
\draw(2,0) to [out = 210, in = 270] node[auto]{}(0.8, 0);
\draw[dashed, red](0,0) to [out = 30, in = 90] node[above]{$\gamma$}(1.2,0);
\draw[dashed, red](1.2,0) to (1.2,-0.8);
\node[](arrw) at (2.5,0) {$\rightarrow$};
\node[circle, fill = black, scale = 0.1](P1) at (4,-1){};
\node[circle, fill = black, scale = 0.1](P2) at (5,-1){};
\node[circle, fill = black, scale = 0.1](P3) at (6,-1){};
\node[](P6) at (4,0){$s_0 $};
\node[](P4) at (6,0){$\cdots$};
\node[](P5) at (5,0.7){$s_1 = s_2$};
\draw (P1) -- node[below] {$\lambda_p \cdot \tilde{\rho}$} (P2) -- node[below]{$\tilde{\alpha}$} (P3) -- node[above]{} (P4) -- node[left]{} (P5)-- node[left]{$\tilde{\alpha}$} (P6) -- node[left]{$\tilde{\beta}$} (P1);
\draw(P5) -- node[auto] {$\tilde{\rho}$} (P1);
\draw (P5) -- node[left] {$\tilde{\rho}$}(P2);
\draw (P5) -- node[left] {$\tilde{\beta}$}(P3);
\draw[dashed, red] (P6) -- (P4);
\end{tikzpicture}
\end{center}

Finally, we have two cases for when $\gamma$ crosses the bigon twice. In this case, $d \geq 4$ and $j >1$. If $\gamma$ crosses both sides of the bigon, then $s_{j-1} = s_j = s_{j+1}$.

\begin{center}
\begin{tikzpicture}[scale=1.8]
\node[circle, fill = black, scale = 0.25] (A) at (0,0) {$w$};
\node[circle, fill = black, scale = 0.25] (B) at (2,0) {$v$};
\node[] at (-.2,0) {$w$};
\node[] at (2.2,0) {$v$};
\node[] (E) at (1,0) {$\times$};
\node[](F) at (0.8, 0){};
\node[](G) at (1.2,0.8){};
\node[] (H) at (1.2, -0.8) {};
\draw (A) to [out=60,in=120] node[auto]{$\alpha$}(B);
\draw (A) to [out=-60,in=-120] node[below]{$\beta$}(B);
\draw(2,0) to [out = 150, in = 90] node[below, xshift = 2]{$\rho$}(0.8, 0);
\draw(2,0) to [out = 210, in = 270] node[auto]{}(0.8, 0);
\draw[dashed, red](G) to node[right] {$\gamma$} (H);
\node[](arrw) at (2.5,0) {$\rightarrow$};
\node[circle, fill = black, scale = 0.1](P0) at (3,-1){};
\node[circle, fill = black, scale = 0.1](P1) at (4,-1){};
\node[circle, fill = black, scale = 0.1](P2) at (5,-1){};
\node[circle, fill = black, scale = 0.1](P3) at (6,-1){};
\node[](P6) at (3,0){$\cdots$};
\node[](P4) at (6,0){$\cdots$};
\node[](P5) at (5,0.7){$s_{j-1} = s_j = s_{j+1}$};
\draw (P1) -- node[below] {$\lambda_p \cdot \tilde{\rho}$} (P2) -- node[below]{$\tilde{\alpha}$} (P3) -- node[above]{} (P4) -- node[left]{} (P5)-- (P6) -- (P0) -- node[below]{$\tilde{\beta}$} (P1);
\draw(P5) -- node[left] {$\tilde{\rho}$} (P1);
\draw (P5) -- node[left] {$\tilde{\rho}$}(P2);
\draw (P5) -- node[left] {$\tilde{\beta}$}(P3);
\draw(P5) -- node[left] {$\tilde{\alpha}$} (P0);
\draw[dashed, red] (P6) -- (P4);
\end{tikzpicture}
\end{center}

Alternatively, $\gamma$ could cross the same side of the bigon both before and after crossing $\tau_{i_j}$. That is to say, $\tau_{i_j} = \tau_{i_{j+1}}$ and $\tau_{i_{j-1}} = \tau_{i_{j+2}}$. If the first point of intersection between $\gamma$ and $\tau_{i_{j-1}} = \tau_{i_{j+2}}$ is closer to $v$ than the second point of intersection, then set $s_{j-1} = s_j $ and $s_{j+1} \neq s_j$; otherwise, set $s_j = s_{j+1} $ and $s_{j-1} \neq s_j$. See below for an example where $\tau_{i_{j-1}} = \tau_{i_{j+2}} = \beta$ and the second point of intersection is closer than the first.

\begin{center}
\begin{tikzpicture}[scale=1.8]
\node[circle, fill = black, scale = 0.25] (A) at (0,0) {$w$};
\node[circle, fill = black, scale = 0.25] (B) at (2,0) {$v$};
\node[] at (-.2,0) {$w$};
\node[] at (2.2,0) {$v$};
\node[] (E) at (1,0) {$\times$};
\node[](F) at (0.8, 0){};
\node[](G) at (1.2,0){};
\node[] (H) at (1.2, -0.8) {};
\node[] (I) at (0.8, -0.8) {};
\draw (A) to [out=60,in=120] node[auto]{$\alpha$}(B);
\draw (A) to [out=-60,in=-120] node[below]{$\beta$}(B);
\draw(2,0) to [out = 150, in = 90] node[below, xshift = 2]{$\rho$}(0.8, 0);
\draw(2,0) to [out = 210, in = 270] node[auto]{}(0.8, 0);
\draw[dashed, red](1.2,0) to node[right] {$\gamma$} (1.2,-0.8);
\draw[dashed, red, <-= -1] (1.2,0) to [out = 90, in = 90] (0.7,0);
\draw[dashed, red] (0.7,0) to (0.7, -0.8);
\node[](arrw) at (2.5,0) {$\rightarrow$};
\node[circle, fill = black, scale = 0.1](P0) at (3.5,0.7){};
\node[](P1) at (4,-1){$s_{j-1}$};
\node[circle, fill = black, scale = 0.1](P2) at (5,-1){};
\node[circle, fill = black, scale = 0.1](P3) at (6,-1){};
\node[](P6) at (3,0){$\cdots$};
\node[](P4) at (6,0){$\cdots$};
\node[](P5) at (5,0.7){$s_j = s_{j+1}$};
\draw (P1) -- node[below] {$\lambda_p \cdot \tilde{\rho}$} (P2) -- node[below]{$\tilde{\alpha}$} (P3) -- node[above]{} (P4) -- (P5)-- node[above]{$\tilde{\alpha}$} (P0) -- (P6) -- (P1);
\draw(P5) -- node[left] {$\tilde{\rho}$} (P1);
\draw (P5) -- node[left] {$\tilde{\rho}$}(P2);
\draw (P5) -- node[left] {$\tilde{\beta}$}(P3);
\draw(P1) -- node[left] {$\tilde{\beta}$} (P0);
\draw(P5) -- (P0);
\draw[dashed, red] (P6) -- (P4);
\end{tikzpicture}
\end{center}

Using these rules in addition to those in \cite{MSW}, we can construct $\widetilde{S_\gamma}$, a $(d+3)$-gon with triangulation $\widetilde{T_\gamma}$ consisting of $d$ internal arcs and $d+3$ boundary arcs. The arc $\widetilde{\gamma} \in \widetilde{S_\gamma}$ crosses all arcs in $\widetilde{T_\gamma}$, and this pattern of crossings resembles the arcs that $\gamma$ crosses in $\mathcal{O}$.

\section{Quadrilateral and bigon lemmas}\label{sec:QuadBigon}

The machinery of our proof that $\phi_\gamma(x_{\widetilde{\gamma}}) = x_{\gamma}$ will be an induction on the number of crossings between $\gamma$ and $T$. To that end, we provide a way to express $x_{\gamma}$ in terms of $x_{\zeta_i}$ where all arcs $\zeta_i$ have less crossings with $T$ than $\gamma$

This was accomplished in \cite{MSW} by Lemma~9.1, known as the quadrilateral lemma. The quadrilateral specified in this lemma gives slightly weaker results when pending arcs are present, but still allows us to prove our expansion formula.

\begin{Lemma}\label{lem:QuadAndBigon}
Let $T$ be a triangulation of an unpunctured orbifold $\mathcal{O}$ and $\gamma$ be a standard arc not in $T$. Then, there exists a quadrilateral $\alpha_1$, $\alpha_2$, $\alpha_3$, $\alpha_4$, of arcs in $\Orb$ such that:
\begin{itemize}\itemsep=0pt
\item $\gamma$ and another arc, $\gamma'$, are the two diagonals of this quadrilateral,
\item $e(\alpha_i, T) \leq e(\gamma, T)$, and
\item $e(\gamma', T) < e(\gamma, T)$.
\end{itemize}

Moreover, if $e(\alpha_i,T) = e(\gamma,T)$ for some $i$, then $\alpha_i$ is a pending arc and ${e(\alpha_j,T) < e(\gamma,T)}$ for all $j \neq i$.

If $\gamma$, instead, is a pending arc, then there exists another pending arc, $\rho$, and a bigon composed of arcs $\beta_1$ and $\beta_2$ such that:
\begin{itemize}\itemsep=0pt
\item $\gamma$ and $\rho$ are the two possible pending arcs contained within the bigon,
\item $e(\beta_i, T) < e(\gamma, T)/2$, and
\item $e(\rho, T) < e(\gamma, T)/2$.
\end{itemize}
\end{Lemma}

Prior to the proof, we need to establish some notation. Let $\gamma_1$ and $\gamma_2$ be two arcs which intersect at a point $b$. This can be an end point of the arcs or not. Let $a$ be another point on~$\gamma_1$ and let $c$ be another point on $\gamma_2$. Then, $(a,b, c\vert \gamma_1,\gamma_2)$ denotes an arc which starts at~$a$, is isotopic to~$\gamma_1$ between $a$ and $b$, is isotopic to $\gamma_2$ between $b$ and $c$, and finally ends at $c$. We can generalize this notion to more arcs which consecutively intersect. We also let $\gamma^-$ denote an arc that is isotopic to $\gamma$ but has opposite orientation.

\begin{proof}We will induct on $e(\gamma, T)$. We have two base cases. If $e(\gamma, T) = 1$, then $\gamma$ must be a~standard arc and is the result of flipping an arc $\tau \in T$, so $\gamma$ is one diagonal in a quadrilateral which is entirely made up of arcs in $T$, and the other diagonal is $\tau$.

The other base case is when $\gamma$ is the result of flipping a pending arc~$\rho$. Then, $e(\gamma,T) = 2$ and~$\gamma$ is a pending arc. It also must be that the arcs, $\beta_1$, $\beta_2$, making up the bigon about~$\gamma$ and~$\rho$ are in $T$ as well.

\looseness=-1 Now, suppose first that $e(\gamma, T) = d$ and $\gamma$ is a standard arc. Label the crossing points between~$\gamma$ and $T$ by $1,2, \ldots, d$. If $\tau$, the arc that crosses $\gamma$ at point $h = \lceil \frac{d}{2} \rceil$, is not a pending arc, then the construction from \cite[Lemma~9.1]{MSW} holds. However, if $\tau$ is a pending arc, and $\gamma$ crosses $\tau$ in spots $j_1,\ldots, j_r$ where $j_\ell = h$, then either the crossing point $h+1$ or $h-1$ is also on~$\tau$.

More explicitly, suppose $j_{\ell} + 1 = j_{\ell + 1} = h+1$ and first let $d$ be even, so that $d = 2h$. Let $s(\gamma)$, $t(\gamma)$ be respectively the start and end of $\gamma$ once we select an orientation. Moreover, suppose that we orient $\tau$, the pending arc containing the intersection points $j_\ell$, $j_{\ell + 1}$ so that it visits $j_\ell$ before $j_{\ell+1}$. Then, assuming that $\ell > 1$, Musiker, Schiffler, and Williams~\cite{MSW} give the following explicit construction for the quadrilateral:
\begin{alignat*}{3}
 & \alpha_1 = \big(s(\gamma), j_{\ell - 1}, j_\ell, s(\gamma) \vert \gamma, \tau, \gamma^-\big), \qquad &&
 \alpha_2 = \big(s(\gamma), j_\ell, j_{\ell + 1}, t(\gamma) \vert \gamma, \tau, \gamma\big), & \\
 & \alpha_3 = \big(t(\gamma), j_{\ell+1}, j_\ell, t(\gamma) \vert \gamma^-, \tau^-, \gamma\big), \qquad &&
 \alpha_4 = \big(t(\gamma), j_\ell, j_{\ell - 1}, s(\gamma) \vert \gamma^-, \tau^-, \gamma^-\big), & \\
 &\gamma' = \big(s(\gamma), j_{\ell - 1}, j_{\ell + 1}, t(\gamma) \vert \gamma, \tau, \gamma\big).\qquad &&&
\end{alignat*}

From these descriptions of $\alpha_i$, we can compute $e(\alpha_i,T)$, and similarly for $\gamma'$. We only highlight a few calculations as the rest are equivalent to the calculations in \cite{MSW}:
\begin{gather*}
 e(\alpha_1,T)= (j_{\ell-1} - 1) + j_\ell < (h-1) + h < d ,\\
 e(\alpha_3,T)= (d - j_{\ell + 1}) + (d - j_{\ell} + 1) = (d - (h+1)) + (d - h + 1) = d.
\end{gather*}

We can see that $\alpha_3$ is a pending arc incident to the same orbifold point as $\tau$. If instead $j_{\ell} - 1 = j_{\ell - 1} = h-1$ and $d$ is still even, then we will find that $e(\alpha_1, T) = e(\gamma,T)$ and $\alpha_1$ will be a pending arc. One can check that $e(\alpha_i,T) < d$ for other $i$ and $e(\gamma',T) < d$ in both these cases.

\begin{center}
\begin{tikzpicture}[scale = 1.3]
\draw(1,4) -- (3,2) -- (1,0) -- (-1,2) -- (1,4);
\node[] at (1,2) {$\times$};
\draw (1,0) to [out = 60, in = 0] (1,2.2);
\draw (1,0) to [out = 120, in = 180] (1,2.2);
\draw (1,0) to [out = 45, in = -45] (1,4);
\draw (1,0) to [out = 135, in = -135] (1,4);
\draw[dashed, red] (-1,2) to [out = -30, in = 210] (3,2); 
\draw[dashed, red] (1,0) to [out = 135, in = 180] (1,2.4);
\draw[dashed, red] (1,2.4) to [out = 0, in = 165] (3,2);
\draw[red] (-1,2) to [out = 30, in = 150] (3,2);
\draw[red] (3,2) -- (1,0) -- (-1,2);
\draw[red] (3,2) to [out = 180, in = 90] (0.7, 1.9);
\draw[red] (3,2) to [out = 200, in = 270] (0.7,1.9);
\end{tikzpicture}
\end{center}

If $d$ is odd, then we will again have that $e(\alpha_i,T) < e(\gamma,T)$ for all $i$ if we follow the recipe for~$\alpha_i$ given in~\cite{MSW}.

Now, let $\gamma$ be a pending arc, and let $\rho \in T$ be the pending arc to the same orbifold point as~$\gamma$. First, note that $d = e(\gamma,T)$ is necessarily even. Let $j$, $j+1$ be the intersections of $\gamma$ and $\rho$. Then, $j = d/2$. Orient $\rho$ so that, like $\gamma$, it passes $j$ before $j + 1$. Define $\beta_1 = (s(\gamma), j_\ell, s(\rho) \vert \gamma, \rho^-)$ and $\beta_2 = (t(\gamma), j_{\ell + 1}, t(\rho) \vert \gamma^-, \rho)$. We can check that all of these arcs cross arcs in $T$ fewer times than $\gamma$:
\begin{itemize}\itemsep=0pt
\item $e(\rho, T) = 0$ as $\rho \in T$,
\item $e(\beta_1, T) = j - 1 < \frac{d}{2}$,
\item $e(\beta_2, T) = k - (j+1) < \frac{d}{2}$.\hfill \qed
\end{itemize}\renewcommand{\qed}{}
\end{proof}

\section[$\widetilde{A_\gamma}$ and $\phi_\gamma$]{$\boldsymbol{\widetilde{A_\gamma}}$ and $\boldsymbol{\phi_\gamma}$} \label{Sec:mapSection}

We first define a map $\pi\colon \widetilde{S_\gamma} \to \mathcal{O}$. Then, we define a morphism, $\phi_\gamma$ between the algebras from these spaces and show it is an algebra homomorphism.

We define $\pi$ from $\{\sigma_1,\ldots, \sigma_{2d +3}\}$ to $\{\tau_1,\ldots, \tau_{n+c}\}$, which will also define $\pi$ on the marked points of each space. Recall that $\tau_{[i_k]}$ is the third side of the triangle formed by $\tau_{i_k}$ and $\tau_{i_{k+1}}$. For completeness, we define $\sigma_a$, $\sigma_b$ to be the two boundary arcs in the first triangle that $\widetilde{\gamma}$ crosses where $\sigma_b$ follows $\sigma_a$ in the clockwise direction. Note that $\widetilde{\gamma}$ inherits an orientation based on the orientation of~$\gamma$. We define $\tau_a$ and $\tau_b$ to be analogous arc in $\mathcal{O}$; note that~$\tau_a$ and~$\tau_b$ are not necessarily on the boundary. Then, we define $\sigma_w$, $\sigma_z$ to be the boundary arcs in the last triangle~$\widetilde{\gamma}$ crosses where $\sigma_z$ follows $\sigma_w$ in the clockwise direction, and define~$\tau_w$ and~$\tau_z$ analogously in $\mathcal{O}$:
\[
\pi(\sigma_j) = \begin{cases} \tau_{i_j}, & 1 \leq j \leq d, \\ \tau_{[\gamma_k]}, & j > d \text{ and } \sigma_j \text{ incident to } \sigma_k \text{ and } \sigma_{k+1}, \\
\tau_{x}, & \sigma_j = \sigma_x \text{ for } x \in \{a,b,w,z\}. \end{cases}
\]

Let $\mathcal{A}$ be the generalized cluster algebra from $\Orb$, as explained in Section~\ref{subsec:genCAfromOrb}. Let $\widetilde{A_\gamma}$ be the cluster algebra corresponding to the polygon $\widetilde{S_\gamma}$ with initial triangulation $\widetilde{T_\gamma} = \{\sigma_1,\ldots, \sigma_d,\sigma_{d+1},\allowbreak \ldots, \sigma_{2d + 3}\}$ where $\sigma_1,\ldots, \sigma_d$ are the arcs in the triangulation and images of the arcs that $\gamma$ crosses in $\Orb$, and $\sigma_{d+1},\ldots, \sigma_{2d + 3}$ are boundary arcs.
In $\widetilde{A_\gamma}$, let $x_{\sigma_i}$ be the variable associated to $\sigma_i$. We treat the variables from boundary arcs, $x_{\sigma_{d+1}},\ldots, x_{\sigma_{2d+3}}$, as coefficients. We also consider $\widetilde{A_\gamma}$ with principal coefficients $\{y_{\sigma_1},\ldots, y_{\sigma_d}\}$; geometrically, we place an elementary (multi)-lamination $\big\{\widetilde{L_1},\ldots, \widetilde{L_d}\big\}$ on $\widetilde{S_\gamma}$ where $\widetilde{L_i}$ is the elementary lamination from $\sigma_i$.
 Let $\mathbb{P} = \text{Trop}(x_{\sigma_{d+1}},\ldots,x_{\sigma_{2d+3}}, y_{\sigma_1},\ldots, y_{\sigma_d})$ be the tropical semifield generated by these elements.

It is clear by construction that $\widetilde{A_\gamma}$ is a type $A_d$, acyclic cluster algebra since the triangulation~$\widetilde{T_\gamma}$ has no internal triangles. Thus, we have the following proposition from Bernstein--Fomin--Zelevinsky.

\begin{Proposition}[{\cite[Corollary~1.21]{BFZ}}] \label{prop:Agamma}The algebra $\widetilde{A_\gamma}$ is the $\mathbb{ZP}$ algebra with set of generators $\{x_{\sigma_1}, \ldots, x_{\sigma_d}, x_{\sigma_1}', \ldots, x_{\sigma_d}'\}$, where $x_{\sigma_k}' = \mu_k(x_{\sigma_k})$, and relations generated by those of the form $x_{\sigma_k}x_{\sigma_k}'$.
\end{Proposition}

We now construct a map, $\phi_\gamma$, from $\widetilde{A_\gamma}$ to $\text{Frac}(\mathcal{A})$. First, we will describe what $\phi_\gamma$ does to the generators of $\widetilde{A_\gamma}$, which we found in Proposition \ref{prop:Agamma}. Then, we will prove that this map is indeed an algebra homomorphism by showing that it sends relations in $\widetilde{A_\gamma}$ to relations in $\mathcal{A}$. We eventually will show that $\phi_\gamma(x_{\widetilde{\gamma}}) = x_\gamma$.

In most cases, we define $\phi_\gamma(x_{\sigma_j}) = x_{\pi(\sigma_j)}$; the exception will be if $\sigma_j = \sigma_{[k]}$ for some $1 \leq k < d$ and $\pi(\sigma_k) = \pi(\sigma_{k+1}) = \rho$ is a pending arc in $\mathcal{O}$. In this case, if the orbifold point incident to~$\sigma_k$ is order~$p$, we set $\phi_\gamma(x_{\sigma_j}) = \lambda_p x_\rho$.
Regardless of whether~$\pi(\sigma_j)$ is a pending arc or standard arc, we set $\phi_\gamma(y_{\sigma_j}) = y_{\pi(\sigma_j)}$.

Next, we need to define the image of $\phi_\gamma$ on the first mutations of the mutable variables in~$\widetilde{A_\gamma}$. If $\pi(\sigma_j) \in T$ is a standard arc, then we set $\phi_\gamma(x_{\sigma_j}') = x_{\pi(\sigma_j)}'$. If $\pi(\sigma_j) = \tau_{i_j}$ is a pending arc in $T$, then either $\pi(\sigma_{j-1}) = \pi(\sigma_j)$ or $\pi(\sigma_{j}) = \pi(\sigma_{j+1})$. Without loss of generality, assume the latter. Let $\delta$ and $\mu$ be the two other arcs in the quadrilateral in $\widetilde{T_\gamma}$ around $\sigma_j$ such that $\delta$ is opposite of~$\sigma_{[j]}$ in this quadrilateral.

 \begin{center}
\begin{tikzpicture}[scale=2]
\node[circle, fill = black, scale = 0.3] (A) at (0,0) {$w$};
\node[circle, fill = black, scale = 0.3] (B) at (2,0) {$v$};
\node[] (E) at (1,0) {$\times$};
\node[](F) at (0.8, 0){};
\node[](G) at (1.2,0.4){};
\node[] (H) at (1.2, -0.4) {};
\draw (A) to [out=60,in=120] node[auto]{}(B);
\draw (A) to [out=-60,in=-120] node[left]{}(B);
\draw(B) to [out = 150, in = 90] node[above, right]{}(0.8, 0);
\draw(B) to [out = 210, in = 270] node[auto]{}(0.8, 0);
\draw[dashed, red](G) to node[right] {$\gamma$} (H);
\node[](arrw) at (2.8,0) {$\rightarrow$};
\node[circle, fill = black, scale = 0.1](P0) at (3,-1){};
\node[circle, fill = black, scale = 0.1](P1) at (4,-1){};
\node[circle, fill = black, scale = 0.1](P2) at (5,-1){};
\node[circle, fill = black, scale = 0.1](P3) at (6,-1){};
\node[](P5) at (5,0.7){};
\draw (P1) -- node[below] {$\sigma_{[j]}$} (P2);
\draw(P5) -- node[left] {$\sigma_j$} (P1);
\draw (P5) -- node[right] {$\sigma_{j+1}$}(P2);
\draw[dashed, red] (4.5,0) -- (5.3,0);
\draw(P1) -- node[left]{$\mu$}(4,0.6);
\draw(5,0.6) -- node[below]{$\delta$}(4,0.6);
\end{tikzpicture}
\end{center}

If $\sigma_{[j]}$ is counterclockwise of $\sigma_j$, as in the diagram above, then we define \[\phi_\gamma(x_{\sigma_j'}) = \lambda_p \cdot \phi_\gamma(x_{\delta}) +\phi_\gamma(y_{\sigma_j})\cdot \phi_\gamma(x_{\mu}) = \lambda_p \cdot x_{\pi(\delta)} + y_{\pi(\sigma_j)} \cdot x_{\pi(\mu)}.
\]

Otherwise, define \[\phi_\gamma(x_{\sigma_j'}) =
\lambda_p\cdot \phi_\gamma(y_{\sigma_j}) \cdot \phi_\gamma(x_{\delta}) + \phi_\gamma(x_{\mu}) = \lambda_p \cdot y_{\pi(\sigma_j)}\cdot x_{\pi(\delta)} + x_{\pi(\mu)}.
\]

\begin{Remark}The expression $\lambda_p y_\rho \alpha + \beta$ is the result when you simplify the self-intersection of the arc below with the skein relation. Compare this with the arc with self-intersection we encounter when proving Proposition \ref{prop:OrbSkeinThreeTerm}.
\end{Remark}

\begin{center}
\begin{tikzpicture}[scale=2]
\node[circle, fill = black, scale = 0.3] (A) at (0,0) {$w$};
\node[circle, fill = black, scale = 0.3] (B) at (2,0) {$v$};
\node[] (E) at (1,0) {$\times$};
\node[](F) at (0.8, 0){};
\node[](G) at (1.2,0){};
\node[](G1) at (1, -0.2){};
\node[](G2) at (1,0.2){};
\draw (A) to [out=60,in=120] node[auto]{$\alpha$}(B);
\draw (A) to [out=-60,in=-120] node[below]{$\beta$}(B);
\draw(B) to [out = 150, in = 90] node[above, right]{$\rho$}(0.8,0);
\draw(B) to [out = 210, in = 270] node[auto]{}(0.8,0);
\draw[red, thick](0,0) to [out = 30, in = 90] node[above]{$\gamma$} (1.2,0);
\draw[red, thick](1.2,0) to [out = 270, in = 0] (1,-0.2);
\draw[red, thick] (1,-0.2) to [out = 180, in = 180] (1,0.2);
\draw[red, thick] (1,0.2) to [out = 0, in = 150] (2,0);
\end{tikzpicture}
\end{center}

\begin{Proposition}
The map $\phi_\gamma$ is an algebra homomorphism; that is, it maps relations in $\widetilde{A_\gamma}$ to relations in $\mathcal{A}$.
\end{Proposition}

\begin{proof}
First, let $\pi(\sigma_j)$ be a standard arc. Then, in $\widetilde{A_\gamma}$, we have a relation \begin{equation}\label{eq:liftPtolemy}
x_{\sigma_j}x_{\sigma_j}' = y_{\sigma_j} \Pi_b x_b + \Pi_c x_c, \end{equation} where $b$ ranges over arcs which are immediately clockwise of $\sigma_j$ in $\widetilde{T_\gamma}$ and $c$ ranges over arcs which are counterclockwise of $\sigma_j$. The image of this relation under $\phi_\gamma$ is \begin{equation}
x_{\pi(\sigma_j)}x_{\pi(\sigma_j)}' = y_{\pi(\sigma_j)}\Pi_b x_{\pi(b)} + \Pi_c x_{\pi(c)}.\end{equation} This is exactly the exchange relation for $x_{\pi(\sigma_j)}$ in $\mathcal{A}$.

\looseness=-1 Now assume $\pi(\sigma_j)$ is a pending arc in $T$, then $x_{\sigma_j}$ has an exchange relation in $\widetilde{A_\gamma}$ akin to equation~\eqref{eq:liftPtolemy}. Using prior notation, in the case where $\sigma_{[j]}$ is clockwise of $\sigma_j$, so that this exchange relation in $\widetilde{A_\gamma}$ is $x_{\sigma_j}' x_{\sigma_j} = y_{\sigma_j} x_\delta x_{\sigma_{[j]}} + x_{\sigma_{j+1}} x_\mu$, we have that $\phi_\gamma$ maps $x_{\sigma_j}' x_{\sigma_j} $ to the following:\[
 ( \lambda_p y_{\pi(\sigma_j)} x_{\pi(\delta)} + x_{\pi(\mu)}) x_{\pi(\sigma_j)}= \lambda_p y_{\pi(\sigma_j)} x_{\pi(\sigma_j)}x_{\pi(\delta)} + x_{\pi(\sigma_j)} x_{\pi(\mu)}.
\]

Moreover, this is equivalent to $\phi_\gamma( y_{\sigma_j} x_\delta x_{\sigma_{[j]}} + x_{\sigma_{j+1}} x_\mu)$ since $\phi_\gamma(x_{\sigma_{[j]}}) = \lambda_p x_{\pi(\sigma_j)}$.
We see a~similar relation when $\delta$ is counterclockwise of~$\sigma_j$. In either case, this is simply an identity in~$\mathcal{A}$. Thus, all relations in $\widetilde{A_\gamma}$ are mapped to relations in~$\mathcal{A}$.
\end{proof}

\begin{Remark} It is reasonable that we send the exchange relation for a pre-image of a pending arc to an identity in $Frac(\mathcal{A})$ since, if both $\sigma_j$ and $\sigma_{j+1}$ correspond to the same arc in $\widetilde{S_\gamma}$, it does not make sense to only mutate one of them.
\end{Remark}

\section[Showing $\phi_\gamma(x_{\widetilde{\gamma}}) = x_\gamma$]{Showing $\boldsymbol{\phi_\gamma(x_{\widetilde{\gamma}}) = x_\gamma}$}\label{sec:mapClusterVar}

In Section \ref{Sec:mapSection}, we defined $\phi_\gamma\colon \widetilde{A_\gamma} \to \mathcal{A}$ on the generators of $\widetilde{A_\gamma}$ and showed that it is in fact an algebra homomorphism. Now, we will show that $\phi_\gamma(x_{\widetilde{\gamma}}) = x_\gamma$. In $\widetilde{S_\gamma}$, we already have expansion formulas thanks to \cite{MSW} (and originally due to~\cite{Musiker-Schiffler}). So, we can import the expansion formula for $x_\gamma$ in $\mathcal{A}$ via our map $\phi_\gamma$.

\begin{Proposition}\label{Proposition5}
Let $\phi_\gamma$ be the map from the last section. Then, $\phi_\gamma(x_{\widetilde{\gamma}}) = x_\gamma$.
\end{Proposition}

\begin{proof}Our proof in the orbifold case will differ from the proof of the analogous result in the surface case \cite[Theorem~10.1]{MSW}, in two ways. First of all, we need to prove this for the case when~$\gamma$ is a pending arc. Then, we need to take account for the case when $\gamma$ is an standard arc and the resulting quadrilateral from the quadrilateral lemma, $\{\alpha_i\}$, is such that $e(\alpha_i,T) = e(\gamma,T)$ for some index $i$. Both of these cases will utilize Lemma~\ref{lem:QuadAndBigon}. We work through these cases simultaneously using induction on~$e(\gamma,T)$.

Let $e(\gamma,T) = d$. If $d = 0$, then $\gamma \in T$, and we already have that $\phi_\gamma(x_{\widetilde{\gamma}}) = x_{\pi(\widetilde{\gamma})} = x_\gamma$. If $d = 1$, then $\gamma$ is a standard arc which crosses one other standard arc, and the statement follows from \cite[Theorem~10.1]{MSW}.

Now, suppose that $d>1$. First, consider the case where $\gamma$ is a pending arc. Then $d$ is necessarily even. Let $\rho \in T$ be the pending arc incident to the same orbifold point as $\gamma$. By Lemma~\ref{lem:QuadAndBigon}, we can find $\beta_1$, $\beta_2$ such that $\beta_1$, $\beta_2$ form the bigon which contains the pending arcs~$\rho$ and~$\gamma$, and $e(\beta_i, T) < \frac{d}{2}$. Suppose the orbifold point incident to~$\gamma$ is order~$p$. Then, in~$\mathcal{A}$, we have that $x_\gamma x_\rho = Y_1 x_{\beta_1}^2 + Y_0 \lambda_p x_{\beta_1} x_{\beta_2} + Y_{-1} x_{\beta_2}^2$, where we can compute $Y_i$ by finding a~sequence of flips from $\rho$ to $\gamma$ and performing the corresponding mutations in the cluster algebra. In Proposition~\ref{prop:ShearCoordMutation}, we will see that we can also compute these $Y_i$ from the orientation of~$\beta_1$ and~$ \beta_2$ and their intersections with the elementary lamination on~$\mathcal{O}$.

We compare this with the scenario in the lift, $\widetilde{S_\gamma}$. Recall $\widetilde{S_\gamma}$ is a polygon triangulated by $\sigma_i$, for $1 \leq i \leq d$. For $j = \frac{d}{2}$, we have that $\pi(\sigma_j) = \pi(\sigma_{j+1}) = \rho$. Moreover, in $\mathcal{O}$, the $\beta_i$ only cross arcs in $\{\pi(\sigma_i)\}_i$, implying that $\widetilde{S_\gamma}$ already contains $\widetilde{S_{\beta_i}}$ and trivially contains $\widetilde{S_\rho}$ as $\rho \in T$. Thus, we can apply $\phi_\gamma$ to $\beta_i$ and $\rho$, as all of these are arcs in the polygon $\widetilde{S_\gamma}$.

 Due to the symmetry of arcs crossed by $\gamma$, there are two lifts of $\beta_1$ and $\beta_2$ to $\widetilde{S_\gamma}$; call them $\beta_{1,i}$, $\beta_{2,i}$ for $i = 1,2$. Moreover, $\beta_{1,i}$, $\beta_{2,i}$, and $\sigma_{[j]}$ form a pentagon in $\widetilde{S_\gamma}$, which is triangulated by $\sigma_j$ and $\sigma_{j+1}$. Let $s_0 = s(\widetilde{\gamma})$ and $s_d = t(\widetilde{\gamma})$ be the start and end of the arc $\widetilde{\gamma}$. Recall we define~$s_j$ to be the vertex shared by $\sigma_j$ and $\sigma_{j+1}$ . Let $a_j$ ($a_{j+1}$) be the other vertex of~$\sigma_j$~($\sigma_{j+1}$). Note that $\pi(s_j) = \pi(a_j) = \pi(a_{j+1})$ since $\pi(\sigma_j) = \pi(\sigma_{j+1})$, and this arc is a pending arc. Then, up to changing indices, $\beta_{1,1}$ connects $s_0$ and $s_j$, and~$\beta_{1,2}$ connects~$a_{j+1}$ and $s_d$. Similarly, $\beta_{2,1}$~connects~$s_0$ and~$a_j$, and $\beta_{2,2}$ connects~$s_j$ and~$s_d$.

\begin{center}
\begin{tikzpicture}[scale=1.8]
\node[circle, fill = black, scale = 0.1](P0) at (3.5,0.7){};
\node[](P1) at (4,-1){};
\node[circle, fill = black, scale = 0.1](P2) at (5,-1){};
\node[circle, fill = black, scale = 0.1](P3) at (6,-1){};
\node[](P6) at (2.8,0){$s_0$};
\node[](P4) at (6.2,0){$s_d$};
\node[] at (3.2,0) {$\cdots$};
\node[] at (5.8,0) {$\cdots$};
\node[](P5) at (5,0.7){$s_j$};
\draw (P1) -- node[below] {$\delta$} (P2) -- node[below]{} (P3) -- node[above]{} (P4) -- (P5)-- node[above]{} (P0) -- (P6) -- (P1);
\draw(P5) -- node[left] {$\sigma_j$} (P1);
\draw (P5) -- node[right] {$\sigma_{j+1}$}(P2);
\draw (P5) -- node[left] {}(P3);
\draw(P1) -- node[left] {} (P0);
\draw(P5) -- (P0);
\draw[dashed, red] (3.5,0) to node[above]{$\tilde{\gamma}$} (5.6,0);
\draw[thick](P6) to node[above]{$\beta_{1,1}$} (P5);
\draw[thick](P2) to node[below]{$\beta_{1,2}$} (P4);
\draw[thick](P6) to node[below]{$\beta_{2,1}$}(P1);
\draw[thick](P5) to node[above]{$\beta_{2,2}$}(P4);
\end{tikzpicture}
\end{center}

Using cluster algebra expansion formulas from triangulated polygons \cite{MSW}, in $\widetilde{A_\gamma}$ we have that $x_{\widetilde{\gamma}}x_{\sigma_j}x_{\sigma_{j+1}} = \widetilde{Y_1} x_{\beta_{1,1}}x_{\beta_{1,2}}x_{\sigma_j} + \widetilde{Y_0} x_{\beta_{1,1}}x_{\beta_{2,2}}x_{\delta} + \widetilde{Y_{-1}}x_{\beta_{2,1}}x_{\beta_{2,2}}x_{\sigma_{j+1}}$. The image of this relation under $\phi_\gamma$ is
\begin{gather}
\phi_\gamma(x_{\widetilde{\gamma}})x_{\rho}^2 = \phi_\gamma(\widetilde{Y_1}) x_{\beta_1}^2x_{\rho} + \phi_\gamma(\widetilde{Y_0}) x_{\beta_1}x_{\beta_2}(\lambda_px_\rho)+ \phi_\gamma(\widetilde{Y_{-1}})x_{\beta_2}^2x_{\rho}\nonumber
\\
\qquad{} \implies\phi_\gamma(x_{\widetilde{\gamma}})x_{\rho} = \phi_\gamma(\widetilde{Y_1}) x_{\beta_1}^2 +\phi_\gamma(\widetilde{Y_0}) \lambda_px_{\beta_1}x_{\beta_2}+ \phi_\gamma(\widetilde{Y_{-1}})x_{\beta_2}^2.\label{genmutate}
\end{gather}

Comparing this with our generalized exchange relation, if we can show that $\phi_\gamma(\widetilde{Y_i}) = Y_i$, we can conclude that $\phi_\gamma(x_{\widetilde{\gamma}}) = x_\gamma$. We postpone this discussion of $y$-variables and laminations to Lemma~\ref{phimapsys} in the next section.

Now, let $\gamma$ be an standard arc in $\mathcal{O}$ with $d = e(\gamma, T)$. Since we are in an orbifold, it may be that the quadrilateral, $\{\alpha_i\}$, which we produce from Lemma~\ref{lem:QuadAndBigon}, has a pending arc $\alpha = \alpha_i$ for some $i$, such that $e(\alpha,T) = d$. In this case, $\widetilde{S_\alpha}$ is not contained in $\widetilde{S_\gamma}$, but we can glue these polygons together as the intersection of arcs crossed by $\alpha$ and $\gamma$ is nonempty. We may also need to glue $\widetilde{S_\gamma'}$ onto this. Details about this gluing may be found in~\cite{MSW}. Denote this glued polygon~$\hat{S}$. The advantage of this larger polygon is a preimage of our quadrilateral $\{\alpha_i\}$ with diagonals $\gamma$, $\gamma'$, lives in~$\hat{S}$. We already showed that $\phi_\alpha(x_{\widetilde{\alpha}}) = x_\alpha$, since $\alpha$ is a pending arc with $e(\alpha,T) = d$. By induction, we also know that $\phi_{\gamma}(x_{\widetilde{\gamma}'}) = x_{\gamma'}$ and for the other~$\alpha_i$, $\phi_{\alpha_i}(x_{\widetilde{\alpha_i}}) = x_{\alpha_i}$.

In $\hat{S}$, by cluster expansion formulas from surfaces, we have the exchange relation $x_{\widetilde{\gamma}} x_{\widetilde{\gamma}'}= \widetilde{Y_+}x_{\widetilde{\alpha_1}}x_{\widetilde{\alpha_3}} + \widetilde{Y_-}x_{\widetilde{\alpha_2}}x_{\widetilde{\alpha_4}}$. The image of this relation under $\phi_\gamma$ is \[
\phi_\gamma(x_{\widetilde{\gamma}}) x_{\gamma'} = \phi_\gamma\big(\widetilde{Y_+}\big)x_{\alpha_1}x_{\alpha_3} + \phi_\gamma\big(\widetilde{Y_-}\big)x_{\alpha_2}x_{\alpha_4}.
\]

Again, we direct our reader to the next section for discussion of laminations on an orbifold and for now assume Lemma~\ref{phimapsys}. By comparing the previous discussion to the Ptolemy relation in~$\mathcal{O}$ applied to the intersection of~$\gamma$ and~$\gamma'$, we conclude that $\phi_\gamma(x_{\widetilde{\gamma}}) = x_\gamma$.
\end{proof}

\section{Laminations on an orbifold}\label{sec:orbifoldLaminations}

We now show that the shear coordinates and elementary laminations for pending arcs defined in Section~\ref{sec:generalizedLaminations} correctly models the mutation of an extended $B$-matrix in a generalized cluster algebra.
Let $L_{i+n}$ be the elementary lamination from arc~$\tau_i \in T$. Recall~$n$ is the number of arcs in the triangulation~$T$.

\begin{Proposition}\label{prop:ShearCoordMutation}
These shear coordinate rules for an orbifold agree with mutation of extended $B$-matrices in the associated generalized cluster algebra.
\end{Proposition}

\begin{proof}First, we show that the shear coordinate associated to a pending arc, $\tau_j$ changes when we flip an standard arc, $\tau_k$, in the same way that the bottom half of the corresponding column (call it column $j$) of the extended
$B$-matrix changes when we mutate at this index,~$k$. The entry~$b_{k,j}$ is positive if and only if $\tau_k$ is counterclockwise of $\tau_j$. For a lamination~$L_i$, with $i > n$, the entry~$b_{ik}$ is positive if and only if $L_i$ intersects the two arcs that are clockwise of $\tau_k$. If both of these situations are true, then $\mu_k(b_{ij})$ will be given by $b_{ij} + b_{ik}b_{kj}$. In a picture, we can see that when we flip $\tau_k$, it will change the bigon around $\tau_j$, so that now $L_i$ will intersect the bigon on the same side twice. This will increase the shear coordinate associated to $L_i$ and $\tau_j$. See picture below, where the shear coordinate $b_{\tau_j}(T, L)$ changes from~0 to~1. We can deal with the case where $b_{ik}$ and $b_{kj}$ are both negative similarly. If these entries are different signs or one is zero, it is clear from pictures that there will be no change to $b_{\tau_j}(T,L)$.

\begin{center}
\begin{tikzpicture}[scale = 2.2]

\tikzset{->-/.style={decoration={
 markings,
 mark=at position #1 with {\arrow{>}}},postaction={decorate}}}

\draw[out=30,in=-30,looseness=1.5] (0,0.3) to (0,1.5);
\draw[out=150,in=-150,looseness=1.5] (0,0.3) to (0,1.5);
\draw(0,0.3) to (1,0.9);
\draw(0,1.5) to (1,0.9);

\draw[in=0,out=45,looseness =1.25] (0,0.3) to (0,1.3);
\draw[in=180,out=135,looseness=1.25] (0,0.3) to (0,1.3);

\draw[fill=black] (0,0.3) circle [radius=1pt];
\draw[fill=black] (0,1.5) circle [radius=1pt];
\draw[fill = black] (1,0.9) circle [radius=1pt];
\draw[thick] (0,1) node {$\mathbf{\times}$};

\node[] at (-0.26,1.4) {};
\node[] at (0.55,0.9) {$\tau_k$};
\node[] at (0,1.36) {$\tau_j$};

\draw[dashed] (-0.5,0.3) to (0.5,1.5);

\draw[out = 10, in = 170, looseness = 0.8, xshift = 3\R] (0, 1.38) to (1,0.9);
\draw[out=140,in=-160,looseness=1.44, xshift = 3\R] (0,0.3) to (0,1.38);
\draw[out=150,in=-150,looseness=1.5, xshift = 3\R] (0,0.3) to (0,1.5);
\draw[xshift = 3\R](0,0.3) to (1,0.9);
\draw[xshift = 3\R] (0,1.5) to (1,0.9);

\draw[xshift = 3\R, in=0,out=45,looseness =1.25] (0,0.3) to (0,1.3);
\draw[xshift = 3\R, in=180,out=135,looseness=1.25] (0,0.3) to (0,1.3);

\draw[xshift = 3\R, fill=black] (0,0.3) circle [radius=1pt];
\draw[xshift = 3\R, fill=black] (0,1.5) circle [radius=1pt];
\draw[xshift = 3\R, fill = black] (1,0.9) circle [radius=1pt];
\draw[xshift = 3\R, thick] (0,1) node {$\mathbf{\times}$};

\node[xshift = 3\R] at (-0.26,1.4) {};

\draw[xshift = 3\R, dashed] (-0.5,0.3) to (0.5,1.5);

\end{tikzpicture}
\end{center}

Next, we want to show that, when we flip a pending arc $\tau_j$, all shear coordinates change according to generalized mutation rules. By set up, it is clear that the shear coordinates associated to that pending arc will flip signs. Recall other entries mutate by $\mu_j(b_{ik}) = b_{ik} + 2 b_{ij} b_{jk}$ if both~$b_{ij}$ and~$b_{jk}$ are positive, $\mu_j(b_{ik}) = b_{ik} - 2 b_{ij} b_{jk}$ if both~$b_{ij}$ and~$b_{jk}$ are negative, and no change otherwise. As before, $b_{jk}$ is positive if and only if $\tau_j$ is counterclockwise of $\tau_k$. For $i > n$, the entry $b_{ij}$ is positive if and only if the lamination $L_i$ intersects the side of the bigon around~$\tau_j$ that is clockwise of~$\tau_j$ as well as~$\tau_j$ itself. Thus, both entries are positive if $L_i$ intersects $\tau_k$ twice, both before and after intersecting~$\tau_j$. If~$\tau_k$ is a pending arc, since we draw this as a loop $L_i$ intersects $\tau_k$ four times, in two pairs. Moreover, the two intersections or pairs of intersections of~$L_i$ and~$\tau_k$ could either both contribute~$-1$, both contribute~0, or one of each contribution. We know that they cannot contribute $+1$ since $L_i$ intersects~$\tau_j$, which is counterclockwise of $\tau_k$. Then, when we flip $\tau_j$, we change the quadrilateral or bigon around~$\tau_k$, depending on whether~$\tau_k$ is standard or pending, which $L_i$ intersects. Thus, we will change the shear coordinate associated to~$\tau_k$ and $L_i$. Because of the two intersections or pairs of intersections, we will change by a multiple of two, as required in the generalized mutation rule.
Fig.~\ref{fig:LamFlipPending} illustrates one example of this situation. Notice that~$b_{6,1}$ changes from~0 to~2.
\end{proof}

\begin{Lemma}\label{phimapsys}
In the language of the previous section, $\phi_\gamma\big(\widetilde{Y_i}\big) = Y_i$.
\end{Lemma}

\begin{proof}Recall the expressions $\phi_\gamma\big(\widetilde{Y_i}\big)$ from equation~\eqref{genmutate}.

First, let $\gamma$ be a pending arc. Then, by the Bigon Lemma, we have a bigon $\beta_1$, $\beta_2$ around $\gamma$ and the arc $\rho \in T$ at the same orbifold point, such that $e(\beta_i, T) < e(\gamma,T)$. We saw that the pre-image of this bigon in $\widetilde{S_\gamma}$ is a pentagon. We want to show that the laminations $L_{\tau_{i_k}}$ contribute the same shear coordinates in the bigon as their pre-images, $L_{\sigma_k}$, contribute in the pentagon in~$\widetilde{S_\gamma}$. However, since~$\rho$ is in the triangulation $T$, and accordingly its images~$\sigma_j$ and~$ \sigma_{j+1}$ in $\widetilde{S_\gamma}$ are in the triangulation $\widetilde{T_\gamma}$, the only elementary laminations that will contribute nontrivially to the relations in either case will be those associated to~$\rho$ in~$\mathcal{O}$, or $\sigma_j$, $\sigma_{j+1}$ in $\widetilde{S_\gamma}$.

In $\widetilde{S_\gamma}$, we have a pentagon with sides $\beta_{1,1}$, $\beta_{1,2}$, the two pre-images of $\beta_1 \in \mathcal{O}$, $\beta_{2,1}$, $\beta_{2,2}$, the two pre-images of $\beta_2$, and $\sigma_{[j]}$, the third arc in the triangle formed by~$\sigma_j$ and $\sigma_{j+1}$. This pentagon is triangulated by~$\sigma_j$ and $\sigma_{j+1}$, and the lift $\widetilde{\gamma}$ is the arc crossing both arcs in this triangulation. By using the skein relations with $y$-variables from~\cite{MW} in $\widetilde{S_\gamma}$ twice, on these two intersections, we have the expansion
\[ x_{\widetilde{\gamma}} = \frac{ y_{\sigma_j} y_{\sigma_{j+1}} x_{\beta_{1,1}}x_{\beta_{1,2}}x_{\sigma_j} + y_{\sigma_{j+1}} x_{\beta_{1,1}}x_{\beta_{2,2}}x_{\sigma_{[j]}} + x_{\beta_{2,1}}x_{\beta_{2,2}}x_{\sigma_{j+1}}}{x_{\sigma_j}x_{\sigma_{j+1}}},
\] and recalling that $\phi_\gamma(y_{\sigma_j}) = \phi_\gamma(y_{\sigma_{j+1}}) = y_\rho$, we see that our map~$\phi_\gamma$ maps the $y$-variables as we hoped.

\begin{center}
\begin{tikzpicture}[scale=2]
\node[circle, fill = black, scale = 0.1](P0) at (3.5,0.7){};
\node[](P1) at (4,-1){};
\node[circle, fill = black, scale = 0.1](P2) at (5,-1){};
\node[circle, fill = black, scale = 0.1](P3) at (6,-1){};
\node[](P6) at (2.8,0){$s_0$};
\node[](P4) at (6.2,0){$s_d$};
\node[] at (3.2,0) {$\cdots$};
\node[] at (5.8,0) {$\cdots$};
\node[](P5) at (5,0.7){$s_j$};
\draw (P1) -- node[below] {$\delta$} (P2) -- node[below]{} (P3) -- node[above]{} (P4) -- (P5)-- node[above]{} (P0) -- (P6) -- (P1);
\draw(P5) -- node[left] {$\sigma_j$} (P1);
\draw[dashed, blue] (5.2,0.7) -- (3.8,-1);
\draw (P5) -- node[right] {$\sigma_{j+1}$}(P2);
\draw[dashed, blue] (5.2,0.7) -- (4.8,-1.1);
\draw (P5) -- node[left] {}(P3);
\draw(P1) -- node[left] {} (P0);
\draw(P5) -- (P0);
\draw[dashed, red] (3.5,0) to node[above]{$\tilde{\gamma}$} (5.6,0);
\draw[thick](P6) to node[above]{$\beta_{1,1}$} (P5);
\draw[thick](P2) to node[below]{$\beta_{1,2}$} (P4);
\draw[thick](P6) to node[below]{$\beta_{2,1}$}(P1);
\draw[thick](P5) to node[above]{$\beta_{2,2}$}(P4);
\end{tikzpicture}
\end{center}

\looseness=-1 Next, let $\gamma$ be an standard arc. From \cite{MSW}, we know that elementary laminations from standard arcs have the same local configuration about $Q$, the quadrilateral corresponding to $\gamma$ and $T$ from the quadrilateral lemma, and $\widetilde{Q}$, the lift of $Q$ in $\widetilde{S_\gamma}$. We need to verify that the same is true for elementary laminations from pending arcs. Suppose that $\rho \in T$ is a pending arc with elementary lamination $L_\rho$, and $\sigma_j, \sigma_{j+1} \in \widetilde{T_\gamma}$ are the pre-images of $\rho$ with elementary laminations~$L_j$,~$L_{j+1}$. In Fig.~\ref{standardcrossespending}, on the left we show one example of intersections of~$\rho$ and~$L_\rho$ with~$Q$, the quadrilateral from applying the quadrilateral lemma to $\gamma$ and $T$. In this case, $b_\gamma(T, L_\rho) = 2$. On the right half we show first the intersections of $\sigma_j$ and $\sigma_{j+1}$ and then the intersections of $L_j$ and $L_{j+1}$ with $\widetilde{Q}$, the lift of $Q$ to $\widetilde{S_\gamma}$. Here, $b_{\widetilde{\gamma}}\big(\widetilde{T_\gamma}, L_j\big) = b_{\widetilde{\gamma}}\big(\widetilde{T_\gamma}, L_{j+1}\big) = 1$. If $y_\rho, y_j,$ and $ y_{j+1}$ are the $y$-variables associated to $L_\rho$, $L_j$ and $L_{j+1}$ respectively, then since $\phi_\gamma(y_j) = \phi_\gamma(y_{j+1}) = y_\rho$, we see that the contribution of laminations is consistent in~$\mathcal{O}$ and~$\widetilde{S_\gamma}$ in this case. The cases $b_\gamma(T, L_\rho) = -2$ and $b_\gamma(T, L_\rho) =0 $ are similar as, again, the local configurations around $\gamma$ and $\widetilde{\gamma}$ look the same.
\end{proof}

\begin{figure}\label{standardcrossespending}\centering
\begin{tikzpicture}[scale = 1.3]
\draw (-1,0) -- (0,1) -- (1,0) -- (0, -1) -- (-1,0);
\draw (0,1) -- (0,-1);
\node[] at (0.8,0.8) {$\times$};
\draw(-0.7, -0.7) to [out = 75, in = 120] (0.9, 0.9);
\draw(-0.7, -0.7) to [out = 30, in = 300] (0.9, 0.9);
\draw (2, 0) -- (3,1) -- (4, 0) -- (3, -1) -- (2,0);
\draw(3,1) -- (3,-1);
\node[] at (3.8,0.8) {$\times$};
\draw[dashed](2.3,-0.7) to [out = 75, in = 120] (3.9, 0.9);
\draw[dashed](2.3, -0.7) to [out = 30, in = 300] (3.9, 0.9);
\draw (5,0) -- (6,1) -- (7,0) -- (6, -1) -- (5,0);
\draw (6,1) -- (6,-1);
\draw(5.3, -0.7) to (7,0.7);
\draw(5.3, -0.7) to (6.7, 1);
\draw (8, 0) -- (9,1) -- (10, 0) -- (9, -1) -- (8,0);
\draw(9,1) -- (9,-1);
\draw[dashed](8.3, -0.7) to (10,0.7);
\draw[dashed] (8.3, -0.7) to (9.7, 1);
\end{tikzpicture}
\caption{From left to right: A standard arc and a pending arc crossing, the elementary lamination from a pending arc and a standard arc crossing, and the lifts of these two scenarios to $\widetilde{S_\gamma}$.}
\end{figure}
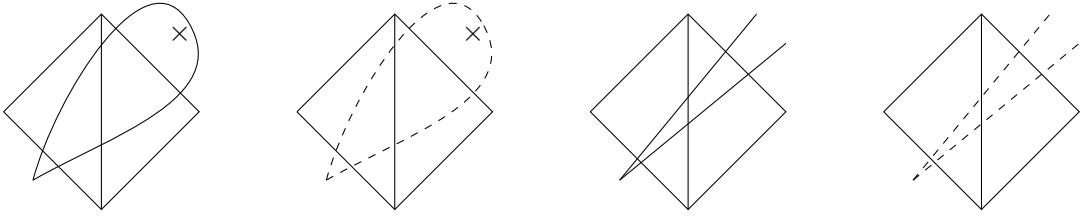

\section{Proof of cluster expansion formula}\label{sec:ClusterExpansionPf}

With the proof of Lemma \ref{phimapsys}, we are ready to complete our proof of Theorem~\ref{Thm:GenSnake+}.

\begin{proof}\looseness=-1 In the statement of Theorem~\ref{Thm:GenSnake+}, we have a fixed orbifold $\mathcal{O} = (S,M,Q)$ with triangulation $T = \{\tau_1,\dots,\tau_n,\tau_{n+1},\dots,\tau_{n+c} \}$ where $\tau_1,\dots,\tau_n$ are internal arcs and $\tau_{n+1},\dots,\tau_{n+c}$ are boundary arcs. This determines the corresponding generalized cluster algebra $\mathcal{A}$ with principal coefficients with respect to the initial generalized seed $\Sigma_T = (\mathbf{x}_T,\mathbf{y}_T,B_T,\mathbf{z})$. For a given arc~$\gamma$ on~$\mathcal{O}$, we defined the lifted triangulated polygon $\widetilde{S_{\gamma}}$, the lifted arc $\widetilde{\gamma}$, and lifted triangulation $\widetilde{T} = \{\sigma_1, \dots, \sigma_d, \sigma_{d+1}, \dots, \sigma_{2d+3} \}$ where $\sigma_1,\dots,\sigma_d$ are internal arcs and $\sigma_{d+1},\dots,\sigma_{2d+3}$ are boundary arcs. The lift $\widetilde{S_\gamma}$ has an associated type $A_d$ ordinary cluster algebra, $\widetilde{A_\gamma}$, where $d = e(\gamma,T)$. We then defined a projection map $\pi\colon \{ \sigma_1,\dots,\sigma_{2d+3} \} \rightarrow \{ \tau_1,\dots,\tau_{n+c} \}$, which in turn allowed us to define an algebra homomorphism $\phi_\gamma\colon \widetilde{A_\gamma} \rightarrow {\rm Frac}(\mathcal{A})$; in general, $\phi_\gamma$~acts by $\phi_\gamma(x_{\sigma_j}) = x_{\pi(\sigma_{j})}$ and $\phi_\gamma(y_{\sigma_j}) = y_{\pi(\sigma_{j})}$ for all $\sigma_j \in \{\sigma_1, \dots, \sigma_{2d+3} \}$. We noted that when~$\gamma$ crosses one or multiple pending arc(s), $\phi_\gamma$ will map some variables associated to boundary arcs in~$\widetilde{S_\gamma}$ to constant multiples of the variables associated to these pending arcs in~$\mathcal{O}$. These multiples are determined by the orders of orbifold points.
Further, we proved in Proposition~\ref{Proposition5} that $\phi_\gamma(x_{\widetilde{\gamma}}) = x_{\gamma}$.

Because $\widetilde{A_\gamma}$ is a type $A_d$ ordinary cluster algebra, we know from the work of Musiker, Schiffler, and Williams~\cite{MSW} that we can build a snake graph $G_{\widetilde{T},\widetilde{\gamma}}$ which has the cluster expansion for~$x_{\widetilde{\gamma}}$ as the generating function for its perfect matchings. This cluster expansion for $x_{\widetilde{\gamma}}$ is in terms of the variables $x_{\sigma_1}, \dots, x_{\sigma_{2d+3}}$ and $y_{\sigma_1},\dots, y_{\sigma_{d}}$. Hence, computing the cluster expansion for~$x_\gamma$ in~$\Sigma_T$ is equivalent to specializing the variables in the generating function for perfect matchings of~$G_{\widetilde{T},\widetilde{\gamma}}$ using the homomorphism $\phi_\gamma$.

By construction, the unlabeled graphs for $G_{\widetilde{T},\widetilde{\gamma}}$ and $G_{T,\gamma}$ are identical. Because $\phi_\gamma(x_{\sigma_j}) = x_{\pi(\sigma_j)}$,
applying $\phi_\gamma$ sends most edges labeled $\sigma_j$ in $G_{\widetilde{T},\widetilde{\gamma}}$ to edges labeled $\pi(\sigma_j)$ in $G_{T,\gamma}$. Similarly, diagonals labeled $y_{\sigma_j}$ are sent to diagonals labeled $y_{\pi(\sigma_j)}$. Hence, applying $\phi_\gamma$ to the generating function for perfect matchings of $G_{\widetilde{T},\widetilde{\gamma}}$ yields the formula in the theorem statement, which is itself the generating function for perfect matchings of $G_{T,\gamma}$, as desired.
 \end{proof}

Now we have an expansion formula for arcs without self-intersections in an unpunctured orbifold $\mathcal{O}$. These correspond to cluster variables in the associated generalized cluster algebra~$\mathcal{A}$. Arcs with self-intersections, i.e., generalized arcs~-- and closed curves do not correspond to cluster variables as they can never appear in a triangulation of $\mathcal{O}$. However, we can still use the rules in Section~\ref{Sec:Construction} to construct snake graphs from these arcs and curves. By applying the expansion formula to these snake graphs, we associate an element of~$\mathcal{A}$ to each generalized arc and closed curve. In the following sections, we will show that this association has desirable properties.

In order to study these arcs and curves, we will associate each with a product of $2\times 2$ matrices, based on breaking up the path of the arc/curve into a sequence of ``elementary steps''. We can use another set of $2\times 2$ matrices to help us compute weighted perfect matchings of graphs. We will show that these two sets of matrices are related. With these connections between arcs/curves, graphs, and matrices, we will be able to investigate properties of one object by studying another. In particular, we will use our matrix formulation to show that our expansion formula for generalized arcs and closed curves respects the skein relations. This work follows closely the work of Musiker--Williams~\cite{MW}, which does these same calculations in the case of a~cluster algebra from a~surface.

\section{Universal snake graph}\label{sec:universalSnakeGraph}

In \cite{MW}, Musiker and Williams compared their snake graph formulas to formulas arising from multiplying together strings of $2 \times 2$ matrices. These $2 \times 2$ matrices came in two types, depending on whether the matrix corresponds to adding a tile to the east or north of a snake graph. We simplify the calculations and arguments of~\cite{MW} by using \emph{universal tiles} to build \emph{universal snake graphs}. Accordingly, we use only one type of $2 \times 2$ matrix which includes both types in~\cite{MW} as specializations. We will similarly see that the universal snake graph is made up of a combination of the pieces used to build standard snake graphs.

For any positive integer $n$, the $n$-tile universal snake graph, $UG_n$, encodes information about the perfect matchings of all $n$ tile ordinary snake graphs, as well as those with extra diagonals that we encounter in the orbifold setting. We will make this statement make more precise. Below is the universal snake graph with~4 tiles~$UG_4$. The horizontal edges are labeled with $a_j$ and the long diagonal edges, which are solid, are labeled with~$b_j$. The dashed lines, labeled~$i_j$, serve as labels for the individual tiles and cannot be used in a perfect mat\-ching.

\begin{center}
\begin{tikzpicture}[scale = 1.85]
\draw[thick] (0,0) to node[left]{$b$} (0,1) to node[left]{$\ell_1$} (0,2) to node[left]{$\ell_2$} (0,3)to node[left]{$\ell_3$} (0,4) to node[above]{$z'$} (1,4) to node[right]{$w'$} (1,3) to node[right]{$r_3$} (1,2) to node[right]{$r_2$} (1,1) to node[right]{$r_1$} (1,0) to node[below]{$a$} (0,0);
\draw[thick, blue] (0,1) to node[right, below, xshift = 15pt]{$a_1$} (1,1);
\draw[thick, blue] (0,2) to node[right, below, xshift = 15pt]{$a_2$} (1,2);
\draw[thick, blue] (0,3) to node[right, below, xshift = 15pt]{$a_3$} (1,3);
\draw[thick, green] (1,0) to node[above, xshift = -15pt, yshift = 10pt]{$b_1$} (0,2);
\draw[thick, green] (1,1) to node[above, xshift = -15pt, yshift = 10pt]{$b_2$} (0,3);
\draw[thick, green] (1,2) to node[above, xshift = -15pt, yshift = 10pt]{$b_3$} (0,4);
\draw[dashed, gray] (0,1) to node[]{$i_1$} (1,0);
\draw[dashed, gray] (0,2) to node[]{$i_2$} (1,1);
\draw[dashed, gray] (0,3) to node[]{$i_3$} (1,2);
\draw[dashed, gray] (0,4) to node[]{$i_4$} (1,3);
\end{tikzpicture}
\end{center}

Note that we can glue $a$ or $b$ to $w'$ or $z'$ to obtain a \emph{universal band graph}. Good matchings of universal band graphs are defined analogously to good matchings of standard band graphs.

If $n$ is even, let $w' = w$ and $z' = z$. Otherwise, $w' = z$ and $z' = w$. As a heuristic, we label the last tile so that the matching of all boundary edges that uses edge $a$ must also include $w$. We call this the \emph{minimal matching} to be consistent with the standard snake graph case. The other matching consisting of only boundary edges will include edges $b$ and $z$, and we call this the \emph{maximal matching}.

We note that we can recover any snake graph we are interested in, as well as others, from the universal snake graph of the appropriate size.
\begin{itemize}\itemsep=0pt
\item Specializing $a_j = 0$ or $b_j = 0$ at each $j$ will recover an ordinary snake graph. Based on the correspondence between snake graphs and sign sequences noted in \cite{Canakci-Schiffler}, we know that there are $2^{n-1}$ snake graphs with $n$ tiles. This is also the number of ways to choose whether $a_j = 0$ or $b_j = 0$ for $j = 1, \ldots, n-1$.
\item If we do not set $a_j = 0$ or $b_j = 0$ at some $j$, but $a_{j-1}b_{j-1} = 0$ and $a_{j+1}b_{j+1} = 0$, we recover a hexagonal tile as in Section \ref{Sec:Construction}.
\item We do not have a geometric interpretation of a graph where $a_jb_j \neq 0$ and $a_{j+1}b_{j+1} \neq 0$, or a graph where $a_j = b_j = 0$.
\end{itemize}

\begin{Remark}
We can think about the universal snake graph $UG_n$ as constructed of two initial triangles and $n-1$ parallelograms with crossing diagonals,
\begin{center}
\begin{tikzpicture}[scale = 1.85]
\draw (1,0) to node[below]{$a$} (0,0) to node[left]{$b$}(0,1);
\draw[dashed, gray] (1,0) to node[above]{$i_1$} (0,1);
\draw[xshift = 4\R, dashed] (1,0) to node[below]{$i_j$} (0,1);
\draw[xshift = 4\R, dashed] (1,1) to node[above]{$i_{j+1}$} (0,2);
\draw[xshift = 4\R] (0,1) to node[left]{$\ell_j$} (0,2) to node[left, yshift = 20pt, xshift = -6pt]{$b_j$} (1,0) to node[right]{$r_j$} (1,1) to node[below, xshift = 20pt]{$a_j$} (0,1);
\draw[xshift = 8\R] (0,2) to node[above]{$z'$} (1,2) to node[right]{$w'$} (1,1);
\draw[xshift = 8\R, dashed] (1,1) to node[below]{$i_n$} (0,2);
\end{tikzpicture}
\end{center}

These parallelograms are essentially a superposition of the north-pointing and east-pointing parallelograms in \cite{MW}. If $b_j = 0$, the parallelogram is genuinely north-facing, and if $a_j = 0$, it is east-facing.
\end{Remark}

\looseness=-1 We also verify some simple properties about this graph and its perfect matchings. First, we explain how to extend the definition of a \emph{twist} to the more complicated tiles in $UG_n$. As in the case of ordinary snake graphs (Theorem~\ref{Thm:PMPoset}), twisting induces a poset structure on the set of perfect matchings of $UG_n$. In Lemma~\ref{PMUniversalSnake}, we see that this poset structure has a simple description.

\begin{figure}[t] \centering
 \includegraphics[scale=0.93]{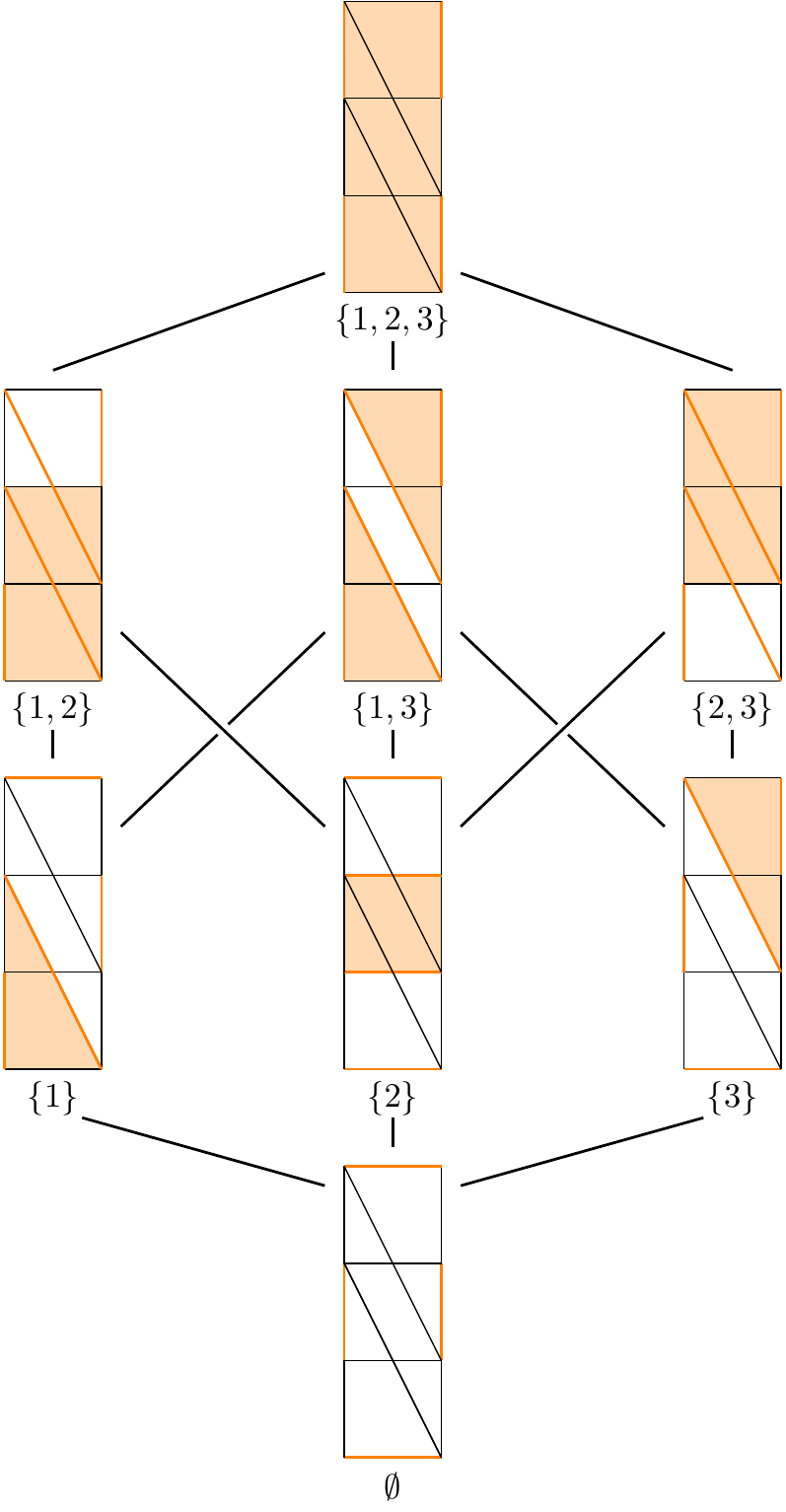}
 \caption{A Hasse diagram showing the poset of perfect matchings of $UG_3$, ranked by height monomial. For each perfect matching, the enclosed tiles are shaded in orange.}
\end{figure}

If a perfect matching uses edges $\ell_{j-1}$ (set $\ell_0 = b$) and $r_{j}$ (set $r_n = w'$) for $1 \leq j \leq n$, we twisting at tile $j$ is accomplished by replacing those edges with the edges $a_{j-1}$ (set $a_0 = a$) and $a_{j}$ (set $a_{n} = z'$). This twist results in another valid perfect matching of $UG_n$. If a perfect matching instead uses edges $\ell_{j}$ (set $\ell_n = z'$) and $r_{j-1}$ (set $r_0 = a$), for $1 \leq j \leq n$, then twisting at tile $j$ is accomplished by replacing those edges with the edges $b_{j-1}$ (set $b_0 = b$) and $b_j$ (set $b_n = w'$). Both types of local move are referred to as a twist at tile $j$.

The poset of perfect matchings of $UG_n$ has some of the same basic properties as the ordinary case described in Section~\ref{section:snakeGraphs} -- that is, the covering relation is given by a twist at single tiles, and the poset rank function is given by the degree of the associated height monomials. As before, the height monomial for a given perfect matching $P$ can be determined by viewing the labels of tiles enclosed by cycles in the symmetric difference $P \ominus P_-$. Note that we consider a tile to be ``enclosed'' by a cycle if the dashed line marking the tile is inside the cycle.

\begin{Lemma}\label{PMUniversalSnake}\quad
 \begin{enumerate}\itemsep=0pt
\item[$1.$] $UG_n$ has $2^n$ perfect matchings.
\item[$2.$] The poset of perfect matchings of $UG_n$ is isomorphic to the poset of subsets of $\{1,\ldots, n\}$ ordered by inclusion, $B_n$. This isomorphism sends a subset $\{i_1,\ldots, i_k\}$ to the matching with weight $y_{i_1}\cdots y_{i_k}$.
\end{enumerate}
\end{Lemma}

\begin{proof}Part 1 is implied by part~2, so we only prove part 2. This is done by induction. It is clear that the claim holds for $UG_1$, as this snake graph is a single tile with only a minimal and maximal matching. The maximal matching covers the minimal matching in the corresponding poset.

Now, suppose our claim holds for $UG_{k-1}$, and consider the poset of perfect matchings of~$UG_k$. This contains a subposet of all matchings using the edge $w$; the minimal matching is in this subposet. Such matchings cannot use edges $z$ and either cannot use $\ell_{k-1}$ or $r_{k-1}$, depending on the parity of $k$. If we remove $z$, $w$ and either $r_{k-1}$ or $\ell_{k-1}$ from~$UG_k$, we have a graph isomorphic to $UG_{k-1}$; hence, the subposet of matchings using $w$ is isomorphic to the poset of perfect matchings of $UG_{k-1}$.

The remaining elements of $UG_{k}$ necessarily use~$z$. The minimal element of this subposet is the perfect matching obtained by twisting the minimal matching at tile $i_k$. For the same reasons as for the matchings using~$w$, this subposet is isomorphic to the poset of perfect matchings of~$UG_{k-1}$.

\looseness=2 Since the poset corresponding to perfect matchings of $UG_{k-1}$ is isomorphic to~$B_{k-1}$, and~$UG_k$ consists of exactly two disjoint subposets isomorphic to $UG_{k-1}$ in the way described, we have that $UG_k$ is isomorphic to $B_k$. Following our same induction, we can show the second statement of part~2; the subposet of matchings using $w$ corresponds to subsets of $\{1,\ldots, k\}$ which do not include $k$ while the subposet of matchings using~$z$ corresponds to subsets which do include~$k$.
\end{proof}

Along with the $y_i$ variables from the poset structure of perfect matchings on $UG_n$, for each edge, $\eta$ in the graph, we associate a formal variable $x_\eta$. Of course, when these graphs come from a surface, these variables will be cluster variables. We associate a product of matrices to $UG_n$ for each $n$. These products will encode all weighted perfect matchings of~$UG_n$. Let~$MG_1 = \left[\begin{smallmatrix} 1 &0 \\ 0 & 1 \end{smallmatrix}\right]$. Then, for $n \geq 2$,
\[
MG_n := m_{n-1}\cdots m_1 = \cdots \begin{bmatrix} \dfrac{x_{\ell_3}}{x_{i_3}} & y_3 x_{b_3} \vspace{1mm}\\ \dfrac{x_{a_3}}{x_{i_3}x_{i_4}} & y_3 \dfrac{x_{r_3}}{x_{i_4}} \end{bmatrix} \begin{bmatrix}\dfrac{x_{r_2}}{x_{i_2}} & y_2 x_{a_2} \vspace{1mm}\\ \dfrac{x_{b_2}}{x_{i_2}x_{i_3}} & y_2 \dfrac{x_{\ell_2}}{x_{i_3}} \end{bmatrix} \begin{bmatrix} \dfrac{x_{\ell_1}}{x_{i_1}} & y_1 x_{b_1} \vspace{1mm}\\ \dfrac{x_{a_1}}{x_{i_1}x_{i_2}} & y_1 \dfrac{x_{r_1}}{x_{i_2}} \end{bmatrix},
\]
where the last terms depend on the parity of $n$. Explicitly,
\begin{gather*}
 m_j := \begin{cases}
 \begin{bmatrix} \dfrac{x_{\ell_j}}{x_{i_j}} & y_jx_{b_j} \vspace{1mm}\\ \dfrac{x_{a_j}}{x_{i_j}x_{i_{j+1}}} & y_j\dfrac{x_{r_j}}{x_{i_{j+1}}} \end{bmatrix} & \text{for odd } j, \vspace{2mm} \\
\begin{bmatrix} \dfrac{x_{r_j}}{x_{i_j}} & y_jx_{a_j} \vspace{1mm}\\ \dfrac{x_{b_j}}{x_{i_j}x_{i_{j+1}}} & y_j\dfrac{x_{\ell_j}}{x_{i_{j+1}}} \end{bmatrix} & \text{for even }j.
 \end{cases}
\end{gather*}

We show that the graphs $UG_n$ and the matrices $MG_n$ satisfy the same relationship as Proposition~5.5 of~\cite{MW}.

\begin{Proposition}\label{prop:MatchingMatrices}
The matrix $MG_n$ is given by $MG_n = \left[\begin{smallmatrix} A_n & B_n \\ C_n & D_n \end{smallmatrix}\right]$ where
\begin{alignat*}{3}
& A_n = \frac{\sum_{P \in S_A} x(P) y(P)}{(x_{i_1}\cdots x_{i_{n-1}}) x_ax_w}, \qquad && B_n = \frac{\sum_{P \in S_B} x(P) y(P)}{(x_{i_2}\cdots x_{i_{n-1}}) x_bx_w},&
\vspace{1mm}\\
& C_n = \frac{\sum_{P \in S_C} x(P) y(P)}{(x_{i_1}\cdots x_{i_{n}}) x_ax_z y_{i_n}}, \qquad && D_n = \frac{\sum_{P \in S_D} x(P) y(P)}{(x_{i_2}\cdots x_{i_{n}}) x_bx_z y_{i_n}}, &
\end{alignat*}
where $S_A$ is the set of matchings using $a$ and $w$ $($this includes the minimal matching$)$, $S_B$ is the set of matchings using~$b$ and~$w$, $S_C$ is the set of matchings using $a$ and $z$, and $S_D$ is the set of matchings using $b$ and $z$ $($this includes the maximal matching$)$.
\end{Proposition}

\begin{proof}\looseness=-1 The proof proceeds by induction. The statement clearly holds for $n = 1$ or $n = 2$. Now, suppose it holds for $n-1$, and consider the graph~$UG_n$. Suppose that~$n$ is even. Then, we have that
\begin{gather*}
MG_n = \begin{bmatrix}\dfrac{x_{\ell_{n-1}}}{x_{i_{n-1}}} & y_{n-1} x_{b_{n-1}} \vspace{2mm}\\ \dfrac{x_{a_{n-1}}}{x_{i_{n-1}}x_{i_n}} & y_{n-1} \dfrac{x_{r_{n-1}}}{x_{i_n}} \end{bmatrix} \begin{bmatrix} \dfrac{x_{r_{n-2}}}{x_{i_{n-2}}} & y_{n-2} x_{a_{n-2}} \vspace{2mm}\\ \dfrac{x_{b_{n-2}}}{x_{i_{n-2}}x_{i_{n-1}}} & y_{n-2} \dfrac{x_{\ell_{n-2}}}{x_{i_{n-1}}} \end{bmatrix} \cdots\\
\hphantom{MG_n=}{}\times
 \begin{bmatrix}\dfrac{x_{r_2}}{x_{i_2}} & y_2 x_{a_2} \vspace{2mm}\\ \dfrac{x_{b_2}}{x_{i_2}x_{i_3}} & y_2 \dfrac{x_{\ell_2}}{x_{i_3}} \end{bmatrix} \begin{bmatrix} \dfrac{x_{\ell_1}}{x_{i_1}} & y_1 x_{b_1} \vspace{1mm}\\ \dfrac{x_{a_1}}{x_{i_1}x_{i_2}} & y_1 \dfrac{x_{r_1}}{x_{i_2}} \end{bmatrix}
\\
\hphantom{MG_n}{}
= \begin{bmatrix}\dfrac{x_{\ell_{n-1}}}{x_{i_{n-1}}} & y_{n-1} x_{b_{n-1}} \vspace{2mm}\\ \dfrac{x_{a_{n-1}}}{x_{i_{n-1}}x_{i_n}} & y_{n-1} \dfrac{x_{r_{n-1}}}{x_{i_n}} \end{bmatrix} M_{n-1} = \begin{bmatrix}\dfrac{x_{\ell_{n-1}}}{x_{i_{n-1}}} & y_{n-1} x_{b_{n-1}} \vspace{2mm} \\ \dfrac{x_{a_{n-1}}}{x_{i_{n-1}}x_{i_n}} & y_{n-1} \dfrac{x_{r_{n-1}}}{x_{i_n}} \end{bmatrix} \begin{bmatrix} A_{n-1} & B_{n-1} \\ C_{n-1} & D_{n-1} \end{bmatrix}
\\
\hphantom{MG_n}{}
= \begin{bmatrix} \dfrac{x_{\ell_{n-1}}}{x_{i_{n-1}}} A_{n-1} + y_{n-1} x_{b_{n-1}} C_{n-1} & \dfrac{x_{\ell_{n-1}}}{x_{i_{n-1}}} B_{n-1} + y_{n-1} x_{b_{n-1}} D_{n-1} \vspace{2mm}\\ \dfrac{x_{a_{n-1}}}{x_{i_{n-1}}x_{i_n}} A_{n-1} + y_{n-1} \dfrac{x_{r_{n-1}}}{x_{i_n}} C_{n-1} & \dfrac{x_{a_{n-1}}}{x_{i_{n-1}}x_{i_n}} B_{n-1} + y_{n-1} \dfrac{x_{r_{n-1}}}{x_{i_n}} D_{n-1} \end{bmatrix}.
\end{gather*}
Consider the subgraph consisting of tiles $i_1,\ldots, i_{n-1}$ as the graph $UG_{n-1}$. Since $n$ is even, the edge that would be labeled~$w$ in this embedded copy of $UG_{n-1}$ (which we call $w_{n-1}$) is labeled~$a_{n-1}$ in $UG_n$. Similarly, the edge labeled $r_{n-1}$ in $UG_n$ would be labeled $z$ in $UG_{n-1}$ and so we call this edge~$z_{n-1}$.

\looseness=2 Let $S_A'$, $S_B'$, $S_C'$, $S_D'$ be the sets of matchings satisfying the description in the proposition for the specified subgraph $UG_{n-1}$. Then, we have that all matchings in $S_A$ correspond either to a~matching in $S_A'$ or in $S_C'$ via the following correspondence. Matchings in $S_A$ use both~$a$ and~$w$; because~$n$ is even, they must also use either $\ell_{n-1}$ or $b_{n-1}$. If one of these matchings uses~$\ell_{n-1}$, it uniquely corresponds to a matching of~$UG_{n-1}$ which uses $a_{n-1} = w_{n-1}$; such a~matching belongs to~$S_A'$. If it uses $b_{n-1}$, then it uniquely corresponds to a matching of $UG_{n-1}$ which uses $r_{n-1} = z_{n-1}$; this matching of $UG_{n-1}$ belongs to~$S_{C}'$.

\begin{center}
 \includegraphics[scale=1]{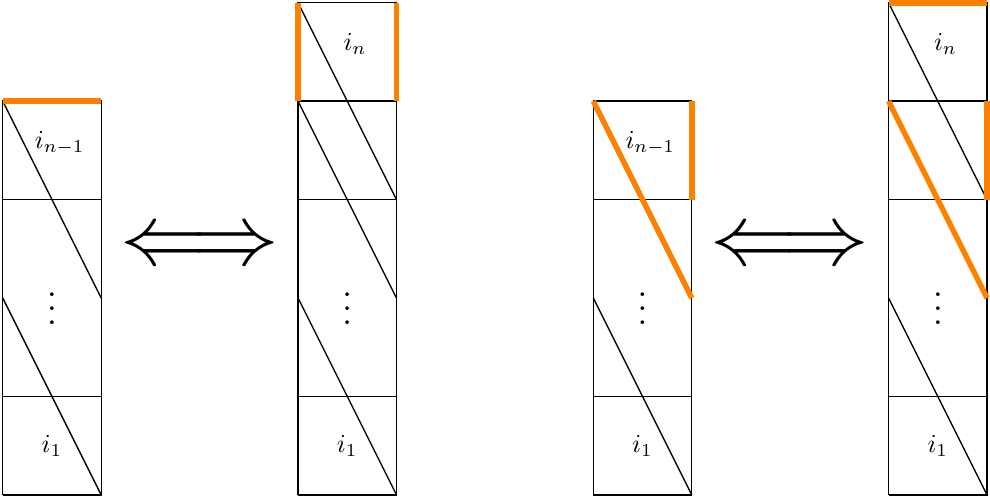}
\end{center}

We can then consider the weights of each matching. If a matching uses the edge $\ell_{n-1}$, then its symmetric difference with the minimal matching of $UG_n$ cannot enclose the tile labeled~$i_{n}$. Hence, its weight must be equal to the weight of the corresponding perfect matching in $S_A'$. If the matching instead uses the edge $b_{n-1}$, then its symmetric difference with the minimal matching must enclose the tile labeled $i_n$, so its weight is given by $y_{i_{n+1}}\cdot (\text{weight of corresponding matching}\allowbreak\text{from }S_C')$. Therefore, the set of matchings in $S_A$ satisfies the relationship
\[ A_n = \frac{x_{\ell_{n-1}}}{x_{i_{n-1}}} A_{n-1} + y_{n-1} x_{b_{n-1}} C_{n-1}.\] The remaining arguments for the other matrix entries and the case where $n$ is odd are very similar, so are not reproduced here.
\end{proof}

By considering several specializations, we can apply Proposition \ref{prop:MatchingMatrices} to band graphs. Note that, while abstractly we can glue $a$ or $b$ to $w$ or $z$ to form a band graph, in order to get a graph which would come from a surface we must either glue $a$ to $z$ or $b$ to $w$. In the first case, if the graph is from a closed curve on a surface, then we would also have $b = i_n$ and $w = i_1$. If we glue~$b$ to~$w$, then $a = i_n$ and $z = i_1$. See Fig.~5 in~\cite{MW}.

\begin{Theorem}\label{thm:upperright}
Let $UG_n$ be a universal snake graph on $n$ tiles. Then, we can express its sum of weighted perfect matchings by{\samepage \[
\sum_P x(P) h(P) = x_{i_1} \cdots x_{i_n} \operatorname{ur} \left( \begin{bmatrix} \dfrac{x_w}{x_{i_n}} & x_z y_{i_n} \vspace{1mm}\\ \dfrac{-1}{x_z} & 0 \end{bmatrix} M_n \begin{bmatrix} 0 & x_a \\ \dfrac{-1}{x_a} & \dfrac{x_b}{x_{i_1}} \end{bmatrix}\right),
\]
where $\emph{ur}$ returns the upper right entry of a matrix.}

Now, let $G$ be the result of gluing $a$ and $z$ in $UG_n$, and setting $b = i_n$ and $w = i_1$. Then, we can express its sum of weighted perfect matchings by\[
\sum_P x(P) h(P) = x_{i_1} \cdots x_{i_n} \operatorname{tr} \left( \begin{bmatrix} \dfrac{x_{i_1}}{x_{i_n}} & x_a y_{i_n} \\ 0 & \dfrac{y_{i_n} x_{i_n}}{x_{i_1}} \end{bmatrix} M_n \right).
\]

Similarly, if $G$ is the result of gluing $b$ and $w$ in $UG_n$ and setting $a = i_n$ and $z = i_1$, then
\[
\sum_P x(P) h(P) = x_{i_1} \cdots x_{i_n} \operatorname{tr} \left( \begin{bmatrix} \dfrac{x_{i_1}}{x_{i_n}} & 0 \\ \dfrac{x_b}{x_{i_1}x_{i_n}} & \frac{y_{i_n} x_{i_n}}{x_{i_1}} \end{bmatrix} M_n \right).
\]
\end{Theorem}

\begin{proof}For the case of $UG_n$, by using $A_n$, $B_n$, $C_n$, $D_n$ as in Proposition~\ref{prop:MatchingMatrices}, we find that \begin{gather}
\operatorname{ur} \left( \begin{bmatrix} \dfrac{x_w}{x_{i_n}} & x_z y_{i_n} \vspace{1mm}\\ \dfrac{-1}{x_z} & 0 \end{bmatrix} M_n \begin{bmatrix} 0 & x_a \vspace{1mm}\\ \dfrac{-1}{x_a} & \dfrac{x_b}{x_{i_1}} \end{bmatrix}\right)\nonumber\\
\qquad{} = \frac{x_a x_w}{x_{i_n}} A_n + \frac{x_b x_w}{x_{i_1}x_{i_n}} B_n + x_a x_z y_{i_n} C_n + \frac{x_b x_z y_{i_n}}{x_{i_n}} D_n.\label{eq:ExpandUR}
\end{gather}

From the definition of $A_n$, we see that
\[ \frac{x_a x_w}{x_{i_n}} A_n = \frac{\sum_{P \in S_A} x(P)h(P)}{x_{i_1}\cdots x_{i_n}}.\] A very similar statement is true for the terms of $B_n$, $C_n$, and $D_n$. Since the sets $S_A$, $S_B$, $S_C$, and $ S_D$ partition all perfect matchings of~$UG_n$, the proof is complete.

Next, consider the case where we obtain $G$ by gluing $a$ and $z$ in $UG_n$ and make appropriate specializations. A good matching of this graph is one which could be extended to a perfect matching of $UG_n$ by adding $a$ or $z$. Thus, the matchings from $A_n$, $C_n$, and $D_n$ all descend to good matchings of $G$. We expand the trace, \[
\operatorname{tr} \left( \begin{bmatrix} \dfrac{x_{i_1}}{x_{i_n}} & y_{i_n} x_a\\ 0 & \dfrac{y_{i_n} x_{i_n}}{x_{i_1}} \end{bmatrix} M_n \right) = \frac{x_{i_1}}{x_{i_n}} A_n + y_{i_n} x_a C_n + \frac{y_{i_n} x_{i_n}}{x_{i_1}} D_n.
\]

We see that the coefficients on $A_n, C_n,$ and $D_n$ are as in equation~\eqref{eq:ExpandUR}, with one less factor of $x_a = x_z$. This matches the relationship between perfect matchings and good matchings. The situation is similar for a band graph obtained from gluing~$b$ and~$w$.
\end{proof}

When a snake graph $UG_n$ or band graph $G$ is associated to an arc or closed curve $\gamma$ on an orbifold $\mathcal{O}$ with triangulation $T$, we define $X_{\gamma, T} := \frac{1}{{\rm cross}(T,\gamma)}\sum_P x(P) h(P) $ where $P$ is the sum of weighted perfect matchings of $UG_n$ or weighted good matchings of $G$, as in Definitions~\ref{def:genArc} and~\ref{def:closedcurve}. These definitions also cover some special cases when it is not clear how to build a snake graph from the arc or curve.

\subsection{Lift for generalized arcs}\label{subsec:LiftGenArc}

We give brief motivation for the crossing diagonals in generalized snake graphs from arcs which wind around orbifold points. If the order, $p$, of the orbifold point is greater then two, then when we lift a piece of such an arc to a $p$-fold cover, the lifted arc passes through a $p$-gon. If $p > 3$, then this is an untriangulated $p$-gon.

The standard snake graph construction relies on an arc passing through a triangulation. However, by using a loosened notion of $T$-paths, we can determine the appropriate expansion formula for such arcs which wind around orbifold points. Since the covers we consider are not triangulated but instead dissected into polygons, which is a setting not fully explored in $T$-path literature, we use this as a heuristic rather than a proof. Sections \ref{sec:$M$-path}--\ref{sec:Skein} will formally verify these formulas.

The concept of $T$-paths was defined originally by Schiffler and Thomas in~\cite{Schiffler-Thomas} to give cluster expansion formulas in unpunctured surfaces and provide a proof of positivity for these cluster algebras as a corollary. Musiker and Gunawan expanded the $T$-path construction to once-punctured disks in~\cite{Musiker-Gunawan}.

As always, let $\gamma$ be an arc on a surface $(S,M)$ with triangulation $T = \{\tau_1, \ldots, \tau_n\}$, and let $d = e(\gamma,T) > 0$. Fix an arbitrary orientation to each arc $\tau \in T$ and to $\gamma$, and let $\tau^-$ be an arc isotopic to $\tau$ with opposite orientation. Let $\tau_{i_1},\ldots, \tau_{i_d}$ be the arcs crossed by $\gamma$, with order determined by $\gamma$'s orientation. Loosely, a (complete) $T$-path is a path $\alpha = (\alpha_1,\ldots, \alpha_{2d+1})$ such that
\begin{enumerate}\itemsep=0pt
 \item[1)] each $\alpha_i$ is equivalent to $\tau_j$ or $\tau_j^-$ for some $\tau_j \in T$,
 \item[2)] for $1 \leq i < n$, $t(\alpha_i) = s(\alpha_{i+1})$,
 \item[3)] $s(\alpha_1) = s(\gamma)$ and $t(\alpha_{2d+1}) = t(\gamma)$,
 \item[4)] (this requirement makes the $T$-path ``complete'') $\alpha_{2j} = \tau_{i_j}$ for $1 \leq j \leq d$.
\end{enumerate}

From each $T$ path we obtain a monomial where variables associated to the arcs crossed on odd steps are in the numerator and variables from arcs crossed on even steps are in the denominator. Then, we sum the monomials from all $T$-paths from $\gamma$ to obtain $x_\gamma$. Note that for a complete $T$-path, each denominator is equal to ${\rm cross}(T,\gamma)$.

Given an arc $\gamma$, the collection of $T$-paths from $\gamma$ are in bijection with perfect matchings of~$G_{T,\gamma}$. Moreover, we can draw each complete $T$-path on $G_{T,\gamma}$ by using the dashed diagonals in each tile as the steps along the arcs crossed by $\gamma$, that is, the even indexed steps. The set of edges used by the odd-indexed steps (those not on dashed edges) is a perfect matching of the graph.

As an example of the sort of arcs we are describing, consider $\gamma$ as below, where the orbifold point is order~5.

\begin{center}
\begin{tikzpicture}[scale = 2]

\tikzset{->-/.style={decoration={
 markings,
 mark=at position #1 with {\arrow{>}}},postaction={decorate}}}

\draw[out=30,in=-30,looseness=1.5] (0,0.3) to (0,1.5);
\draw[out=150,in=-150,looseness=1.5] (0,0.3) to (0,1.5);

\draw[in=0,out=45,looseness =1.25] (0,0.3) to (0,1.3);
\draw[in=180,out=135,looseness=1.25] (0,0.3) to (0,1.3);

\draw[orange,thick,in=165,out=0] (-0.5,1.2) to (0,1.15);
\draw[orange, thick,out=-15,in=0,looseness = 2] (0,1.15) to (0,0.8);
\draw[orange,thick,out=180,in=180,->=0.5,looseness=1.5] (0,1.2) to (0,0.8);
\draw[orange,thick,out=0,in=10,looseness=1.4] (0,1.2) to (0.1,0.7);
\draw[orange,thick] (0.1,0.7) to (-0.5,0.6);

\draw[fill=black] (0,0.3) circle [radius=2pt];
\draw[fill=black] (0,1.5) circle [radius=2pt];
\draw[thick] (0,1) node {$\mathbf{\times}$};
\node[scale=0.8] at (0,1.36) {$c$};
\end{tikzpicture}
\end{center}

One lift of this configuration is as below.

\begin{center}
\begin{tikzpicture}[scale = 1.5]
\tikzset{->-/.style={decoration={
 markings,
 mark=at position #1 with {\arrow{>}}},postaction={decorate}}}
\draw(0:2\R) -- (36:2\R) -- (72:2\R) -- (36*3:2\R) -- (36*4:2\R) -- (36*5:2\R) -- (36*6:2\R) -- (36*7:2\R) -- (36*8:2\R) -- (36*9:2\R) -- (0:2\R);
\draw(0:2\R) -- (72:2\R) to node[above] {$c_2$} (36*4:2\R) -- (36*6:2\R) -- (36*8:2\R) to node[above]{$c_1$} (0:2\R);
\draw[thick] (0:2\R) -- (72:2\R);
\draw[thick, blue, dashed] (-72:2\R) -- (36*4:2\R);
\draw[thick, green, dashed] (-72:2\R) -- (72:2\R);
\draw[thick, green, dashed] (36*4:2\R) -- (0:2\R);
\draw[thick, orange, <- = 1] (126:2.1\R) -- (-18:2.1\R);
\end{tikzpicture}
\end{center}

We consider possible sub-path $(\alpha_{2i},\alpha_{2i+1}, \alpha_{2(i+1)})$ of a $T$-path from the lift of $\gamma$. As in the definition, $\alpha_{2i}$ will go along $c_1$ in some direction and $\alpha_{2(i+1)}$ will go along $c_2$, as required by the definition. Then, $\alpha_{2i+1}$ will have to connect $t(\alpha_{2i})$ and~$s(\alpha_{2(i+1)})$. If we lift the requirement that each $\alpha_j$ is an arc in the triangulation, we see that the four diagonals highlighted (one being a~side of the polygon in this case) all connect end points of $c_1$ and $c_2$.

Each of the polygons in the lifts is regular since all sides correspond to the same arc in the orbifold. Thus, we can use elementary geometry to write the lengths of these diagonals in terms of the length of the sides of the polygon.

\begin{Definition}\label{def:kdiagonal}
A \emph{$k$-diagonal} in a polygon is one which skips $k-1$ vertices. For instance, boundary edges in a $p$-gon are both 1-diagonals and $(p-1)$-diagonals.
\end{Definition}

The following lemma comes from work in \cite{Lang}.

\begin{Lemma} \label{lem:polygonlength}
A $k$-diagonal in a regular $p$-gon with sides of length $s$ has length $U_{k-1}(\lambda_p) \cdot s$, where $U_k(x)$ denotes the $k$-th normalized Chebyshev polynomial as in Definition~{\rm \ref{def:Chebyshev}}.
\end{Lemma}

The fact that we have four options for potential steps between $c_1$ and $c_2$ leads to the hexagonal tiles discussed in Section \ref{Sec:Construction}. Note that the configuration between the two dashed lines in these tiles looks similar to the lift of the generalized arc above. In Lemma~\ref{lem:ChebyshevMatrices}, we will see these Chebyshev polynomials also arise from products of matrices in ${\rm SL}_2(\Real)$.

\section[$M$-path from an arc in a triangulated orbifold]{$\boldsymbol{M}$-path from an arc in a triangulated orbifold} \label{sec:$M$-path}

In the previous section, we associated a product of $2 \times 2$ matrices to the universal snake graph. Now, following the construction of \cite{MW}, we will also associate products of matrices to arbitrary arcs or curves on a triangulated orbifold; Theorem \ref{Thm:ArcsAndGraphs} will show a relationship between these two systems of matrices. This method will allow us to extend our snake graph formula to generalized arcs and closed curves.

Similar to the graph case, we break arcs or closed curves into a series of \emph{elementary steps} and associate $2 \times 2$ matrices to each step. Arcs do not have a unique associated $M$-path, but in Section \ref{sec:StandardMPath} we will both describe a convention for which $M$-path to use and show that the statistics we use do not depend on the path.

While the start and terminal point of an arc on an orbifold coincide with the set of marked points, elementary steps and the $M$-paths in general go between points which are near marked points but are not marked points themselves. To formalize this, draw a small circle, $h_m$, around each marked point $m$. These should be small enough that $h_m$ does not intersect $h_{m'}$ for another distinct marked point $m'$. If $\tau$ is an arc incident to $m$, let $v_{m,\tau}$ be the intersection of~$\tau$ and~$h_m$. If $\tau$ is a standard arc, we define $v^+_{m,\tau}$ (respectively, $v^-_{m,\tau}$) to be a point on $h_m$ that is counterclockwise (clockwise) of $v_{m,\tau}$. If $\tau$ is a pending arc, we define $v^{-,-}_{m,\tau}$, $v^{-,+}_{m,\tau}$, $v^{+,-}_{m,\tau}$, and $v^{+,+}_{m,\tau}$ to be, in counterclockwise order, four spots along $h_m$ such that $v^{+,+}_{m,\tau}$ is clockwise of all of $\tau$, $v^{-,-}_{m,\tau}$ is counterclockwise of all of~$\tau$, and~$v^{-,+}_{m,\tau}$ and~$v^{+,-}_{m,\tau}$ are contained within~$\tau$, drawn as a~loop, such that $v^{-,+}_{m,\tau}$ is counterclockwise from $v^{+,-}_{m,\tau}$.

\begin{center}
\begin{tikzpicture}
\draw (-2,0) -- (1,0);
\draw(-0.5,0) to node[right] {$\tau$} (-0.5,2);
\node[] at (-0.5,-0.2) {$m$};
\node[] at (-1.2, 0.3){$v^+_{m,\tau}$};
\node[] at (0.2, 0.3){$v^-_{m,\tau}$};
\node[circle, fill = black, scale = 0.2] at (-0.7,0.2){};
\node[circle, fill = black, scale = 0.2] at (-0.3,0.2){};

\draw(3,0) -- (7,0);
\draw(5,0) -- (3.5,2);
\draw(5,0) -- (6.5,2);
\node[] at (5, -0.2){$m$};
\node[circle, fill = black, scale = 0.2] at (4.7,0.2){};
\node[] at (4.2, 0.3){$v_{m,\tau}^{+,+}$};
\node[circle, fill = black, scale = 0.2] at (5.2,0.5){};
\node[] at (5.4, 1.2){$v_{m,\tau}^{+,-}$};
\node[circle, fill = black, scale = 0.2] at (4.8,0.5){};
\node[] at (4.6,1.2){$v_{m,\tau}^{-,+}$};
\node[circle, fill = black, scale = 0.2] at (5.3,0.2){};
\node[] at (5.8, 0.3){$v_{m,\tau}^{-,-}$};
\end{tikzpicture}
\end{center}

Given an arc $\gamma$, with end points $s(\gamma)$ and $t(\gamma)$, any representative $M$-path will go between~$v^\pm_{s(\gamma),\tau}$ or $v^{\pm,\pm}_{s(\gamma),\tau}$ and $v^\pm_{t(\gamma), \tau'}$ or $v^{\pm,\pm}_{t(\gamma), \tau'}$where $\tau$ and $\tau'$ are arcs in the triangulation incident to $s(\gamma)$ and $t(\gamma)$ respectively.

First, we recall the three types of elementary steps used in the surface case~\cite{MW}:
\begin{itemize}\itemsep=0pt
 \item An elementary step of type 1 goes from $v^\pm_{m,\tau}$ to $v^\mp_{m,\tau'}$ where $\tau$ and $\tau'$ share an endpoint and border the same triangle in $T$. If $\sigma$ is the third side of the triangle, then we associate the matrix $\left[\begin{smallmatrix} 1 & 0 \\ \pm \frac{x_{\sigma}}{x_\tau x_{\tau'}} & 1 \end{smallmatrix}\right]$. The sign of $\frac{x_{\sigma}}{x_\tau x_{\tau'}}$ is positive if we travel from $v^+_{m,\tau}$ to $v^-_{m,\tau'}$ and negative otherwise.
 \item An elementary step of type 2 goes from $v^\pm_{m,\tau}$ to $v^\mp_{m,\tau}$; that is, this step crosses the arc $\tau$. We associate the matrix $\left[\begin{smallmatrix} 1 & 0 \\ 0 & y_\tau \end{smallmatrix}\right]$ if we go from $v^-_{m,\tau}$ to $v^+_{m,\tau}$ and $\left[\begin{smallmatrix} y_\tau & 0 \\ 0 & 1 \end{smallmatrix}\right]$ otherwise.
 \item An elementary step of type 3 follows an arc $\tau$ in the triangulation. That is, if $\tau$ connects marked points $m$ and $m'$, then this step goes from $v^\pm_{m, \tau}$ to $v^\mp_{m',\tau}$. We associate to this the matrix $\left[\begin{smallmatrix} 0 & \pm x_\tau \\ \frac{\mp 1}{x_\tau} & 0 \end{smallmatrix}\right]$. We use $+ x_\tau$ and $\frac{-1}{x_\tau}$ if this step sees $\tau$ on the right and uses the opposite signs if it sees~$\tau$ on the left.
\end{itemize}

As we are working in a triangulated orbifold, we show how to update these elementary steps to interactions with a pending arc. In particular, we show how to decompose a portion of an arc winding around an orbifold point into a sequence of elementary steps; combining this with the above elementary steps will allow us to decompose any arc in a triangulated orbifold.

First, we can go from $v^{\pm,\mp}_{m,\rho}$ to $v^{\mp,\pm}_{m,\rho}$ where $\rho$ is a pending arc. We also examine this local configuration in a cover.

\begin{center}
\begin{tikzpicture}[scale=2]
\node[circle, fill = black, scale = 0.3, xshift = -6\R] (A) at (-4.1,0) {$w$};
\node[circle, fill = red, scale = 0.2, xshift = -6\R](B1) at (-2.62, 0.1){};
\node[circle, fill = red, scale = 0.2, xshift = -6\R](B1) at (-2.62, -0.1){};
\node[] (E) at (-3.8,0) {$\times$};
\node[xshift = -6\R](F) at (0.8, 0){};
\draw[xshift = -6\R] (0,0) to [out=60,in=120] node[auto]{$\alpha$}(2,0);
\draw[xshift = -6\R] (0,0) to [out=-60,in=-120] node[below]{$\beta$}(2,0);
\draw[green, xshift = -6\R](2,0) to [out = 140, in = 90, looseness = 1] node[auto]{$\rho$}(0.8,0);
\draw[green, xshift = -6\R](0.8,0) to [out = 270, in = 220, looseness = 1] node[auto]{}(2,0);
\draw[red, xshift = -6\R] (1.7, 0.1) to node[right]{$\gamma$} (1.7,-0.1);
\node[circle, fill = black, scale = 0.3, xshift = -6\R] (B) at (-2.05,0) {$v$};

 \draw[green] (0:1.5\R) to node[auto]{$\rho$} (90:1.5\R) to node[auto]{$\rho$} (135 + 45:1.5\R) to node[auto]{$\rho$} (270 :1.5\R) to node[auto]{$\rho$} (0:1.5\R);
 \draw (0:1.5\R) to node[above]{$\alpha$} (45:1.5\R) to node[above]{$\beta$} (90:1.5\R) to node[above]{$\alpha$} (135:1.5\R) to node[above]{$\beta$} (135 + 45:1.5\R) to node[above]{$\alpha$} (180 + 45:1.5\R) to node[above]{$\beta$} (270:1.5\R) to node[above]{$\alpha$} (270 + 45:1.5\R)to node[above]{$\beta$} (0:1.5\R);
 \draw[red] (80:1.2\R) to node[below]{$\gamma$} (100:1.2\R);
 \node[circle, fill = red, scale = 0.2](B1) at (100:1.2\R){};
 \node[circle, fill = red, scale = 0.2](B1) at (80:1.2\R){};
 \draw[dashed, green] (0:1.5\R) to node[below]{$\lambda_p \cdot \rho$} (180:1.5\R);
\end{tikzpicture}
\end{center}

In the cover, this resembles an elementary step of type 1 from \cite{MW}. Accordingly, we associate to this a matrix $\left[\begin{smallmatrix} 1 & 0 \\ \pm \frac{\lambda_p \rho}{\rho^2} & 1 \end{smallmatrix}\right] = \left[\begin{smallmatrix} 1 & 0 \\ \pm \frac{\lambda_p}{\rho} & 1 \end{smallmatrix}\right]$. As for an elementary step of type 1 in a surface, we use~$\frac{\lambda_p}{\rho}$ if we travel clockwise (from $v^{\mp, \pm}_{m,\rho}$ to $v^{\pm, \mp}_{m,\rho}$) and use $\frac{-\lambda_p}{\rho}$ otherwise.

If $\gamma$ does not have self-intersections, this is the only sort of step we will see. But if $\gamma$ winds $k\geq 1$ times around the orbifold point, we will also see an elementary step of type~3 along the pending arc, again between $v^{\pm,\mp}_{m,\rho}$ to $v^{\mp,\pm}_{m,\rho}$.

\begin{center}
\begin{tikzpicture}[scale=1.8]
\node[circle, fill = black, scale = 0.3, xshift = -6\R] (A) at (-4.1,0) {$w$};
\node[circle, fill = red, scale = 0.2, xshift = -6\R](B1) at (-2.54, 0.1){};
\node[circle, fill = red, scale = 0.2, xshift = -6\R](B1) at (-2.54, -0.1){};
\node[] (E) at (-3.8,0) {$\times$};
\draw[xshift = -6\R] (0,0) to [out=60,in=120] node[auto]{$\alpha$}(2,0);
\draw[xshift = -6\R] (0,0) to [out=-60,in=-120] node[below]{$\beta$}(2,0);
\draw[green, xshift = -6\R](2,0) to [out = 140, in = 90, looseness = 1] node[auto]{$\rho$}(0.8,0);
\draw[green, xshift = -6\R](0.8,0) to [out = 270, in = 220, looseness = 1] node[auto]{}(2,0);
\draw[red, xshift = -6\R](1.8, 0.1) to [out = 150, in = 90, looseness = 1] node[auto]{}(0.9,0);
\draw[red, xshift = -6\R](1.8, -0.1) to [out = 210, in = 270, looseness = 1] node[below]{$\gamma$}(0.9,0);

 \draw[green] (0:1.5\R) to node[auto]{$\rho$} (90:1.5\R) to node[left]{$\rho$} (135 + 45:1.5\R) to node[auto]{$\rho$} (270 :1.5\R) to node[auto]{$\rho$} (0:1.5\R);
 \draw (0:1.5\R) to node[above]{$\alpha$} (45:1.5\R) to node[above]{$\beta$} (90:1.5\R) to node[above]{$\alpha$} (135:1.5\R) to node[above]{$\beta$} (135 + 45:1.5\R) to node[above]{$\alpha$} (180 + 45:1.5\R) to node[above]{$\beta$} (270:1.5\R) to node[above]{$\alpha$} (270 + 45:1.5\R)to node[above]{$\beta$} (0:1.5\R);
 \draw[red] (175:1.2\R) to node[right]{$\gamma$} (100:1.2\R);
 \node[circle, fill = red, scale = 0.2]() at (175:1.2\R){};
 \node[circle, fill = red, scale = 0.2](B1) at (100:1.2\R){};
\end{tikzpicture}
\end{center}

This configuration resembles an elementary step of type 3, and as a result, we associate the matrix $\left[\begin{smallmatrix} 0 & \pm \rho \\ \frac{\mp 1}{\rho} & 0 \end{smallmatrix}\right]$, with the same rule for determining the sign as before.

We can treat an elementary step of type 2 across a pending arc, that is, between $v^{\pm,\pm}_{m,\rho}$ and $v^{\pm,\mp}_{m,\rho}$, the same as for a pending arc.

\begin{Definition}If $\kappa$ is an $M$-path whose sequence of elementary steps has associated matrices $\eta_1,\ldots, \eta_n$, then we define $M(\kappa) = \eta_n \cdots \eta_1$.
\end{Definition}

\section[Standard $M$-path]{Standard $\boldsymbol{M}$-path}\label{sec:StandardMPath}

We give an algorithm of assigning a $M$-path, $\kappa_\gamma$, to an arc or closed curve $\gamma$, which consists of a series of connected elementary steps. We say that this convention produces the ``standard $M$-path'' associated to $\gamma$. As an informal heuristic, we will pick an orientation of~$\gamma$, then always travel along the right of $\gamma$.

First, we utilize the symmetry about an orbifold point to choose a convenient representative for $\gamma$. At each pending arc that $\gamma$ crosses, we choose a representative that winds clockwise and less than $p$ times around the incident orbifold point, with one exception. If $\gamma$ crosses a pending arc which is based at a vertex to the left of $\gamma$, and if $\gamma$ is isotopic to one which winds~0 times around this orbifold point, then we will use a representative of $\gamma$ which winds $p$ times around this orbifold point. The reason why we choose this will be made clear in the description of $\kappa_\gamma$.

As before, let $\gamma$ be an arbitrary arc on an orbifold $\mathcal{O}$ with triangulation $T = \{\tau_1,\ldots, \tau_n\}$. Let $\tau_{i_1},\ldots, \tau_{i_d}$ be the arcs which $\gamma$ crosses, with order determined by an orientation on $\gamma$.

Suppose the first triangle that $\gamma$ cuts through has sides $\alpha,\beta,\tau_{i_1}$, in clockwise order, so that~$\alpha$ and~$\beta$ share an endpoint at $s(\gamma)$. Then, $\kappa_\gamma$ will start at $v_{s(\gamma),\alpha}^-$, and follow~$\alpha$ with a step of type~3, followed by a step of type~1 from~$a$ to~$\tau_{i_1}$.

Similarly, suppose the last triangle that $\gamma$ cuts through has sides $w$, $z$, $\tau_{i_d}$ in clockwise order, with $w$ and $z$ both touching~$t(\gamma)$. Then, the last few steps of $\kappa_\gamma$ will be a step of type~2 crossing~$\tau_{i_d}$, a step of type 1 from $\tau_{i_d}$ to $z$, and a step of type~3 along~$z$. Then, $\kappa_\gamma$ will end at~$v^+_{t(\gamma),z}$.

We next explain the sequence of steps we use between $\tau_{i_j}$ and $\tau_{i_{j+1}}$ for $1 \leq j \leq n-1$, where these are both standard arcs. This sequence will involve crossing~$\tau_{i_j}$ but not~$\tau_{i_{j+1}}$. First, we use a step of type 2 to cross $\tau_{i_j}$. Then, if $\tau_{i_j}$ and $\tau_{i_{j+1}}$ share a vertex to the right of~$\gamma$, then we use a step of type 1. Call this sequence of a step of type 2 and a step of type 1 a \emph{compound step of type~A}. If $\tau_{i_j}$ and $\tau_{i_{j+1}}$ share a vertex to the left of $\gamma$, let $\sigma_j$ be the third arc in this triangle. Then we use a step of type 1 between $\tau_{i_j}$ and $\sigma_j$, a step of type 3 along $\sigma_j$, and a step of type 1 between $\tau_{i_j}$ and $\sigma_j$. We call this sequence a \emph{compound step of type B}. A~``step'' will be assumed to be elementary unless otherwise specified.

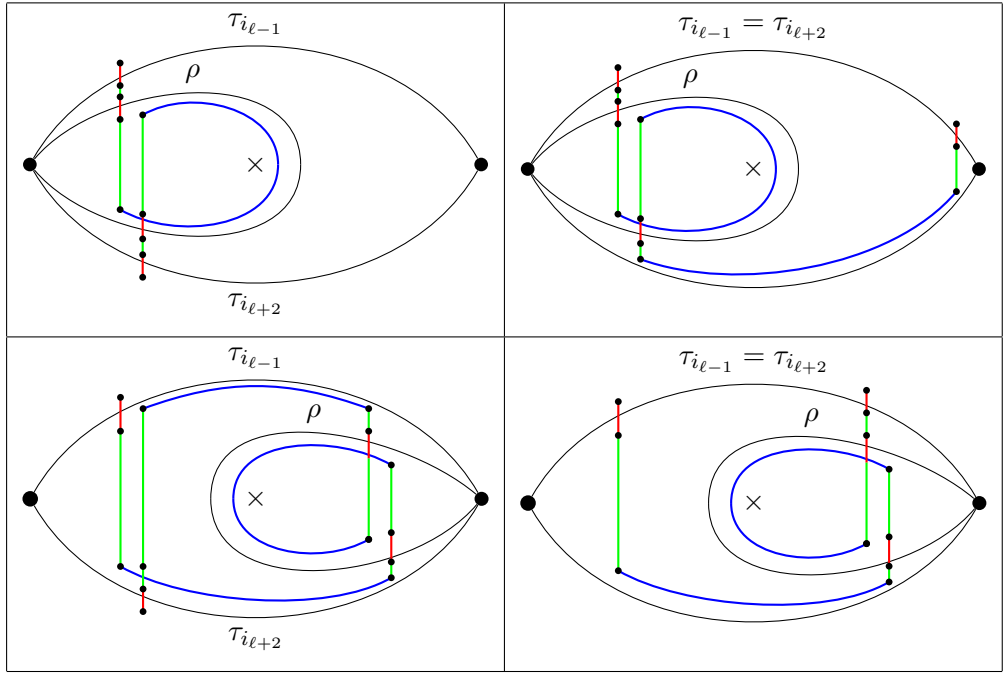
\begin{figure}
\begin{center}
\begin{tabular}{|c|c|}
\hline
 \begin{tikzpicture}[scale=3]
\node[circle, fill = black, scale = 0.5] (A) at (0,0) {};
\node[circle, fill = black, scale = 0.5] (B) at (2,0) {};
\node[] (E) at (1,0) {$\times$};
\draw (A) to [out=60,in=120] node[auto]{$\tau_{i_{\ell-1}}$}(B);
\draw (A) to [out=-60,in=-120] node[below]{$\tau_{i_{\ell+2}}$}(B);
\draw[](A) to [out = 50, in = 90, looseness = 1] node[auto]{$\rho$}(1.2,0);
\draw[](1.2,0) to [out = 270, in = 310, looseness = 1] node[auto]{}(A);
\draw[red, thick](0.4, 0.45) to (0.4, 0.35);
\draw[green, thick] (0.4, 0.35) to (0.4, 0.3);
\draw[red, thick] (0.4,0.3) to (0.4, 0.2);
\draw[green, thick] (0.4, 0.2) to (0.4, -0.2);
\draw[blue, thick] (0.4, -0.2) to [out = -30, in = 270] (1.1,0);
\draw[blue, thick] (1.1, 0) to [out = 90, in = 30] (0.5, 0.22);
\draw[green, thick] (0.5, 0.22) to (0.5, -0.22);
\draw[red, thick] (0.5, -0.22) to (0.5, -0.33);
\draw[green, thick] (0.5, -0.33) to (0.5, -0.4);
\draw[red, thick] (0.5, -0.4) to (0.5, -0.5);
\node[circle, fill = black, scale = 0.25] at (0.4, 0.2){};
\node[circle, fill = black, scale = 0.25] at (0.4, 0.3){};
\node[circle, fill = black, scale = 0.25] at (0.4, 0.35){};
\node[circle, fill = black, scale = 0.25] at (0.4, 0.45){};
\node[circle, fill = black, scale = 0.25] at (0.5, -0.4){};
\node[circle, fill = black, scale = 0.25] at (0.5, -0.5){};
\node[circle, fill = black, scale = 0.25] at (0.5, -0.33){};
\node[circle, fill = black, scale = 0.25] at (0.5, 0.22){};
\node[circle, fill = black, scale = 0.25] at (0.5, -0.22){};
\node[circle, fill = black, scale = 0.25] at (0.4, -0.2){};
\end{tikzpicture}
& \begin{tikzpicture}[scale = 3]
\node[circle, fill = black, scale = 0.5] (A) at (0,0) {};
\node[circle, fill = black, scale = 0.5] (B) at (2,0) {};
\node[] (E) at (1,0) {$\times$};
\draw (A) to [out=60,in=120] node[auto]{$\tau_{i_{\ell-1}} = \tau_{i_{\ell+2}}$}(B);
\draw (A) to [out=-60,in=-120] node[below]{}(B);
\draw[](A) to [out = 50, in = 90, looseness = 1] node[auto]{$\rho$}(1.2,0);
\draw[](1.2,0) to [out = 270, in = 310, looseness = 1] node[auto]{}(A);
\draw[red, thick](0.4, 0.45) to (0.4, 0.35);
\draw[green, thick] (0.4, 0.35) to (0.4, 0.3);
\draw[red, thick] (0.4,0.3) to (0.4, 0.2);
\draw[green, thick] (0.4, 0.2) to (0.4, -0.2);
\draw[blue, thick] (0.4, -0.2) to [out = -30, in = 270] (1.1,0);
\draw[blue, thick] (1.1, 0) to [out = 90, in = 30] (0.5, 0.22);
\draw[green, thick] (0.5, 0.22) to (0.5, -0.22);
\draw[red, thick] (0.5, -0.22) to (0.5, -0.33);
\draw[green, thick] (0.5, -0.33) to (0.5, -0.4);
\draw[blue, thick] (0.5, -0.4) to [out = -20, in = 230, looseness = 0.8] (1.9, -0.1);
\draw[green, thick] (1.9, -0.1) to (1.9, 0.1);
\draw[red, thick] (1.9, 0.1) to (1.9, 0.2);
\node[circle, fill = black, scale = 0.25] at (0.4, 0.2){};
\node[circle, fill = black, scale = 0.25] at (0.4, 0.3){};
\node[circle, fill = black, scale = 0.25] at (0.4, 0.35){};
\node[circle, fill = black, scale = 0.25] at (0.4, 0.45){};
\node[circle, fill = black, scale = 0.25] at (0.5, -0.4){};
\node[circle, fill = black, scale = 0.25] at (1.9, -0.1){};
\node[circle, fill = black, scale = 0.25] at (0.5, -0.33){};
\node[circle, fill = black, scale = 0.25] at (0.5, 0.22){};
\node[circle, fill = black, scale = 0.25] at (0.5, -0.22){};
\node[circle, fill = black, scale = 0.25] at (0.4, -0.2){};
\node[circle, fill = black, scale = 0.25] at (1.9, 0.1){};
\node[circle, fill = black, scale = 0.25] at (1.9, 0.2){};
\end{tikzpicture}\\\hline
 \begin{tikzpicture}[scale=3]
\node[circle, fill = black, scale = 0.3] (A) at (0,0) {$w$};
\node[circle, fill = black, scale = 0.3] (B) at (2,0) {$v$};
\node[] (E) at (1,0) {$\times$};
\draw (A) to [out=60,in=120] node[auto]{$\tau_{i_{\ell-1}}$}(B);
\draw (A) to [out=-60,in=-120] node[below]{$\tau_{i_{\ell+2}}$}(B);
\draw[](B) to [out = 140, in = 90, looseness = 1] node[above]{$\rho$}(0.8,0);
\draw[](0.8,0) to [out = 270, in = 230, looseness = 1] node[auto]{}(B);
\draw[red, thick](0.4, 0.45) to (0.4, 0.3);
\draw[green, thick] (0.4, 0.3) to (0.4, -0.3);
\draw[blue, thick] (0.4, -0.3) to [out = -30, in = 210, looseness = 0.7] (1.6, -0.35);
\draw[green, thick](1.6, -0.35) to (1.6, -0.28);
\draw[red, thick] (1.6, -0.28) to (1.6, -0.15);
\draw[green, thick] (1.6, -0.15) to (1.6, 0.15);
\draw[blue, thick] (1.6, 0.15) to [out = 150, in = 90] (0.9, 0);
\draw[blue, thick] (0.9, 0) to [out = 270, in = 210] (1.5, -0.18);
\draw[green, thick] (1.5, -0.18) to (1.5, 0.18);
\draw[red, thick] (1.5, 0.18) to (1.5, 0.3);
\draw[green, thick] (1.5, 0.3) to (1.5, 0.4);
\draw[blue, thick] (1.5, 0.4) to [out = 160, in = 20] (0.5, 0.4);
\draw[green, thick] (0.5, 0.4) to (0.5, -0.4);
\draw[red, thick] (0.5, -0.4) to (0.5, -0.5);
\node[circle, fill = black, scale = 0.25] at (0.4, -0.3){};
\node[circle, fill = black, scale = 0.25] at (1.5, 0.3){};
\node[circle, fill = black, scale = 0.25] at (1.5, 0.4){};
\node[circle, fill = black, scale = 0.25] at (0.5, 0.4){};
\node[circle, fill = black, scale = 0.25] at (0.5, -0.4){};
\node[circle, fill = black, scale = 0.25] at (0.5, -0.5){};
\node[circle, fill = black, scale = 0.25] at (0.4, 0.3){};
\node[circle, fill = black, scale = 0.25] at (0.4, 0.45){};
\node[circle, fill = black, scale = 0.25] at (0.5, -0.3){};
\node[circle, fill = black, scale = 0.25] at (1.6, -0.35){};
\node[circle, fill = black, scale = 0.25] at (1.6, -0.28){};
\node[circle, fill = black, scale = 0.25] at (1.6, -0.15){};
\node[circle, fill = black, scale = 0.25] at (1.6, 0.15){};
\node[circle, fill = black, scale = 0.25] at (1.5, -0.18){};
\node[circle, fill = black, scale = 0.25] at (1.5, -0.18){};
\end{tikzpicture} & \begin{tikzpicture}[scale = 3]
\node[circle, fill = black, scale = 0.3] (A) at (0,0) {$w$};
\node[circle, fill = black, scale = 0.3] (B) at (2,0) {$v$};
\node[] (E) at (1,0) {$\times$};
\draw (A) to [out=60,in=120] node[auto]{$\tau_{i_{\ell-1}} = \tau_{i_{\ell+2}}$}(B);
\draw (A) to [out=-60,in=-120] node[below]{}(B);
\draw[](B) to [out = 140, in = 90, looseness = 1] node[above]{$\rho$}(0.8,0);
\draw[](0.8,0) to [out = 270, in = 230, looseness = 1] node[auto]{}(B);
\draw[red, thick](0.4, 0.45) to (0.4, 0.3);
\draw[green, thick] (0.4, 0.3) to (0.4, -0.3);
\draw[blue, thick] (0.4, -0.3) to [out = -30, in = 210, looseness = 0.7] (1.6, -0.35);
\draw[green, thick](1.6, -0.35) to (1.6, -0.28);
\draw[red, thick] (1.6, -0.28) to (1.6, -0.15);
\draw[green, thick] (1.6, -0.15) to (1.6, 0.15);
\draw[blue, thick] (1.6, 0.15) to [out = 150, in = 90] (0.9, 0);
\draw[blue, thick] (0.9, 0) to [out = 270, in = 210] (1.5, -0.18);
\draw[green, thick] (1.5, -0.18) to (1.5, 0.18);
\draw[red, thick] (1.5, 0.18) to (1.5, 0.3);
\draw[green, thick] (1.5, 0.3) to (1.5, 0.4);
\draw[red, thick] (1.5, 0.4) to (1.5, 0.5);
\node[circle, fill = black, scale = 0.25] at (0.4, 0.3){};
\node[circle, fill = black, scale = 0.25] at (0.4, 0.45){};
\node[circle, fill = black, scale = 0.25] at (0.4, -0.3){};
\node[circle, fill = black, scale = 0.25] at (1.6, -0.35){};
\node[circle, fill = black, scale = 0.25] at (1.6, -0.28){};
\node[circle, fill = black, scale = 0.25] at (1.6, -0.15){};
\node[circle, fill = black, scale = 0.25] at (1.6, 0.15){};
\node[circle, fill = black, scale = 0.25] at (1.5, -0.18){};
\node[circle, fill = black, scale = 0.25] at (1.5, -0.18){};
\node[circle, fill = black, scale = 0.25] at (1.5, 0.3){};
\node[circle, fill = black, scale = 0.25] at (1.5, 0.4){};
\node[circle, fill = black, scale = 0.25] at (1.5, 0.5){};
\end{tikzpicture}\\ \hline
\end{tabular}
\end{center}
\caption{Sequences of elementary steps to use when $\gamma$ crosses a pending arc twice consecutively. If you are viewing this in color, the green steps are type 1, red steps are type 2 and blue steps are type 3.}\label{fig:MPathRules}
\end{figure}

Now we explain the protocol when $\gamma$ crosses a pending arc $\rho = \tau_{i_\ell}$. First, we assume that $\rho$ is not the first or last arc that $\gamma$ crosses, so $\tau_{i_{\ell-1}}$ and $\tau_{i_{\ell+2}}$ are not necessarily distinct arcs in the bigon or monogon surrounding $\rho$. We give rules for the transition from $\tau_{i_{\ell-1}}$ to $\tau_\ell = \rho$, for the winding inside $\rho$, and the transition from $\tau_{\ell+1}$ to $\tau_{\ell+2}$. These depend on whether $\rho$ is based at a vertex to the right or left of $\gamma$, and whether $\tau_{i_{\ell-1}}$ and $\tau_{i_{\ell+2}}$ are distinct or not. These will not depend on whether $\tau_{i_{\ell-1}}$ and $\tau_{i_{\ell+2}}$ are standard or pending.

First, suppose that $\rho$ is based at a marked point $w$ to the right of $\gamma$, and that~$\tau_{i_{\ell-1}}$ and~$\tau_{i_{\ell+2}}$ are distinct. Then, between $\tau_{i_{\ell-1}}$ and $\rho$, we use a compound step of type~$A$. Between the two crossings of $\rho$, we use an elementary step of type~2 to cross $\rho$. If $\gamma$ winds $k \geq 0$ times around the orbifold point incident to $\rho$, we include a step of type~1 from $v^{-,+}_{w,\rho}$ to $v^{+,-}_{w,\rho}$ followed by $k$ iterations of a step of type 3 along $\rho$ and a step of type 1 from $v^{-,+}_{w,\rho}$ to $v^{+,-}_{w,\rho}$. Finally, we transition from $\rho = \tau_{i_{\ell+1}}$ to $\tau_{i_{\ell+1}}$ with a compound step of type $A$. See the top left of Fig.~\ref{fig:MPathRules}.

Otherwise, we have that $\tau_{i_{\ell-1}} = \tau_{i_{\ell+2}}$, so that $\gamma$ crosses the same arc both before and after crossing $\rho$. Then, at the transition from $\rho = \tau_{i_{\ell+1}}$ to $\tau_{i_{\ell+1}}$, we instead use a compound step of type $B$. The earlier part of the sequence remains the same. See the top right of Fig.~\ref{fig:MPathRules}.

Now suppose that $\rho$ is based to the left of $\gamma$, and first suppose that $\tau_{i_{\ell-1}}$ and $\tau_{i_{\ell+2}}$ are distinct. We can use a compound step of type $B$ to transition from $\tau_{i_{\ell-1}}$ to $\rho$. From our choice of a representative of $\gamma$, we know that $\gamma$ winds $k \geq 1$ times around the orbifold point incident to $\rho$. We can use the same algorithm for the sequence of steps within the pending arc $\rho$, but we will only include $k-1$ self-intersections. Then, we use another compound step of type $B$ to transition from $\rho$ to $\tau_{i_{\ell+2}}$. Note that while $\kappa_\gamma$ only intersects itself $k-1$ times inside the pending arc $\gamma$, it intersects itself one more time outside the pending arc. Thus, $\kappa_\gamma$ remains homotopic to $\gamma$. See the bottom left of Fig.~\ref{fig:MPathRules}.

If $\tau_{i_{\ell-1}}$ and $\tau_{i_{\ell+2}}$ are not distinct, then we can include all $k$ self-intersections in the pending arc $\rho$. In this case, we use a compound step of type $A$ when transitioning from $\tau_{i_{\ell+1}}$ to $\tau_{i_{\ell+2}}$. See the bottom right of Fig.~\ref{fig:MPathRules}.

\begin{Example} As an example, here is the corresponding expansion of matrices for the piece of $\gamma$ portrayed in the case $\tau_{i_{\ell-1}} \neq \tau_{i_{\ell+2}}$ and the pending arc $\rho = \tau_{i_\ell}$ is based to the right of $\gamma$, as in the top left of Fig.~\ref{fig:MPathRules}. For convenience, let $\alpha = \tau_{i_{\ell-1}}$ and $\beta = \tau_{i_{\ell + 2}}$
 \begin{gather*}
 \begin{bmatrix} 1 & 0 \\ \dfrac{x_\alpha}{x_\beta x_\rho} & 1 \end{bmatrix}
 \begin{bmatrix} 1 & 0 \\ 0 & y_{\beta} \end{bmatrix} \begin{bmatrix} 1 & 0 \\ \dfrac{\lambda_p}{x_\rho} & 1 \end{bmatrix}
\begin{bmatrix} 0 & x_\rho \\ \dfrac{-1}{x_\rho} & 0 \end{bmatrix}
 \begin{bmatrix} 1 & 0 \\ \dfrac{\lambda_p}{x_\rho} & 1 \end{bmatrix}
 \begin{bmatrix} 1 & 0 \\ 0 & y_{\rho} \end{bmatrix}
 \begin{bmatrix} 1 & 0 \\ \dfrac{x_\beta}{x_\alpha x_\rho} & 1 \end{bmatrix}
 \begin{bmatrix} 1 & 0 \\ 0 & y_{\alpha} \end{bmatrix}
\\
\qquad{}
= \begin{bmatrix} 1 & 0 \\ \dfrac{x_\alpha}{x_\beta x_\rho} & y_\beta \end{bmatrix}
\begin{bmatrix} \lambda_p & y_{\rho} x_{\rho} \\ \dfrac{(\lambda_p^2 - 1)}{x_{\rho}} & \lambda_p y_{\rho}\end{bmatrix}
 \begin{bmatrix} 1 & 0 \\ \dfrac{x_\beta}{x_\alpha x_\rho} & y_\alpha \end{bmatrix}.
\end{gather*}

In the second line, we multiply the matrices within each compound step. Notice that these matrices resemble those we assigned to pieces of the universal snake graph. We will eventually solidify this connection.
\end{Example}

The cases we have yet to discuss are when the first or last arc that $\gamma$ crosses is a pending arc. We again will vary our procedure based on whether the pending arc is based at a vertex to the right or to the left of~$\gamma$ or if it is based at $s(\gamma)$.

Let $P = s(\gamma)$ be the start of $\gamma$. Recall we choose a representative of $\gamma$ which winds counterclockwise around any orbifold point it encounters. If the pending arc, $\rho$, is based at a vertex other than $P$, then we use a step of type~3 along~$\alpha$, the boundary edge to the right of~$\gamma$, and a~step of type~1 from~$\alpha$ to~$\rho$.
The following steps will depend on how many times~$\gamma$ winds around this orbifold point and which arc $\gamma$ crosses next. These use the same compound steps as in the earlier discussion. For example, on the left-hand side of Fig.~\ref{fig:MPathRulesStart} if~$\gamma$ crosses $\alpha$ after winding around the orbifold point, then after a compound step which winds around the orbifold point, as drawn, we will use a compound step of type B to transition from~$\rho$ to~$\alpha$.

Next, suppose $s(\gamma)$ is also the unique marked point incident to~$\rho$ and~$\gamma$ winds at least once around the orbifold point. Note that if~$\gamma$ does not wind at least once around the orbifold point, it is isotopic to an arc that does not cross the pending arc. See the right-hand side of Fig.~\ref{fig:MPathRulesStart}.
As in the case when $\rho$ is not based at~$s(\gamma)$, the following steps depends on which arc~$\gamma$ crosses next.

\begin{figure}\centering
\begin{tabular}{|c|c|}
\hline
\begin{tikzpicture}[scale=3]
\node[] at (2.1,0) {$P$};
\node[circle, fill = black, scale = 0.3] (A) at (0,0) {$w$};
\node[circle, fill = black, scale = 0.3] (B) at (2,0) {$v$};
\node[] (E) at (1,0) {$\times$};
\draw (A) to [out=60,in=120] node[auto]{$\alpha$}(B);
\draw (A) to [out=-60,in=-120] node[below]{$\beta$}(B);
\draw[](A) to [out = 50, in = 90, looseness = 1] node[auto]{$\rho$}(1.2,0);
\draw[](1.2,0) to [out = 270, in = 310, looseness = 1] node[auto]{}(A);
\draw[blue, thick](1.9, 0.1) to [out = 130, in = 20] (0.4, 0.35);
\draw[green, thick] (0.4, 0.35) to (0.4, 0.28);
\draw[red, thick] (0.4,0.3) to (0.4, 0.2);
\draw[green, thick] (0.4, 0.2) to (0.4, -0.2);
\draw[blue, thick] (0.4, -0.2) to [out = -30, in = 270] (1.1,0);
\draw[blue, thick] (1.1, 0) to [out = 90, in = 30] (0.5, 0.22);
\draw[green, thick] (0.5, 0.22) to (0.5, -0.22);
\node[circle, fill = black, scale = 0.25] at (1.9, 0.1){};
\node[circle, fill = black, scale = 0.25] at (0.4, 0.35){};
\node[circle, fill = black, scale = 0.25] at (0.4, 0.3){};
\node[circle, fill = black, scale = 0.25] at (0.4, 0.2){};
\node[circle, fill = black, scale = 0.25] at (0.4, -0.2){};
\node[circle, fill = black, scale = 0.25] at (0.5, 0.22){};
\node[circle, fill = black, scale = 0.25] at (0.5, -0.22){};
\end{tikzpicture}
&
\begin{tikzpicture}[scale=3]
\node[] at (-0.1,0) {$P$};
\node[circle, fill = black, scale = 0.3] (A) at (0,0) {$w$};
\node[circle, fill = black, scale = 0.3] (B) at (2,0) {$v$};
\node[] (E) at (1,0) {$\times$};
\draw (A) to [out=60,in=120] node[auto]{}(B);
\draw (A) to [out=-60,in=-120] node[below]{}(B);
\draw[](A) to [out = 50, in = 90, looseness = 1] node[auto]{$\rho$}(1.2,0);
\draw[](1.2,0) to [out = 270, in = 310, looseness = 1] node[auto]{}(A);
\draw[green, thick] (0.2, 0.12) to (0.2, -0.12);
\draw[blue, thick] (0.1, -0.08) to [out = -40, in = 270] (1.1,0);
\draw[blue, thick] (1.1, 0) to [out = 90, in = 30] (0.2, 0.12);
\node[circle, fill = black, scale = 0.25] at (0.2, 0.12){};
\node[circle, fill = black, scale = 0.25] at (0.2, -0.12){};
\node[circle, fill = black, scale = 0.25] at (0.1, -0.08){};
\end{tikzpicture}
\\ \hline
\end{tabular}
\caption{Sequences of elementary steps when a pending arc is the first arc which $\gamma$ crosses. If you are viewing this in \textcolor{green}{color}, the green steps are type 1, red steps are type 2 and blue steps are type 3.}\label{fig:MPathRulesStart}
\end{figure}

The cases where the last arc that $\gamma$ crosses is a pending arc, $\rho$ are very similar. If $t(\gamma)$ is distinct from the unique marked point incident to $\rho$, the final compound step will start with a~step of type~2 to cross~$\rho$, then a step of type 1 and a step of type~3. We can see this by traveling the opposite direction along the $M$-path on the left of Fig.~\ref{fig:MPathRulesStart}.

If $t(\gamma)$ is the marked point incident to $\rho$, then our final compound step will be as in the case when $\rho$ is incident to $s(\gamma)$, but again with the order reversed.

Finally, we consider the case when $\gamma$ is a closed curve. Pick a triangle, $\Delta$, such that $\gamma$ consecutively crosses two of its arcs. Label these arcs $\tau_{i_1}$ and $\tau_{i_n}$ such that $\tau_{i_1}$ immediately follows $\tau_{i_n}$ in a clockwise order. Let $q$ be the endpoint of $\tau_{i_1}$ which is not also an endpoint of~$\tau_{i_n}$. Then, the standard $M$-path, $\kappa_\gamma$, will start and stop at $v_{q, \tau_{i_1}}^-$. The $M$-path can start with a~compound step of type $A$ or $B$, depending on whether $\tau_{i_1}$ and $\tau_{i_2}$ share a vertex to the right or left of the chosen orientation of $\gamma$. Then, since by construction $\tau_{i_n}$ and $\tau_{i_1}$ share an endpoint to the left of~$\gamma$, $\kappa_\gamma$ will end with a compound step of type $B$.

\subsection[Upper right entry does not depend on choice of $M$-path]{Upper right entry does not depend on choice of $\boldsymbol{M}$-path}

Lemma~4.8 in \cite{MW} shows that the upper right (trace) of matrices from $M$-paths associated to arcs (closed curves) on a surface does not depend on our choice of $M$-path. For instance, if $\gamma$ is a closed curve, the trace of $M(\kappa)$ for an $M$-path $\kappa$ from $\gamma$ does not depend on $\kappa's$ start and end point since trace is invariant under cyclic permutations.

\begin{Lemma}[{\cite[Lemma 4.8]{MW}}]\label{lem:MWMPathDontCare}
Let $\gamma_1$ and $\gamma_2$ be a generalized arc and closed curve with no contractible kinks, respectively, on a triangulated surface $(S,M)$. Then, given $\kappa_1$ and~$\kappa_2$, two $M$-paths associated to~$\gamma_1$, we have $\vert {\rm ur}(M(\kappa_1)) \vert = \vert {\rm ur}(M(\kappa_2))\vert $. If $\kappa_1'$ and $\kappa_2'$ are two $M$-paths associated to $\gamma_2$, we have $\vert {\rm tr}(M(\kappa_1)) \vert = \vert {\rm tr}(M(\kappa_2)) \vert$.
\end{Lemma}

Since we are in an orbifold, there are more ways to adjust an $M$-path associated to an arc $\gamma$; in particular, if an $M$-path winds $k$ times around an orbifold point of order~$p$, we can adjust it to wind $k + m p$ times for any integer $m$. We show in Lemma~\ref{lem:Mpathdontcare} that these adjustments still do not affect the statistics of the matrices which we care about.

\begin{Lemma}\label{lem:Mpathdontcare}
Let $\kappa_1$ and $\kappa_2$ be two $M$-paths which are identical except at one orbifold point of order $p$, such that at this orbifold point $\kappa_1$ winds $k$ times and $\kappa_2$ winds $k + mp$ times where $m \in \Z$. Then, up to universal sign, $M(\kappa_1) = M(\kappa_2)$. In particular, $\vert {\rm ur}(M(\kappa_1)) \vert = \vert {\rm ur}(M(\kappa_2))\vert$ and $\vert {\rm tr}(M(\kappa_1)) \vert = \vert {\rm tr}(M(\kappa_2)) \vert $.
\end{Lemma}

To prove this, we will prove a lemma about products of the elementary matrices which correspond to an $M$-path winding around an orbifold point of order $p$. It turns out products of these matrices have Chebyshev polynomials, evaluated at~$\lambda_p$, as coefficients.

\begin{Lemma} \label{lem:ChebyshevMatrices}
Let $k \geq 0$, and let $U_k(x)$ be the $k$-th normalized Chebyshev polynomial of the second kind. Then,
\begin{equation}\label{eqn:chebymatrix}
\left( \begin{bmatrix} 1 & 0 \\ \dfrac{\lambda_p}{x_\rho} & 1 \end{bmatrix} \begin{bmatrix} 0 & x_\rho \\ \dfrac{-1}{x_\rho} & 0 \end{bmatrix} \right)^k = \begin{bmatrix} -U_{k-2} (\lambda_p) & U_{k-1}(\lambda_p) \cdot x_\rho \vspace{1mm}\\ \dfrac{-U_{k-1}(\lambda_p)}{x_\rho} & U_k(\lambda_p) \end{bmatrix}
\end{equation}
and
\begin{equation}\label{eqn:negchebymatrix}
\left( \begin{bmatrix} 0 & -x_\rho \\ \dfrac{1}{x_\rho} & 0 \end{bmatrix} \begin{bmatrix} 1 & 0 \\ \dfrac{-\lambda_p}{x_\rho} & 1 \end{bmatrix} \right)^k = \begin{bmatrix} U_{k} (\lambda_p) & -U_{k-1}(\lambda_p) \cdot x_\rho \vspace{1mm}\\ \dfrac{U_{k-1}(\lambda_p)}{x_\rho} & -U_{k-2}(\lambda_p) \end{bmatrix}.
\end{equation}
\end{Lemma}

\begin{proof}Recall our convention that $U_{-1}(x) = 0$, $U_0(x) = 1$, and the normalized recurrence for $\ell > 0$: $U_\ell(x) = x U_{\ell-1}(x) - U_{\ell-2}(x)$. This proof follows by induction and the recurrence for Chebyshev polynomials.
\end{proof}

\begin{Remark}Equation \ref{eqn:negchebymatrix} can also be thought about as making sense of what matrix should be assigned to a generalized arc which winds $k$ times clockwise around an orbifold point.
\end{Remark}

\begin{Remark}Compare the matrices in Lemma \ref{lem:ChebyshevMatrices} with the statement of Proposition \ref{prop:MatchingMatrices} and the labels we include in hexagonal tiles from an arc with nontrivial winding about an orbifold point. In particular, note that if we consider a hexagonal tile as $UG_2$, then there is exactly one perfect matching in each $A_n$, $B_n$, $C_n$, and~$D_n$, and each matching uses exactly one edge with label $U_\ell(\lambda_p)x_\rho$ for some $\ell$ and for the pending arc $\rho$.
\end{Remark}

Now, we can prove Lemma~\ref{lem:Mpathdontcare}.

\begin{proof}If $m = 0$, this lemma is trivial. If $m > 0$, then the expansion of $M(\kappa_2)$ into elementary matrices will have a term $\left( \left[\begin{smallmatrix} 1 & 0 \\ \frac{\lambda_p}{\rho} & 1 \end{smallmatrix}\right] \left[\begin{smallmatrix} 0 & \rho \\ \frac{-1}{\rho} & 0 \end{smallmatrix}\right] \right)^{mp} = \left(\left[\begin{smallmatrix} -U_{p-2} (\lambda_p) & U_{p-1}(\lambda_p) \cdot \rho \\ \frac{-U_{p-1}(\lambda_p)}{\rho} & U_p(\lambda_p) \end{smallmatrix}\right]\right)^m $. Recall that $U_{p-2}(\lambda_p) = 1$, $U_{p-1}(\lambda_p) = 0$, and $U_p(\lambda_p) = -1$. Thus, this extra factor in the expansion of $M(\kappa_2)$ is simply $\pm \text{Id}$, where the sign depends on the parity of $m$. Thus, $\vert {\rm ur}(M(\kappa_1)) \vert = \vert {\rm ur}(M(\kappa_1)) \vert$. The case where $m < 0$ is similar.
\end{proof}

From Lemmas~\ref{lem:MWMPathDontCare} and~\ref{lem:Mpathdontcare}, we see that we can always use the standard $M$-path, $\kappa_\gamma$ for any generalized arc or closed curve~$\gamma$ and not affect the upper right or trace, respectively, of the associated matrix.

We have that the upper right (trace) of the matrix from an $M$-path for most arcs (closed curves) is a well-defined statistic. However, if $\gamma$ is a closed curve, there are some cases where~$\gamma$ does not cross any arcs on~$T$ and so the $M$-path may be ambiguous. Musiker and Williams deal with curves which are contractible or enclose a single puncture~\cite{MW}. In an orbifold, we can also have a curve which encloses a single orbifold point.

 We turn to the normalized Chebyshev polynomials \emph{of the first kind} which Musiker, Schiffler, and Williams use to describe arcs such as we are describing.

\begin{Definition}[{\cite[Definition 2.33 and Proposition 2.34]{MSW-bases}}]\label{def:ChebyShevFirstKind}
Let $T_\ell(x)$ denote the $\ell$-th norma\-lized Chebyshev polynomial of the first kind, for $\ell \geq -1$. These are given by initial polynomials $T_{0}(x) = 2$, $T_1(x) = x$, and the recurrence,
\[
T_\ell(x) = xT_{\ell-1}(x) - T_{\ell-2}(x).
\]
\end{Definition}

While the definition in \cite{MSW-bases} keeps track of an extra variable $Y$, for now we set $Y = 1$. See Section \ref{sec:RelateToPuncture} for a related discussion of $y$-variables.

We give the following as a corollary of Proposition 4.2 of \cite{MSW-bases}.

\begin{Proposition}\label{prop:MSWChebyshev}
Let $\gamma$ be isotopic to a closed loop encompassing a single orbifold point with $k\geq 0$ self intersections. Then, $x_\gamma = T_{k+1}(\lambda_p)$.
\end{Proposition}

This proposition largely follows from the following relationship amongst the normalized Chebyshev polynomials discussed here.

\begin{Lemma}\label{lem:RelatingChebyshev}
Let $U_\ell(x)$ and $T_\ell(x)$ be normalized Chebyshev polynomials of the second and first kind, respectively, as in Definitions {\rm \ref{def:Chebyshev}} and {\rm \ref{def:ChebyShevFirstKind}}. Then, for $\ell \geq 1$, \[
T_\ell(x) = U_\ell(x) - U_{\ell-2}(x).
\]
\end{Lemma}

Lemma~\ref{lem:RelatingChebyshev} can be proved using induction and the recurrence relations for each type of Chebyshev polynomials.

\begin{proof}[Proof of Proposition~\ref{prop:MSWChebyshev}]
If $\gamma$ intersects itself $k \geq 0$ times, we can build an $M$-path for $\gamma$, call it $\kappa$, which is a sequence of $k+1$ steps of types~1 and~3. From Lemmas~\ref{lem:ChebyshevMatrices} and~\ref{lem:RelatingChebyshev}, we immediately see that ${\rm tr}M(\kappa) = U_{k+1}(\lambda_p) - U_{k-1}(\lambda_p) = T_{k+1}(\lambda_p)$. We see this follows naturally from Proposition~4.2 of~\cite{MSW-bases} since if $\xi$ is an essential loop around this single orbifold point and~$\kappa$ is an~$M$-path from~$\xi$, then we have ${\rm tr}(M(\kappa)) = \lambda_p$.
\end{proof}

Now, we are prepared to state a complete definition.

\begin{Definition}\label{def:chi_gamma}
Let $\gamma$ be a generalized arc and $\gamma'$ be a closed curve on an unpunctured orbifold $\mathcal{O}$ with triangulation~$T$. Then, $\chi_{\gamma,T} = \vert \emph{ur} (M(\kappa_\gamma)) \vert$.

If $\gamma'$ is contractible, set $\chi_{\gamma',T} = -2$. If $\gamma'$ is isotopic to a closed loop encompassing a single orbifold point of order $p$, with $k \geq 0$ self-intersections, let $\chi_{\gamma',T} = T_{k+1}(\lambda_p)$. Otherwise, let $\chi_{\gamma',T} = \vert {\rm tr} (M(\kappa_{\gamma'})) \vert$.
\end{Definition}

In Theorem \ref{Thm:ArcsAndGraphs} we will compare $\chi_{\gamma,T}$ with $X_{\gamma,T}$. Recall we found $X_{\gamma,T}$ by building a snake graph from $\gamma$.

\begin{Remark}Musiker and Williams show in \cite[Section~4]{MW} that their matrices, after specializations, generalize work of Fock and Goncharov in~\cite{FG3} which also associated matrix products to paths in triangulated surfaces as a way to construct coordinates on the corresponding Teichmuller space.

In their paper defining generalized cluster algebras, Chekhov and Shapiro update the matrix products which compute $X$-coordinates (in the sense of Fock--Goncharov) to include orbifolds~\cite{Chekhov-Shapiro}.
They accomplish this by assigning the matrix $F_p = \left(\begin{smallmatrix} 0 & 1 \\ -1 & -\lambda_p \end{smallmatrix}\right)$ to the piece of a path going around an orbifold point. If an arc winds $k$ times around an orbifold point, they include $(-I_2)^{k-1}F_p^k$ where~$I_2$ is a $2 \times 2$ identity matrix.

Notice that when $k = 1$ and when we specialize $x_\rho = 1$ in the matrix in equation \eqref{eqn:chebymatrix}, we get a matrix similar to~$F_p$. Thus, we can interpret this matrix as recording a composition of steps of types~1 and~3. (Musiker--Williams also have matrices which differ by a sign along the diagonal from Fock--Goncharov, which does not affect the desired matrix statistics.) This is akin to the quasi-elementary steps which Musiker--Williams associate to matrices from Fock--Goncharov which correspond to a path turning left or right inside a triangle. Thus, we can interpret these new matrices from Chekhov--Shapiro as a way to record turning ``inside'' a pending arc (when pending arcs as visualized as loops around orbifold points, as shown in Section~\ref{sec:orbifolds}).
\end{Remark}

\subsection{Connecting arcs and snake graphs}

So far, given an arc or closed curve $\gamma$ on an orbifold $\mathcal{O}$ with corresponding generalized cluster algebra $\mathcal{A}$, we have provided two elements of~$\mathcal{A}$ from $\gamma$: $X_{\gamma,T}$ and $\chi_{\gamma,T}$. We now show these are always the same element of $\mathcal{A}$

\begin{Theorem}\label{Thm:ArcsAndGraphs}
Let $\mathcal{O}$ be an unpunctured orbifold with triangulation $T$, and let $\gamma$ be any arc or closed curve on $\mathcal{O}$. Let $e(\gamma, T) = d \geq 1$ and $G_{T,\gamma}$ be the snake graph $($or band graph$)$ constructed from $\gamma$. Then,
\begin{equation}\label{MandGraphEqn}
\chi_{\gamma,T} = X_{\gamma,T} = \frac{1}{{\rm cross}(T,\gamma)} \sum_P x(P) h(P),
\end{equation}
where the summation ranges over all perfect matchings~$P$ of~$G_{T,\gamma}$.
\end{Theorem}

\begin{proof}First, we briefly discuss the case where $d=0$ for use in later portions of the proof. If $d = 0$, then $\gamma \in T$. In this case, the standard $M$-path is a step of type~3 along $\gamma$ and its associated matrix is simply $\left[\begin{smallmatrix} 0 & \pm x_\gamma \\ \frac{\mp 1}{x_\gamma} & 0 \end{smallmatrix}\right]$. The snake graph $G_{T,\gamma}$ consists of two vertices connected by a single edge with label $\gamma$. Such a graph has exactly one perfect matching.

Now, we are prepared to consider $d > 0$. First, we consider the case when $\gamma$ is an arc. Recall Theorem \ref{thm:upperright}, \[\frac{1}{{\rm cross}(T,\gamma)} \sum_P x(P) h(P)
=
{\rm ur} \left( \begin{bmatrix} \dfrac{x_w}{x_{i_d}} & x_z y_{i_d} \vspace{1mm}\\ \dfrac{-1}{x_z} & 0 \end{bmatrix} M_d \begin{bmatrix} 0 & x_a \\ \dfrac{-1}{x_a} & \dfrac{x_b}{x_{i_1}} \end{bmatrix}\right).\] Moreover, recall that $M_d = m_{d-1} \cdots m_1$ where for $j \geq 1$\[
m_{2j} = \begin{bmatrix}\dfrac{r_{2j}}{i_{2j}} & y_{2j} a_{2j} \vspace{1mm}\\ \dfrac{b_{2j}}{i_{2j}i_{2j+1}} & y_{2j} \dfrac{\ell_{2j}}{i_{2j+1}} \end{bmatrix}, \qquad m_{2j-1} = \begin{bmatrix} \dfrac{\ell_{2j-1}}{i_{2j-1}} & y_{2j-1} b_{2j-1} \vspace{1mm}\\ \dfrac{a_{2j-1}}{i_{2j-1}i_{2j}} & y_{2j-1} \dfrac{r_{2j-1}}{i_{2j}} \end{bmatrix}.
\] These elements of the matrices $m_i$ are labels of the edges of the universal snake graph $UG_d$. When our graph comes from an arc on a triangulated orbifold, the labels of the edges of the graph correspond to arcs in the orbifold. Here, the first triangle that $\gamma$ passes through has sides~$a$,~$b$, and $\tau_{i_1}$ in clockwise order, and $\tau_{i_1}$ is the arc which~$\gamma$ crosses. Similarly, the last triangle that $\gamma$ crosses through has sides~$w$,~$z$, and~$\tau_{i_d}$ in clockwise order, and~$\tau_{i_d}$ is the last arc which $\gamma$ crosses.

We gave an algorithm for determining $\kappa_\gamma$, the standard $M$-path of $\gamma$, in terms of a sequence of compound steps. If $\gamma$ crosses $d$ arcs in $T$, we use $d-1$ compound steps, as well as initial and final sequence of elementary steps. Each compound step has an associated matrix.

In most cases, the product of matrices associated to the elementary steps before the first crossing in the standard $M$-path is $\left[\begin{smallmatrix} 0 & x_a \\ \frac{-1}{x_a} & \frac{x_b}{x_{i_1}} \end{smallmatrix}\right]$. When $\gamma$ first crosses a pending arc, $\rho$, and $s(\gamma)$ is also the unique marked point incident to $\rho$, this matrix is of the form $\left[\begin{smallmatrix} * & x_a \\ * & \frac{x_b}{x_{i_1}} \end{smallmatrix}\right]$. However, the terms in the first column will not affect the upper right entry of the product of matrices. This is similar for the product of matrices associated to the elementary steps at and after the last crossing in $\kappa_\gamma$.

Next, we compare the matrices $m_i$ in the description of $M_d$ with the matrices from each compound step of $\kappa_\gamma$. In our rules for $\kappa_\gamma$, if $\gamma$ crosses two consecutive standard arcs, $\tau_{i_j}$~and~$\tau_{i_{j+1}}$, which share a vertex to the right of $\gamma$, then we use a compound step of type $A$. Multiplying the elementary matrices these correspond to gives the matrix $\left[\begin{smallmatrix} 1 & 0 \\ \frac{x_{c_j}}{x_{i_j}x_{i_{j+1}}} & y_{i_j} \\ \end{smallmatrix}\right]$ where $c_j$ is the third edge in the triangle formed by $\tau_{i_j}$ and $\tau_{i_{j+1}}$. If~$\gamma$ crosses a standard arc~$\tau_{i_j}$ and then a pending arc $\tau_{i_{j+1}}$ and $\tau_{i_{j+1}}$ is based to the right of~$\gamma$, then we have the same form of matrix.

From our construction of snake graphs, if $\tau_{i_j}$ and $\tau_{i_{j+1}}$ share a vertex to the right of $\gamma$, and $i_j$ is odd, then we use a north-pointing parallelogram, so $b_j = 0$. In this case, $r_j = i_{j+1}$ and $\ell_j = i_j$.
If $i_j$ is even, then we use an east-pointing parallelogram, so $a_j =0$, $r_j = i_j$ and $\ell_j = i_{j+1}$.
These specializations apply even if $\tau_{i_{j+1}}$ is a pending arc. In either case, the matrix from the $j$-th compound step in $\kappa_\gamma$ matches the matrix we use for the $j$-th parallelogram in~$G_{T,\gamma}$.

If $\tau_{i_j}$ and $\tau_{i_{j+1}}$ share a vertex to the left of $\gamma$, or if $\tau_{i_{j+1}}$ is a pending arc based to the left of $\gamma$, then we use a compound step of type $B$ to transition between these. Multiplying the matrices in this compound step gives $\left[\begin{smallmatrix} \frac{x_{i_{j+1}}}{x_{i_j}} & y_{i_j}x_{c_j} \\ 0 & \frac{y_{i_j}x_{i_j}}{x_{i_{j+1}}} \end{smallmatrix}\right] $. When constructing $G_{T,\gamma}$, if $i_j$ is odd, we use an east-pointing parallelogram and if $i_j$ is even, we use a north-pointing parallelogram. When we use the relevant specializations, we see again that in either case the matrix $m_j$ matches the matrix in the expansion of~$M(\kappa)$ from this compound step.

Next, consider when $\tau_{i_j} = \tau_{i_{j+1}}$; this implies $\tau_{i_j}$ is a pending arc. Suppose that $\gamma$ winds $k \geq 0$ times around the orbifold point enclosed by $\tau_{i_j}$. Then, the product matrices from the series of elementary steps from the standard $M$-path are \[
\left( \begin{bmatrix} 1 & 0 \\ \dfrac{\lambda_p}{\rho} & 1 \end{bmatrix} \begin{bmatrix} 0 & \rho \\ \dfrac{-1}{\rho} & 0 \end{bmatrix} \right)^k \begin{bmatrix} 1 & 0 \\ \dfrac{\lambda_p}{\rho} & 1 \end{bmatrix} \begin{bmatrix} 1 & 0 \\ 0 & y_{i_j} \end{bmatrix}.
\]

By Lemma \ref{lem:ChebyshevMatrices} and the recurrence relation for Chebyshev polynomials, we have \[
\begin{bmatrix} -U_{k-2} (\lambda_p) & U_{k-1}(\lambda_p) \cdot x_\rho \vspace{1mm}\\ \dfrac{-U_{k-1}(\lambda_p)}{x_\rho} & U_k(\lambda_p) \end{bmatrix} \begin{bmatrix} 1 & 0 \vspace{1mm}\\ \dfrac{\lambda_p}{x_\rho} & 1 \end{bmatrix} \begin{bmatrix} 1 & 0 \vspace{1mm}\\ 0 & y_{i_j} \end{bmatrix} = \begin{bmatrix} U_k(\lambda_p) & U_{k-1}(\lambda_p) \cdot y_{i_j} x_\rho \vspace{1mm}\\ \dfrac{U_{k+1}(\lambda_p)}{x_\rho} & U_k(\lambda_p) \cdot y_{i_j} \end{bmatrix}.
\]

If $k = 0$, then $U_{-1}(\lambda_p) = 0$, and this resembles the case when $\tau_{i_j}$ and $\tau_{i_{j+1}}$ are standard arcs which share a vertex to the right of $\gamma$. If $k = p-2$, then $U_{(p-2)+1}(\lambda_p) = 0$, and this resembles the case where two consecutive standard arcs share a vertex to the left of $\gamma$. If $0 < k < p-2$, then all four entries of this matrix are nonzero.

In our construction of $G_{T,\gamma}$, when an arc winds $k > 0$ times around an orbifold point, we associate a hexagonal tile. In the language of the universal snake graph, we construct this with a parallelogram where both diagonals are included. Moreover, we have $i_j = i_{j+1} = x_\rho$, and $r_j = \ell_j = U_k(\lambda_p) x_\rho$. Then, either $a_j = U_{k-1}(\lambda_p) x_\rho$ and $b_j = U_{k+1}(\lambda_p) x_\rho$ or vice versa.

When $\tau_{i_j}$ is a pending arc and $\tau_{i_{j+1}}$ is not a pending arc, we have to consider both whether~$\tau_{i_j}$ is to the left or right of $\gamma$ and whether $\tau_{i_{j-2}}$ is distinct from $\tau_{i_{j+1}}$. We saw these four cases in the description of $\kappa_\gamma$. If $\tau_{i_{j+1}}$ and $\tau_{i_{j-2}}$ are distinct arcs, then we use the same compound step between $\tau_{i_{j-2}}$ and $\tau_{i_{j-1}}$ as we do between $\tau_{i_{j}}$ and $\tau_{i_{j+1}}$. For example, if $\tau_{i_{j+1}}$ and $\tau_{i_{j-2}}$ are distinct arcs and $\tau_{i_j}$, the pending arc, is based to the right of $\gamma$, then between $\tau_{i_j}$ and $\tau_{i_{j+1}}$ we use a~compound step of type $A$, just as we use between $\tau_{i_{j-2}}$ and $\tau_{i_{j-1}}$. When constructing $G_{T,\gamma}$ in this case, at indices $j-2$ and $j$ we either have both parallelograms facing north or both facing east. Thus, in the standard labeling of the universal snake graph, either both $a_{j-2} = a_j = 0$ or $b_{j-2} = b_j = 0$. Conversely, if $\tau_{i_{j-2}} = \tau_{i_{j+1}}$, in $\kappa_\gamma$ we use opposite compound steps between $\tau_{i_{j-2}}$ and $\tau_{i_{j-1}}$ and between $\tau_{i_{j}}$ and $\tau_{i_{j+1}}$. In the construction of $G_{T,\gamma}$, we use opposite parallelograms at indices $j-2$ and $j$. By specializing the entries of the matrices $m_i$ from the parallelograms at each case, we will see that the matrices from the graph and $M$-path agree again.

Putting all these cases together demonstrates that the matrices from the compound step decomposition of $\kappa_\gamma$ largely match the matrices used in Theorem \ref{prop:MatchingMatrices} to encode weighted perfect matchings of $G_{\gamma, T}$. The initial and final matrices will not necessarily completely match. However, we can conclude the following \[\vert {\rm ur}(M(\kappa_\gamma)) \vert = \left| \ur\left( \begin{bmatrix} \dfrac{x_w}{x_{i_d}} & x_z y_{i_d} \vspace{1mm}\\ \dfrac{-1}{x_z} & 0 \end{bmatrix} M_d \begin{bmatrix} 0 & x_a \vspace{1mm}\\ \dfrac{-1}{x_a} & \dfrac{x_b}{x_{i_1}} \end{bmatrix} \right) \right| =\frac{1}{{\rm cross}(T,\gamma)} \sum_P x(P) h(P).
\]

Now, let $\gamma$ be a closed curve, and let $q = v_{\tau_{i_1},m}^\pm$ be a point chosen on $\gamma$ for $\kappa_\gamma$ to start and end. Then, $\tau_{i_1}$ and $\tau_{i_d}$ form two sides of the triangle which $q$ lives in; call the third side of this triangle~$a$. We see that we can start with a compound step of type $A$ or $B$ since we start adjacent to the first arc which $\gamma$ crosses. However, when we cross $\tau_{i_d}$, we will need to include a compound step of type $B$ to return to $q$, since $\tau_{i_1}$ follows $\tau_{i_d}$ immediately clockwise. This compound step has matrix $\left[\begin{smallmatrix} \frac{x_{i_1}}{x_{i_d}}& 0 \\ \frac{x_{a}}{x_{idn}x_{i_{1}}} & \frac{y_{i_d}x_{i_1}}{x_{i_d}} \end{smallmatrix}\right]$. Thus, $M(\kappa_\gamma) = {\rm tr}\left(\left[\begin{smallmatrix} \frac{x_{i_1}}{x_{i_d}}& 0 \\ \frac{x_{a}}{x_{i_d}x_{i_{1}}} & \frac{y_{i_d}x_{i_1}}{x_{i_d}} \end{smallmatrix}\right] M_d \right)$. By Theorem~\ref{thm:upperright}, this is equivalent to the weighted sum of perfect matchings of the band graph~$G_{T,\gamma}$.
\end{proof}

\begin{Corollary}
Let $\gamma$ be an arc or closed curve on $\mathcal{O}$ with no contractible kinks which winds at most $p-1$ times around any particular orbifold point of order $p$. Then the coefficients of the Laurent expansion for $x_{\gamma}$ obtained from the cluster expansion formula in Theorem~{\rm \ref{Thm:ArcsAndGraphs}} are non-negative.
\end{Corollary}

\begin{proof}
By Lemmas \ref{lem:ChebyPeriodic} and \ref{lem:polygonlength}, $U_k(\lambda_p) \geq 0$ if $0 < k \leq p-1$. Therefore, for such $\gamma$, $x_\gamma$ is given by a sum of weighted perfect matchings of a graph with positive edge weights and the coefficients are manifestly non-negative.
\end{proof}

\section[Skein relations with $y$-variables]{Skein relations with $\boldsymbol{y}$-variables}\label{sec:Skein}

The following definition will be useful for a condensed discussion of the skein relations in \cite{MW}.

\begin{Definition}
A \emph{multicurve}, $C$, on $\mathcal{O}$ is a finite multi-set of arcs and closed curves on $\mathcal{O}$. If these arcs and curves are $\gamma_1,\ldots, \gamma_n$, then we define the monomial $\chi_{C,T}$ to be the product $\chi_{\gamma_1,T}\cdots \chi_{\gamma_n,T}$.
\end{Definition}

Musiker and Williams \cite{MW} prove in Propositions 6.4, 6.5, and 6.6 that, in the surface case, the quantities $\chi_{\gamma,T}$ respect the skein relation. Let $C$ be the multicurve which consists of either~$\gamma_1$ and~$\gamma_2$, two generalized arcs or closed curves which intersect, or $\gamma$, an arc or closed curve with points of self-intersection. At one point of intersection between $\gamma_1$ and $\gamma_2$, or one point of self-intersection on $\gamma$, we can use \emph{smoothing}. This will create two new multicurves, call them $C_1$ and $C_2$; amongst the arcs in~$C_1$ and~$C_2$, there is at least one less intersection than amongst the arcs in~$C$. See~\cite{MW} for more details about the process of smoothing and the proofs of these propositions.

\begin{Theorem}[{\cite[Propositions 6.4, 6.5, and 6.6]{MW}}]\label{MWSkein}
Let $\mathcal{A}$ be the cluster algebra associated to surface $(S,M)$ with initial triangulation $T$.
Given $C$, a multicurve consisting of two intersecting arcs/curves or one arc with self-intersection, and $C_1$, $C_2$ the multicurves resulting in smoothing one point of intersection in $C$, in $\mathcal{A}$ we have \[
\chi_{C,T} = \pm Y_1 \chi_{C_1,T} \pm Y_2 \chi_{C_2,T}.
\]
Here, $Y_1$ and $Y_2$ are monomials in the $y$-variables and can be computed by analyzing the intersections of the arcs/curves in $C$, $C_1$, and $ C_2$ with the elementary laminations from the initial triangulation~$T$.
\end{Theorem}

We can use the same skein relations from~\cite{MW} when we have pending arcs or, more generally, arcs which wind around orbifold points, by treating them as any arc on a surface.

\begin{Proposition}\label{Prop:OrbSkein}
Theorem {\rm \ref{MWSkein}} holds on an unpunctured orbifold~$\mathcal{O}$.
\end{Proposition}

\begin{proof}
The proofs of these propositions from Musiker Williams work just as well on an orbifold. The smoothing of $C$ in either case is the same, and we have shown a way to encode the expansion of the cluster algebra element associated to an arc or closed curve by a product of a sequence of matrices. Thus, we can use the same matrix equalities \cite[Lemma~6.11]{MW} which are the fundamental tool in their proofs. Since we do not consider punctures in our surface, we do not need to use their machinery of the loosened $M$-path.
\end{proof}

These skein relations show that our choice of cluster algebra element $\chi_{\gamma,T}$ associated to a~generalized arc or closed curve $\gamma$ is the right choice. In particular, we can decompose $\gamma$ into a~sum of products of ordinary arcs. By Theorem~\ref{Thm:GenSnake+}, we already know the correct cluster algebra elements to associate to the ordinary arcs. Proposition~\ref{Prop:OrbSkein} shows that the associated cluster algebra elements satisfy the same decomposition. Since $\chi_{\gamma,T} = X_{\gamma,T}$, we can conclude our expansion formula provides the right choice of cluster algebra element for arbitrary arcs and closed curves.

Standard skein relations resemble the binomial exchange relation in an ordinary cluster algebra. When two pending arcs intersect, we can use the standard skein relation twice to recover a~three-term relation which models the generalized exchanges in the generalized cluster algebras we consider.

\begin{Proposition}\label{prop:OrbSkeinThreeTerm}
Let $\gamma_1$, $\gamma_2$ be two distinct pending arcs to the same orbifold point in an unpunctured orbifold $\mathcal{O}$ with triangulation $T$. Choose an orientation for $\gamma_1$, and let $q_1,\ldots, q_{2\ell}$ be the intersections of $\gamma_1$ and $\gamma_2$, with order determined by the orientation of $\gamma_1$. Orient $\gamma_2$ so that it visits $q_\ell$ before $q_{\ell + 1}$. Let $\beta_1 = (\gamma_1, \gamma_2^- \vert s(\gamma_1), q_\ell, s(\gamma_2))$ and let $\beta_2 = (\gamma_1^-, \gamma_2 \vert t(\gamma_1), q_{\ell + 1}, t(\gamma_2))$. Then, \[
\chi_{\gamma_1,T} \chi_{\gamma_2,T} = Y_0 \chi_{\beta_1,T}^2 + Y_1 \lambda_p \chi_{\beta_1,T} \chi_{\beta_2,T} + Y_2 \chi_{\beta_2,T}^2.
\]
\end{Proposition}

\begin{proof}
Note that $\beta_1$ and $\beta_2$ form a bigon around the orbifold point incident to $\gamma_1$ and $\gamma_2$ such that these are the two pending arcs inside the bigon.

First, we use Proposition~6.4 of \cite{MW} to smooth $\gamma_1$ and $\gamma_2$ at $q_\ell$. In the vocabulary of Theorem~\ref{MWSkein}, if $C = \{\gamma_1, \gamma_2\}$, then $C_1 = \{\beta_1,\alpha\}$ and $C_2 = \{\beta_2,\beta_2\}$ where $\alpha = (\gamma_1^-, \gamma_2 \vert t(\gamma_1), q_\ell, t(\gamma_2))$. Note that $\alpha$ has one self-intersection. We can then use Proposition 6.6 of \cite{MW} to smooth $\alpha$. If $\bar{C} = \{\alpha\}$, then after smoothing we get the two multicurves $\bar{C}_1 = \{\xi, \beta_2\}$ and $\bar{C}_2 = \{\beta_1 \}$ where~$\xi$ is an essential loop around the orbifold point incident to~$\gamma_1$ and~$\gamma_2$. Then, we can decompose~$C_1$ to $C_{1,1} = \{\xi, \beta_1,\beta_2\}$ and $C_{1,2} = \{\beta_1,\beta_1\}$.

\begin{center}
\begin{tikzpicture}[scale=1.8]
\node[circle, fill = black, scale = 0.3] (A) at (-3,0) {$w$};
\node[circle, fill = black, scale = 0.3] (B) at (-1,0) {$v$};
\node[] (E) at (-2,0) {$\times$};
\node[](F) at (-2.2, 0){};
\node[](G) at (-1.8,0){};
\node[](G1) at (-2, -0.2){};
\node[](G2) at (-2,0.2){};
\draw (-3,0) to [out=60,in=120] node[auto]{}(-1,0);
\draw (-3,0) to [out=-60,in=-120] node[below]{}(-1,0);
\draw[red,thick](-1,0) to [out = 150, in = 90] node[above, right]{}(-2.2,0);
\draw[red,thick](-1,0) to [out = 210, in = 270] node[below]{$\gamma_2$}(-2.2,0);
\draw[red, thick](-3,0) to [out = 30, in = 90] node[above]{$\gamma_1$} (-1.8,0);
\draw[red, thick](-3,0) to [out = -30, in = 270] (-1.8,0);
\node[circle, fill = black, scale = 0.3] at (-3,0){};
\node[circle, fill = black, scale = 0.3] at (-1,0){};
\node[] at (-0.5,0){$=$};
\node[circle, fill = black, scale = 0.3] (A) at (0,0) {$w$};
\node[circle, fill = black, scale = 0.3] (B) at (2,0) {$v$};
\node[] (E) at (1,0) {$\times$};
\draw (A) to [out=60,in=120] node[auto]{}(B);
\draw[red, thick] (A) to [out=-60,in=-120] node[below]{$\beta_1$}(B);
\draw[red, thick](0,0) to [out = 30, in = 90] node[above]{} (1.2,0);
\draw[red, thick](1.2,0) to [out = 270, in = 0] (1,-0.2);
\draw[red, thick] (1,-0.2) to [out = 180, in = 180] (1,0.2);
\draw[red, thick] (1,0.2) to [out = 0, in = 150] (2,0);
\node[circle, fill = black, scale = 0.3] (A) at (0,0) {$w$};
\node[circle, fill = black, scale = 0.3] (B) at (2,0) {$v$};
\node[] at (2.5,0){$+$};
\node[circle, fill = black, scale = 0.3] (A) at (3,0) {$w$};
\node[circle, fill = black, scale = 0.3] (B) at (5,0) {$v$};
\node[] (E) at (4,0) {$\times$};
\draw[red, thick] (A) to [out=60,in=120] node[auto]{}(B);
\draw[red, thick] (3,0.1) to [out=60,in=120] node[auto]{$\beta_2^2$}(5,0.1);
\draw[] (A) to [out=-60,in=-120] node[below]{}(B);
\end{tikzpicture}
\begin{tikzpicture}[scale=1.8]
\node[] at (-3.2,0){$=$};
\node[circle, fill = black, scale = 0.3] (A) at (-3,0) {$w$};
\node[circle, fill = black, scale = 0.3] (B) at (-1,0) {$v$};
\node[] (E) at (-2,0) {$\times$};
\node[](F) at (-2.2, 0){};
\node[](G) at (-1.8,0){};
\node[](G1) at (-2, -0.2){};
\node[](G2) at (-2,0.2){};
\draw (-3,0) to [out=60,in=120] node[auto]{}(-1,0);
\draw[red,thick] (-3,0) to [out=-60,in=-120] node[below]{}(-1,0);
\draw[red,thick] (-3,-0.1) to [out=-60,in=-120] node[below]{$\beta_1^2$}(-1,-0.1);
\node[circle, fill = black, scale = 0.3] at (-3,0){};
\node[circle, fill = black, scale = 0.3] at (-1,0){};
\node[] at (-0.5,0){$+$};
\node[circle, fill = black, scale = 0.3] (A) at (0,0) {$w$};
\node[circle, fill = black, scale = 0.3] (B) at (2,0) {$v$};
\node[] (E) at (1,0) {$\times$};
\draw[red,thick] (A) to [out=60,in=120] node[auto]{$\beta_2$}(B);
\draw[red, thick] (A) to [out=-60,in=-120] node[below]{$\beta_1$}(B);
\draw[red, thick](1,0.2) to [out = 0, in = 0] (1,-0.2);
\draw[red, thick] (1,-0.2) to [out = 180, in = 180] (1,0.2);
\node[circle, fill = black, scale = 0.3] (A) at (0,0) {$w$};
\node[circle, fill = black, scale = 0.3] (B) at (2,0) {$v$};
\node[] at (2.5,0){$+$};
\node[circle, fill = black, scale = 0.3] (A) at (3,0) {$w$};
\node[circle, fill = black, scale = 0.3] (B) at (5,0) {$v$};
\node[] (E) at (4,0) {$\times$};
\draw[red, thick] (A) to [out=60,in=120] node[auto]{}(B);
\draw[red, thick] (3,0.1) to [out=60,in=120] node[auto]{$\beta_2$}(5,0.1);
\draw[] (A) to [out=-60,in=-120] node[below]{}(B);
\end{tikzpicture}
\end{center}

There are no crossings amongst the arcs in $C_{1,1}$, $C_{1,2}$, and $C_2$, and we have that $x_C = x_{C_{1,1}} + x_{C_{1,2}} + x_{C_2}$. By Proposition \ref{prop:MSWChebyshev}, $x_\xi = \lambda_p$ where $p$ is the order of this orbifold point. This brings us to the desired equality \[
\chi_{\gamma_1,T} \chi_{\gamma_2,T} = \widetilde{Y} \chi_{\beta_1,T} \chi_{\alpha,T} + Y_2 \chi_{\beta_2,T}^2 = Y_0 \chi_{\beta_1,T}^2 + Y_1 \chi_{\beta_1,T} \chi_{\beta_2,T} + Y_2 \chi_{\beta_2,T}^2.\tag*{\qed}
\]\renewcommand{\qed}{}
\end{proof}

\section{Relationship to punctures}\label{sec:RelateToPuncture}

Throughout the paper, we restricted our discussion to unpunctured orbifolds. Recall that a~puncture is a marked point which appears in the interior of a surface or orbifold. The original snake graph construction in~\cite{MSW} does handle surfaces with punctures. In this section, we give some interesting examples which illustrate how some results from~\cite{MSW} and~\cite{MW} concerning punctures can be recovered by treating the puncture as an orbifold point with infinite order.

As motivation, recall that an arc which winds $k$ times around an orbifold point of order $p$ is isotopic to an arc winding $k \pm np$ times for any integer~$n$ -- even if this means that the winding arc switches directions. This type of isotopy does not exist for arcs winding around punctures; thus, in some sense we could consider the puncture to have infinite order. Moreover, note that $\lim\limits_{p \to \infty} \lambda_p = \lim\limits_{p \to \infty} 2 \cos(\pi/p) = 2$. Thus, a loop which is contractible to an orbifold point of infinite order has the same weight as a loop contractible to a puncture.

We will specifically consider the case of a puncture inside a self-folded triangle, as this most closely resembles a pending arc, and compare the $x$-variables and $y$-variables in these situations. We note that by specializing $\lambda_\infty = 1 + y_r $, we can nearly recover the $F$-polynomials from these cluster algebra elements.

Previously, we discussed normalized Chebyshev polynomials of the second kind. Now, we introduce another formal variable to these Chebyshev polynomials.

\begin{Definition}\label{def:ChebyCoeff}
Let $\big\{U_k^Y(x)\big\}_k$ be a family of polynomials indexed by $k = -1,0,1,\ldots$ such that $U_{-1}^Y(x) = 0$, $U_0^Y(x) = 1$, and for $k \geq 1$, \[
U_k^Y(x) = x\cdot U_{k-1}^Y(x) - Y\cdot U_{k-1}^Y(x).
\]
\end{Definition}

For example, $U_1^Y(x) = x$, $U_2^Y(x) = x^2 - Y$, and $U_3^Y(x) = x^3 - 2Yx$.

We then record some results about our normalized Chebyshev polynomials, with and without coefficients, for later use.
\begin{Lemma}\samepage
Let $U_k$ and $U_k^Y$ be as in Definitions {\rm \ref{def:Chebyshev}} and~{\rm \ref{def:ChebyCoeff}}. Then for $k \geq 1$,
\begin{enumerate}\itemsep=0pt
\item[$1)$] $U_k(2) = k+1$,
\item[$2)$] $U_k^Y(1 + Y) = 1 + Y + \cdots + Y^k$.
\end{enumerate}
\end{Lemma}

\begin{proof}The first statement follows from the second by setting $Y=1$, so we need only prove the second statement. We do so by induction. The statement clearly holds for $U_1^Y(x)$, and by definition $U_2^Y(1+Y) = (1+Y)^2 - Y = 1 + Y + Y^2$. Then, for $k \geq 3$,
\begin{align*}
 U_k^Y(1+Y) &= (1+Y)\cdot U_{k-1}^Y(1+Y) - Y\cdot U_{k-2}^Y(1+Y) \\
 &= (1+Y)\big(1+Y + \cdots + Y^{k-1}\big) - Y\big(1+Y + \cdots + Y^{k-2}\big) = 1 + Y + \cdots + Y^k
\end{align*}
as desired.
\end{proof}

\begin{figure}[h!]\centering
\begin{tabular}{|c|c|}
\hline
 \includegraphics[]{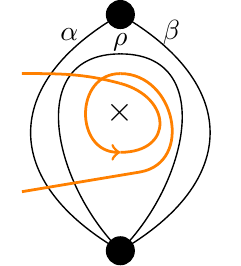} & \includegraphics[]{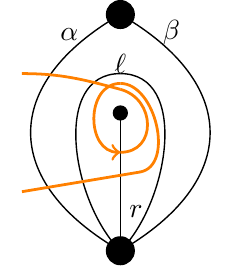} \\\hline
 \hline
 \includegraphics[valign=m, scale = 1.3]{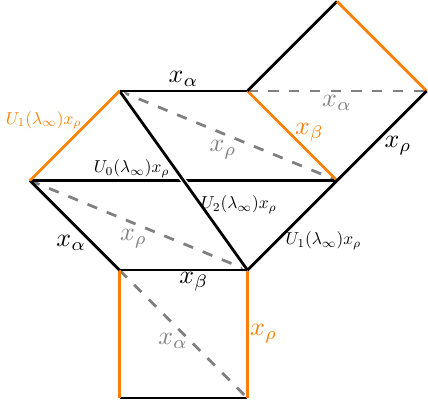} & $\begin{array}{c} \includegraphics[]{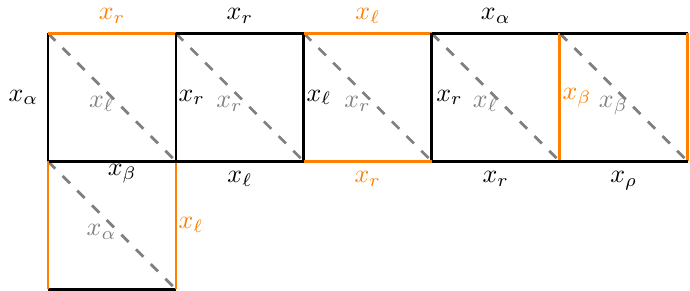} \\ \includegraphics[]{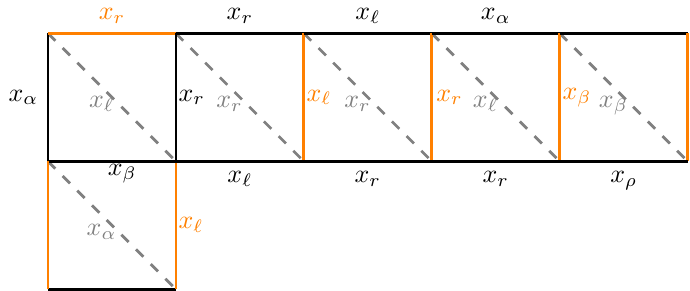} \end{array}$ \\ \hline
\end{tabular}
\caption{A comparison of the local snake graphs for arcs crossing pending arcs and self-folded triangles. The highlighted perfect matching of the generalized snake graph corresponds to the two highlighted perfect matchings in ordinary snake graph for the punctured surface case.}\label{Fig:OrbAndPuncture1}
\end{figure}

\begin{figure}[h!] \centering
\begin{tabular}{|c|c|}
\hline\includegraphics[scale = 1.3]{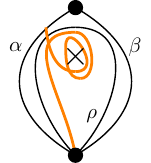} & \includegraphics[scale = 1.3]{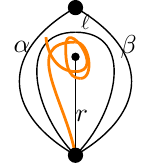}\\ \hline
\hline\includegraphics[]{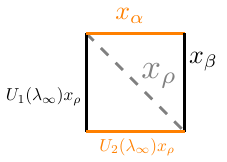} & \includegraphics[]{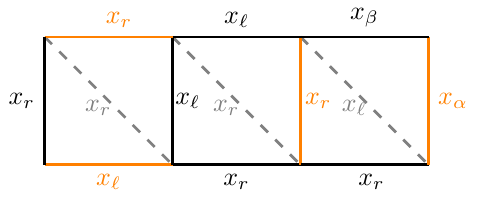}\\
& \includegraphics[]{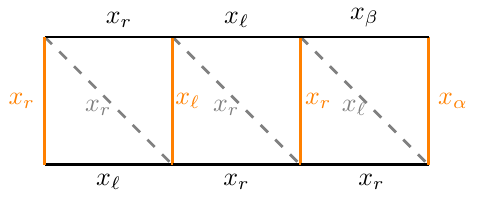}\\
& \includegraphics[]{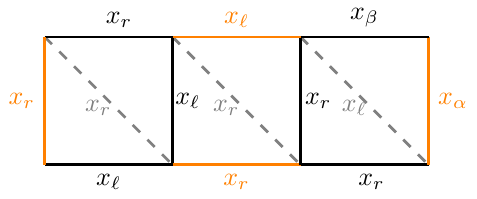}\\ \hline
\end{tabular}
\caption{Another comparison of generalized and ordinary snake graphs, for an example of arcs crossing the pending arc or loop a single time.} \label{Fig:OrbAndPuncture2}
\end{figure}

\looseness=-1 In Fig.~\ref{Fig:OrbAndPuncture1}, we compare the generalized snake graph from an arc that crosses a pending arc twice and has a single self-intersection to the ordinary snake graph for an analogous arc where the orbifold point has been replaced by a puncture and the pending arc by a self-folded triangle. If we set $\lambda_\infty = 2$, so that $U_0(\lambda_\infty) = 1$, $U_1(\lambda_\infty) = 2$ and $U_2(\lambda_\infty) = 3$, several edge labels on the generalized snake graph become positive integer multiples of cluster variables. A perfect matching which uses one of these edges corresponds to multiple perfect matchings in the snake graph from the surface case. The highlighted perfect matchings in Fig.~\ref{Fig:OrbAndPuncture1} show an example of this.

In the generalized snake graph, the perfect matching $P$ uses an edge labeled $U_1(\lambda_\infty)x_\rho = 2x_\rho$. Considering only the arcs drawn, $x(P) = 2 x_\beta x_\rho^2$. Recall that in the denominator of the cluster expansion formula, we have the crossing monomial $x_\rho^2$. However, a factor of $x_\rho$ also appears in each of the other terms in the numerator. Canceling this factor gives the reduced weight $2 x_\beta x_\rho$. In the ordinary snake graph, we show two matchings, which each have weight $x_\beta x_r^2 x_\ell^2$. Similar to the orbifold case, here the crossing monomial is $x_r^2 x_\ell^2$.
By observation, we see that every perfect matching will use at least two edges labeled $x_r$ and one edge labeled $x_\ell$; thus, we can cancel a~factor of $x_r^2x_\ell$ to obtain the reduced weight $x_\beta x_\ell$ for each perfect matching. Setting $\lambda_p = 1 + y_r$, so that $U_1(\lambda_p) = 1 + y_r$, also makes sense in this example, since the two highlighted matchings differ only by a twist at a tile with diagonal label $r$.

Fig.~\ref{Fig:OrbAndPuncture2} shows another example, where this time the arc only crosses the pending arc or loop a single time. The generalized snake graph perfect matching which uses the edge labeled $U_2(\lambda_p) x_\rho = \big(1 + y_r + y_r^2\big) x_\rho$ corresponds to three perfect matchings of the analogous ordinary snake graph. The terms associated to these matchings, with $y$-variables, are $\big(1 + y_r + y_r^2\big) x_r^2 x_\ell x_\alpha$.

A natural direction for future work would be to study generalized snake graphs in the punctured setting. If punctures can be completely understood as orbifold points of infinite order, then a pending arc incident to a orbifold point of infinite order should be understood as the product of the loop and radius arcs in a once-punctured monogon. Taking this viewpoint, our generalized snake graphs are still able to handle the once-punctured monogon case. It would be non-trivial, however, to extend our generalized snake graphs to arbitrary tagged triangulated orbifolds, where there may be multiple arcs incident to a single given puncture.

\subsection*{Acknowledgements} We would like to thank Gregg Musiker for many helpful discussions and his patient generosity in reading and providing feedback on many iterations of this paper. We would also like to thank the anonymous referees for their helpful comments.

\LastPageEnding

\end{document}